\documentclass[a4paper,12pt]{amsart}

\usepackage{amssymb,amsthm,mathrsfs}
\usepackage{upgreek}
\usepackage[all]{xy}

\usepackage{setspace}
\usepackage{hyperref}
\usepackage{tikz}

\renewcommand{\rho}{\uprho}
\newtheorem{theorem}[equation]{Theorem}
\newtheorem{lemma}[equation]{Lemma}
\newtheorem{corollary}[equation]{Corollary}
\newtheorem{proposition}[equation]{Proposition}

\theoremstyle{definition}
\newtheorem{example}[equation]{Example}

\newtheorem{warning}[equation]{Warning}
\newtheorem*{definition*}{Definition}

\theoremstyle{remark}
\newtheorem{remark}[equation]{Remark}

\makeatletter%
\@addtoreset{equation}{subsection}%
\makeatother%

\def \O {\mathscr{O}}

\def \CC {\mathscr{C}}

\def \P {\mathbb{P}}
\def \GG {\mathbb{G}}
\newcommand{\GGa}{\mathbb G_{\mathrm{a}}}
\newcommand{\GGm}{\mathbb G_{\mathrm{m}}}

\def \Q {\mathbb{Q}}

\def \Z {\mathbb{Z}}
\def \F {\mathrm{F}}

\def \SS {\mathfrak{S}}
\def \A {\mathfrak{A}}
\def \Gr {\mathrm{Gr}}
\def \Fl {\mathrm{Fl}}

\def \red {\mathrm{red}}

\def \mumu {\boldsymbol{\mu}}

\newcommand{\dd}{\operatorname{d}}

\def \Aut {\mathrm{Aut}}
\def \Cl {\mathrm{Cl}}
\def \Pic {\mathrm{Pic}}
\def \Sym {\mathrm{Sym}}

\newcommand{\diag}{\operatorname{diag}}
\newcommand{\reduced}{\operatorname{red}}
\newcommand{\g}{\operatorname{g}}

\def \Sing {\mathrm{Sing}\,}
\def \GL {\mathrm{GL}}
\def \SL {\mathrm{SL}}
\def \PGL {\mathrm{PGL}}

\def \PSO {\mathrm{PSO}}

\def \Oct {\mathrm{Oct}}
\def \Icos {\mathrm{Icos}}

\def\udot {{\:\raisebox{3pt}{\text{\circle*{1.5}}}}}

\def \ge {\geqslant}
\def \le {\leqslant}

\newcommand{\xref}[1]{\textup{\ref{#1}}}
\newcommand{\CA}{{\mathscr{A}}}
\newcommand{\CB}{{\mathscr{B}}}
\newcommand{\D}{\mathop{\mathsf{D}}\nolimits}
\newcommand{\CE}{{\mathscr{E}}}
\newcommand{\CF}{{\mathscr{F}}}

\newcommand{\CL}{{\mathscr{L}}}
\newcommand{\CM}{{\mathscr{M}}}
\newcommand{\CN}{{\mathscr{N}}}
\newcommand{\CO}{{\mathscr{O}}}
\newcommand{\CQ}{{\mathscr{Q}}}
\newcommand{\CU}{{\mathscr{U}}}
\newcommand{\CV}{{\mathscr{V}}}
\newcommand{\SQ}{{\mathsf{Q}}}
\newcommand{\PP}{{\mathbb{P}}}

\newcommand{\Hom}{\mathop{\mathsf{Hom}}\nolimits}
\newcommand{\Ext}{\mathop{\mathsf{Ext}}\nolimits}
\newcommand{\Tor}{\mathop{\mathsf{Tor}}\nolimits}
\newcommand{\CHom}{\mathop{\mathscr{H}\!\mathit{om}}\nolimits}
\newcommand{\CExt}{\mathop{\mathscr{E}\!\mathit{xt}}\nolimits}
\newcommand{\Ker}{\mathop{\mathsf{Ker}}}
\newcommand{\Coker}{\mathop{\mathsf{Coker}}}
\newcommand{\Hilb}{\mathop{\mathsf{Hilb}}\nolimits}
\newcommand{\Pf}{\mathop{\mathsf{Pf}}}
\newcommand{\LGr}{\mathop{\mathrm{LGr}}}

\DeclareMathOperator{\Orb}{Orb}
\DeclareMathOperator{\oOrb}{\overline{Orb}}

\title[Lines, conics, and automorphisms of Fano threefolds]
{Hilbert schemes of lines and conics\\[1ex] and automorphism groups of Fano threefolds}

\author{Alexander Kuznetsov }
\author{Yuri Prokhorov}
\author{Constantin Shramov}
\address{
Alexander Kuznetsov:\newline
Steklov Mathematical Institute of Russian Academy of Sciences,
\newline
8 Gubkina street, Moscow
\newline
The Poncelet Laboratory, Independent University of Moscow
\newline
Laboratory of Algebraic Geometry, National Research University Higher School of Economics, Moscow
}
\email{akuznet@mi.ras.ru}
\address{
\newline\phantom{ppp}
Yuri Prokhorov:\newline
Steklov Mathematical Institute of Russian Academy of Sciences,
\newline
8 Gubkina street, Moscow
\newline
Laboratory of Algebraic Geometry, National Research University Higher School of Economics, Moscow
\newline
Department of Algebra, Moscow State
University, Moscow
}
\email{prokhoro@mi.ras.ru}

\address{
Constantin Shramov:\newline
Steklov Mathematical Institute of Russian Academy of Sciences,
\newline
8 Gubkina street, Moscow
\newline
Laboratory of Algebraic Geometry, National Research University Higher School of Economics, Moscow
}
\email{costya.shramov@gmail.com}

\thanks{The authors were partially supported by
the Russian Academic Excellence Project ``5-100'',
by RFBR grant 15-01-02164, and by
the Program of the Presidium of the Russian Academy of Sciences No.~01 ``Fundamental Mathematics and
its Applications'' under grant \mbox{PRAS-18-01}.
A.K was also supported by RFBR 14-01-00416, 15-51-50045 and by the Simons foundation.
Yu.P. was also supported by RFBR 15-01-02158 and 15-51-50045.
C.S. was also supported by RFBR 14-01-00160 and 15-01-02158 and the Dynasty foundation.}

\begin{document}

\maketitle

\begin{abstract}
We discuss various results on Hilbert schemes of lines and conics and automorphism groups of smooth Fano threefolds
of Picard rank~1.
Besides a general review of facts well known to experts, the paper contains some new results,
for instance, we give a description of the Hilbert scheme of conics on any smooth Fano threefold of index 1 and genus 10.
We also show that the action of the automorphism group of a Fano threefold~$X$ of index~2 (respectively,~1) on an irreducible component of its Hilbert scheme of lines (respectively, conics) is faithful
if the anticanonical class of~$X$ is very ample except for some explicit cases.

We use these faithfulness results to prove finiteness of the automorphism groups of most Fano threefolds and classify explicitly all Fano threefolds with infinite automorphism group.
We also discuss a derived category point of view on the Hilbert schemes of lines and conics, and use it to identify some of them.
\end{abstract}

\tableofcontents

\section{Introduction}

\subsection{Setup and main results}
We work over an algebraically closed field $\Bbbk$ of characteristic~$0$.

The purpose of this paper is to survey the results on Hilbert schemes of lines and conics and
automorphism groups of Fano threefolds of Picard rank~$1$. These are usually known to experts,
but sometimes are scattered in the literature or even
in the mathematical folklore.

Let $X$ be a Fano threefold with at worst canonical Gorenstein singularities. In this case, the number
\begin{equation*}
\g(X)=-\frac{1}{2}K_X^3+1
\end{equation*}
is called the \textit{genus} of $X$.
By Riemann--Roch theorem and Kawamata--Viehweg vanishing one has
\begin{equation*}
\dim |-K_X|=\g(X)+1
\end{equation*}
(see e.\,g. \cite[2.1.14]{Iskovskikh-Prokhorov-1999}).
In particular, $\g(X)$ is an integer, and $\g(X)\ge 2$.
Recall that~\mbox{$\Pic(X)$} is a finitely generated torsion free abelian group, so that
\begin{equation*}
\Pic(X)\cong \Z^{\rho(X)}
\end{equation*}
(this holds even for Fano varieties with log terminal singularities, see e.g.~\mbox{\cite[Proposition 2.1.2]{Iskovskikh-Prokhorov-1999}}).
The integer $\rho(X)$ is called the \emph{Picard rank} of~$X$.
The maximal number $\iota=\iota(X)$
such that $-K_X$ is divisible by $\iota$ in~\mbox{$\Pic(X)$}
is called the \textit{Fano index}, or just~\emph{index}, of~$X$.
Let $H$ be a divisor class such that
\begin{equation*}
-K_X\sim\iota(X) H.
\end{equation*}
The class $H$ is unique since $\Pic(X)$ is torsion free.
Define the \textit{degree} of $X$ as
\begin{equation*}
\dd(X)=H^3.
\end{equation*}

In this paper we concentrate on smooth Fano threefolds of Picard rank 1.
Their classification can be found in~\mbox{\cite[\S12.2]{Iskovskikh-Prokhorov-1999}} (see also \cite{mukai1989biregular}).
We recall it in Tables~\xref{table:Fanos-i-ge-2} and~\xref{table:Fanos-i-1}
which contain the lists of Fanos with index at least two and index one, respectively.
For our purposes it will be important to know for each type of Fano threefolds the minimal integer~$m_0$ such that~\mbox{$m_0H$} is very ample. We list these $m_0$ in the last columns of the tables.

\begin{table}[h]
\caption{Smooth Fano threefolds with $\rho=1$ and $\iota\ge 2$}\label{table:Fanos-i-ge-2}
\begin{tabular}{|c|c|c|p{11.5cm}|c|}
\hline $\iota$ & $\dd$ & $h^{1,2}$ & Brief description & $m_0$ \\
\hline
\hline $4$ & $1$ & 0 & $\mathbb{P}^{3}$ & 1 \\
\hline
\hline $3$ & $2$ & 0 & quadric hypersurface in $\mathbb{P}^{4}$ & 1 \\
\hline
\hline $2$ & $1$ & 21 & hypersurface in $\mathbb{P}(1,1,1,2,3)$ of degree $6$ & 3 \\
\hline $2$ & $2$ & 10 & double cover of $\P^3$ branched in a quartic surface & 2 \\
\hline $2$ & $3$ & 5 & cubic hypersurface in $\mathbb{P}^{4}$ & 1 \\
\hline $2$ & $4$ & 2 & complete intersection of two quadrics in $\mathbb{P}^{5}$ & 1 \\
\hline $2$ & $5$ & 0 & section of $\mathrm{Gr}(2,5)\subset\mathbb{P}^9$ by a linear subspace of codimension $3$ & 1 \\
\hline
\end{tabular}
\end{table}

\begin{table}[h]
\caption{Smooth Fano threefolds with $\rho=1$ and $\iota=1$}\label{table:Fanos-i-1}
\begin{tabular}{|c|c|c|p{11.5cm}|c|}
\hline $\g$ & $\dd$ & $h^{1,2}$ & Brief description & $m_0$ \\
\hline
\hline $2$ & $2$ & $52$ & double cover of $\P^3$ branched in a sextic surface & 3 \\
\hline $3$ & $4$ & $30$ & (a) quartic hypersurface in $\mathbb{P}^4$, or & 1 \\ &&&
(b) double cover of a smooth quadric in
$\mathbb{P}^{4}$ branched in an intersection with a quartic & 2 \\
\hline $4$ & $6$ & $20$ & complete intersection of a quadric and a cubic in $\mathbb{P}^{5}$ & 1 \\
\hline $5$ & $8$ & $14$ & complete intersection of three quadrics $\mathbb{P}^{6}$ & 1 \\
\hline $6$ & $10$ & $10$ & (a) section of $\mathrm{Gr}(2,5)\subset\mathbb{P}^9$ by a linear subspace of codimension~$2$ and a quadric, or
\newline
(b) double cover of the Fano threefold $Y$ with $\iota(Y)=2$ and $\dd(Y)=5$ branched in an anticanonical divisor & 1 \\
\hline $7$ & $12$ & $7$ & section of a connected component of the orthogonal Lagrangian Grassmannian $\mathrm{OGr}_{+}(5,10)\subset\P^{15}$ by a linear subspace of codimension~$7$ & 1 \\
\hline $8$ & $14$ & $5$ & section of $\mathrm{Gr}(2,6)\subset\mathbb{P}^{14}$ by a linear subspace of codimension $5$ & 1 \\
\hline $9$ & $16$ & $3$ & section of the symplectic Lagrangian Grassmannian \mbox{$\LGr(3,6)\subset\P^{13}$} by a linear subspace of codimension~$3$ & 1 \\
\hline $10$ & $18$ & $2$ & section of the homogeneous space $\mathrm{G}_2/P\subset\mathbb{P}^{13}$ by a linear subspace of codimension~$2$ & 1 \\
\hline $12$ & $22$ & $0$ & zero locus of three sections of the rank $3$ vector bundle $\Lambda^2\CU^\vee$, where $\CU$ is the universal subbundle on $\mathrm{Gr}(3,7)$ & 1 \\
\hline
\end{tabular}
\end{table}

Note that although in some cases (for $\rho = \iota = 1$ and $\g = 3$ or $\g = 6$) there are two types of Fano threefolds, they belong to the same deformation family.

The first main result of this paper is an explicit description of the Hilbert schemes
of lines $\Sigma(Y)$ on Fano threefolds $Y$ of Picard rank 1, index 2 and degree $\dd(Y) \ge 3$
and the Hilbert schemes of conics $S(X)$ on Fano threefolds $X$ of Picard rank~1, index~1
and genus~\mbox{$\g(X) \ge 7$}
(by lines and conics we mean lines and conics in the embedding given by the linear system~$|H|$).
We collect the results we have in the following theorem.

\begin{theorem}
\label{theo:main1}
Let $Y$ be a smooth Fano threefold with $\rho(Y)=1$, $\iota(Y) = 2$, and~\mbox{$\dd(Y)\ge 3$}.
Then the Hilbert scheme of lines $\Sigma(Y)$ is a smooth irreducible surface and moreover:
\begin{itemize}
\item[(2.3)]
if $\dd(Y)=3$, then $\Sigma(Y)$ is a minimal surface of general type
with irregularity $5$, geometric genus $10$ and canonical degree~\mbox{$K_{\Sigma(Y)}^2=45$};
\item[(2.4)]
if $\dd(Y)=4$, then $\Sigma(Y)$ is an abelian surface;
\item[(2.5)]
if $\dd(Y)=5$, then $\Sigma(Y)\cong\P^2$.
\end{itemize}

Let $X$ be a smooth Fano threefold $\rho(X)=1$, $\iota(X)=1$, and $\g(X)\ge 7$.
Then the Hilbert scheme of conics $S(X)$ is a smooth irreducible surface and moreover:
\begin{itemize}
\item[(1.7)]
if $\g(X)=7$, then $S(X)$ is the symmetric square of a smooth curve of genus~$7$;
\item[(1.8)]
if $\g(X)=8$, then $S(X)$ is a minimal surface of general type
with irregularity $5$, geometric genus $10$ and canonical degree~\mbox{$K_{S(X)}^2=45$};
\item[(1.9)]
if $\g(X)=9$, then $S(X)$ is a ruled surface isomorphic to the projectivization of a simple rank $2$ vector bundle on a smooth curve of genus $3$;
\item[(1.10)]
if $\g(X)=10$, then $S(X)$ is an abelian surface;
\item[(1.12)]
if $\g(X)=12$, then $S(X)\cong\P^2$.
\end{itemize}
\end{theorem}

To be honest, most of the information provided by Theorem~\xref{theo:main1}
can be found in the literature
(see \cite{Altman1977}, \cite{DesaleRamanan}, \cite{Furushima1989a},
\cite{Tennison1974}, \cite{Iskovskikh-1980-Anticanonical},
\cite{Puts1982}, \cite{Markushevich1981},
\cite{Iliev-Fano-3folds-genus9}, \cite{iliev2007manivel}, \cite{BrambillaFaenzi}, etc).
Our goal was, in a sense, in collecting all the results together, and cleaning things a bit.
One new improvement here is the case $\iota(X) = 1$ and $\g(X) = 10$, where originally
in~\cite[Proposition 3]{iliev2007manivel} a description of $S(X)$
was known for general $X$ only.
Another improvement is the case $\iota(X) = 1$ and $\g(X) = 9$ where
it was previously known that $S(X)$ is a projectivization of a vector bundle
over a curve of genus~$3$, but simplicity of the vector bundle was known only for
a general threefold~$X$ (see~\cite[\S3]{BrambillaFaenzi}).
Also, our proof for the even genus cases, i.e.\ $\iota(X) = 1$ and \mbox{$\g(X) \in \{8,10,12\}$},
emphasizes the relation between Fano threefolds of index~1 and~2.
We show that if $Y$ is a Fano threefold of index 2 and degree
\begin{equation*}
\dd(Y) = \frac{\g(X)}{2} - 1
\end{equation*}
associated to $X$ by~\cite{kuznetsov2009derived}
(see also Appendix~\ref{section-Lines-and-conics}) then $S(X) \cong \Sigma(Y)$.

For small degrees and genera the situation with the Hilbert schemes of lines and conics is much more complicated.
For instance, in the case $\iota(Y) = 2$ and $\dd(Y) = 2$ the scheme~\mbox{$\Sigma(Y)$} may be singular and in the case $\iota(X) = 1$ and $\g(X) = 6$ the scheme $S(X)$ may be even reducible.
Furthermore, for small values of $\g(X)$ it is quite hard to get a satisfactory explicit description of~$S(X)$.
Say, for~\mbox{$\g(X) = 2$} the only more or less explicit description of~$S(X)$ we are aware of is as a $240$-to-$1$ branched cover of $\P^2$, which is not very much useful.
So, it seems that our result is a kind of optimal in that direction.

A description of the Hilbert schemes of lines and conics allows to produce some results on the automorphism
groups of the corresponding varieties.
The automorphism groups act on the Hilbert schemes and we prove that the action is faithful in all cases listed in Theorem~\xref{theo:main1}.
In fact, we deduce faithfulness from a more general result (Theorem~\xref{theorem:faithful}).
In particular, it applies to all $X$ with $\g(X) \ge 4$ and to some~$X$ with~\mbox{$\g(X) = 3$},
and shows that the action on an irreducible component of a Hilbert scheme is faithful unless $X$ is a double cover of a smooth Fano threefold~$Y$
with $\rho(Y) = 1$, $\iota(Y) = 2$, and~$\dd(Y) = \g(X) - 1$, in which case there is an irreducible component of
the Hilbert scheme of conics on $X$
on which the Galois involution of the double cover acts trivially.

We note that in most of the cases listed in Theorem~\xref{theo:main1},
the maximal linear algebraic subgroup
of the automorphism group of the surface $\Sigma(Y)$ and $S(X)$
is at most finite; the only exceptions are cases~\mbox{$(2.5)$} and~\mbox{$(1.12)$}. Due to the faithfulness result, this proves
that the automorphism groups of the threefolds
listed in Theorem~\xref{theo:main1}, except the threefolds~\mbox{$(2.5)$} and~\mbox{$(1.12)$},
are finite as well.

Our second main result is an extension of this observation to the following general statement
describing all possible infinite automorphism groups of Fano threefolds of Picard rank~$1$.

\begin{theorem}
\label{theorem:Prokhorov}
Let $X$ be a smooth
Fano threefold with $\rho(X)=1$. Then the group~\mbox{$\Aut(X)$} is finite
unless one of the following cases occurs:
\begin{itemize}
\item
$\iota(X) = 4$ so that $X\cong \P^3$; then $\Aut(X) \cong \PGL_4(\Bbbk)$;
\item
$\iota(X) = 3$ so that $X$ is a quadric in $\P^4$; then $\Aut(X) \cong \PSO_5(\Bbbk)$;
\item
$\iota(X)=2$, $\dd(X)=5$; then $\Aut(X)\cong\PGL_2(\Bbbk)$;
\item\label{theorem:Prokhorov-i=1}
$\iota(X)=1$, $\g(X)=12$, and $X$ is one of the following
\begin{enumerate}
\item\label{theorem:Prokhorov-i=1-mu}
$X=X_{22}^{\mathrm{MU}}$ is the Mukai--Umemura threefold; then $\Aut(X)\cong \PGL_2(\Bbbk)$;
\item\label{theorem:Prokhorov-i=1-a}
$X=X_{22}^{\mathrm a}$ is the threefold of Example~\textup{\xref{example-V22-a}}; then $\Aut(X)\cong \GG_{\mathrm a} \rtimes \mumu_4$;
\item\label{theorem:Prokhorov-i=1-m}
$X=X_{22}^{\mathrm m}(u)$ is a threefold from the one-dimensional family
of Example~\textup{\xref{example-V22-m}}; then $\Aut(X)\cong \GG_{\mathrm m} \rtimes \mumu_2$.
\end{enumerate}
\end{itemize}
\end{theorem}

Note that a Fano threefold with $\rho(X) = 1$, $\iota(X) = 2$, and $\dd(X) = 5$ is unique (up to isomorphism),
see~\cite[Theorem~II.1.1]{Iskovskikh-1980-Anticanonical} or~\cite[3.3.1--3.3.2]{Iskovskikh-Prokhorov-1999}).

Again, we should say that almost all results of Theorem~\ref{theorem:Prokhorov} were already known,
see~\cite{Prokhorov-1990c}.
The new results here is the explicit description of $\Aut(X_{22}^{\mathrm{a}})$ and $\Aut(X_{22}^{\mathrm{m}}(u))$.

Using the classification of Fano threefolds of Picard rank~$1$
(see~\cite[\S12.2]{Iskovskikh-Prokhorov-1999},
or Tables~\xref{table:Fanos-i-ge-2} and~\xref{table:Fanos-i-1})
we conclude that Theorem~\xref{theorem:Prokhorov} implies the following.

\begin{corollary}
\label{corollary:Jac}
Let $X$ be a smooth
Fano threefold with $\rho(X)=1$. If the group $\Aut(X)$ is infinite, then $h^{1,2}(X)=0$.
\end{corollary}

Our proof of Theorem~\xref{theorem:Prokhorov} relies on a classification
of smooth Fano threefolds. It would be interesting to find a proof
of Corollary~\xref{corollary:Jac} that does not depend on a classification,
and use it to obtain an alternative proof of Theorem~\xref{theorem:Prokhorov}.
Note that \cite[Theorem~1]{Tolman10} can be considered
as a symplectic counterpart of Corollary~\xref{corollary:Jac},
and to some extent can be used to recover it; namely, the results of~\cite{Tolman10}
imply that an automorphism group of a smooth Fano threefold $X$
of Picard rank~$1$ can contain a subgroup isomorphic to~$\GG_{\mathrm m}$ only if $h^{1,2}(X)=0$.

\subsection{Applications and future directions}

One of the motivations for writing this paper was the problem of classification of finite subgroups of the Cremona group of rank~3 (cf.~\cite{Prokhorov2009e}, \cite{Prokhorov2011a},
\cite{Prokhorov-2-elementary}, \cite{Prokhorov-Shramov-p-groups}).
This classification problem reduces to investigation of finite automorphism groups of Fano threefolds of Picard number 1 with terminal singularities and Mori fiber spaces.
In particular, it includes classification of finite group of automorphisms of smooth Fano threefolds of Picard number 1.
Notice that here it is important to consider \textit{all} possible Fano threefolds in their deformation classes,
while restricting to \textit{general} Fano threefolds (as it is practiced by many authors) does not work.
This is why we try to push our arguments forward in full generality.

Our results allow to give explicit upper bounds on some parameters of automorphism groups,
which may be useful for further applications to studying birational automorphisms
(see~\cite{Prokhorov-Shramov-JCr3}, and
cf.~\cite{Prokhorov-Shramov-J}, \cite{ProkhorovShramov-RC},
\cite{ProkhorovShramov-dim3}, \cite{Yasinsky2016a}).

It would be interesting to understand which results
of this paper can be generalized to the case of singular Fano threefolds, and
what kind of conclusions one can make about their automorphism groups
(cf.~\cite{Prokhorov-planes}, \cite{Prokhorov-e-QFano7}, \cite{Prokhorov-factorial-Fano-e},
\cite{PrzyjalkowskiShramov2016}).

It would be also nice to extend the results of this paper to higher dimensions.
Naturally, the questions we discuss here become much more complicated.
Besides other things, no classification of higher-dimensional Fano varieties is available
(though, there are some partial results, see, e.g.,
\cite{Kuchle,Kuznetsov-Kuchle,Kuznetsov-Kuchle-c5,Kuznetsov-2018}).

\subsection{Outline of the paper}

The plan of our paper is as follows.

In \S\xref{section:hilbert-schemes} we collect the necessary results about Hilbert schemes of lines and conics on Fano threefolds.
In~\S\ref{subsection:lines-and-conics} we discuss general properties of Hilbert schemes,
while in~\S\ref{subsection:lines-index-2} and~\S\ref{subsection:conics-index-1} we concentrate on Hilbert scheme of lines and conics respectively on Fano threefolds.
Some rather technical parts of the material were moved to Appendix~\xref{appendix:new} for the readers convenience.
Another part of the arguments that uses derived categories perspective and technique is collected in Appendix~\xref{section-Lines-and-conics}.
The main result of this section is a proof of Theorem~\xref{theo:main1}.

In \S\xref{section:automorphisms} we recall various general results
about automorphism groups of algebraic varieties, including actions
on invariant linear systems and some well-known finiteness assertions.

In \S\xref{section:finiteness-for-Fanos} we
prove finiteness of automorphism groups for all Fano threefolds of Picard rank~1 except those listed in Theorem~\xref{theorem:Prokhorov}.
We start in~\S\xref{subsection:faithfulness-general} by proving a general faithfulness result (Theorem~\xref{theorem:faithful}) for an algebraic group action
on irreducible components of Hilbert schemes of (anticanonical) conics on Fano varieties of arbitrary dimension.
In~\S\xref{subsection:act-on-lines} we apply it to the action of the automorphism group of a Fano threefold of index $2$
and degree at least~$3$ on (an irreducible component of) the Hilbert scheme of lines, and
in~\S\xref{subsection:act-on-conics} we apply it to the action of the automorphism group of a Fano threefold of index $1$
and genus at least~$3$ on (an irreducible component of) the Hilbert scheme of conics.
For~\mbox{$\dd \ge 3$} and~\mbox{$\g \ge 7$} we prove faithfulness of these actions and combining it with the description
of Hilbert schemes provided by Theorem~\xref{theo:main1}, deduce finiteness of the automorphism group.
Finally,
in~\S\xref{subsection:finite-Fano} we prove finiteness of the automorphism groups in the remaining (easy)
cases in a more straightforward way.

In \S\xref{section:infinite} we study Fano threefolds of index $1$ and
genus $12$ with infinite automorphism groups via a double projection method
and complete our proof of Theorem~\xref{theorem:Prokhorov}.
In~\S\ref{subsection:v5} we discuss geometry of the Fano threefold $Y$ of index~2 and degree~5 and give an explicit description of its Hilbert scheme of lines.
In~\S\ref{subsection:v22} we explain the double projection method and describe the relation
between the Hilbert scheme of lines on a Fano threefold $X$ of index~1 and genus~12 and the Hilbert scheme of lines on $Y$.
In~\S\ref{subsection:v22-special} we explain the construction of threefolds with infinite automorphism groups, and
in~\S\ref{subsection:aut-explicit} we describe explicitly their Hilbert schemes of lines and automorphisms groups.

In Appendix~\xref{appendix:new} we collect some well-known facts about conics.
First, we remind a classification of surfaces whose Hilbert scheme of conics is at least two-dimensional.
After that we remind a description of normal bundles of reducible and non-reduced conics.

In Appendix~\xref{section-Lines-and-conics} we prove Theorem~\xref{theorem:S-vs-Sigma} relating the Hilbert schemes of conics on Fano threefolds of index 1 and genera $8$, $10$, and $12$
to the Hilbert schemes of lines on Fano threefolds of index 2 and degrees $3$, $4$, and $5$, respectively.
The proof is based on the relation between derived categories of these threefolds established in~\cite{kuznetsov2009derived}.
We remind this approach, discuss some details of the relation, and then prove Theorem~\xref{theorem:S-vs-Sigma}.
We also write down proofs for some well-known results of Mukai concerning vector bundles on Fano threefolds that we could not find in the literature.

\medskip
\textbf{Notation and conventions.}

As we already mentioned,
we work over an algebraically closed field $\Bbbk$ of characteristic $0$.
We assume that the Fano varieties appearing in the paper are smooth unless
the converse is mentioned explicitly.
We remind about the smoothness assumption
only at the most important points of our exposition.

We use the following notation throughout the paper.
By $\Pic(X)$ and $\Cl(X)$ we denote the Picard group and the class group of Weil divisors on the variety $X$, respectively.
Linear equivalence of divisors is denoted by~$\sim$.

For a Fano threefold $X$ we keep the notation $\rho(X)$, $\iota(X)$, $\dd(X)$, and $\g(X)$
for the Picard rank, the Fano index, the degree, and the genus of $X$, respectively.
If $\rho(X)=1$, we always denote by $H$ or $H_X$ the ample generator of~\mbox{$\Pic(X)\cong\mathbb{Z}$}.

If $Z$ is a subscheme in $X$, we denote by $[Z]$ the point corresponding to $Z$ in the appropriate Hilbert scheme,
and by~\mbox{$\Aut(X;Z)$} the group of automorphisms of $X$ that preserve~$Z$.
Similarly, if $[D]$ is a divisor class in $\Pic(X)$ or~$\Cl(X)$,
we denote by $\Aut(X;[D])$ the group of automorphisms of $X$ that
preserve the class~$[D]$.

By $\Gr(k,n)$ we denote the Grassmannian of vector subspaces of dimension~$k$ in a vector space of dimension~$n$;
similarly, by $\Gr(k,W)$ we denote the Grassmannian of vector subspaces of dimension~$k$ in a vector space~$W$.
By a linear section of a Grassmannian we always mean
its linear section in the Pl\"ucker embedding, i.e. in the embedding
defined by the ample generator of its Picard group.
By $v_2\colon\P(V)\to\P(\Sym^2 V)$ we denote the second Veronese
embedding.
We denote by $\mumu_m$ the group of~\mbox{$m$-th} roots of unity (isomorphic to a cyclic group of order $m$).

\medskip
\textbf{Acknowledgements.}
We are grateful to Olivier Debarre, Francesco Russo, Richard Thomas, and Fyodor Zak for useful discussions.
We are also grateful to the referee for reading our paper.

\section{Hilbert schemes of lines and conics}\label{section:hilbert-schemes}

In this section we discuss general properties of Hilbert schemes of lines and conics
on Fano threefolds and give an explicit description for some of them.

\subsection{General properties of Hilbert schemes}
\label{subsection:lines-and-conics}

Let $X$ be a projective variety with a fixed ample divisor class $H$. Recall that
a \emph{line} (or an \emph{$H$-line} to be more precise) on~$X$ is a subscheme $L \subset X$ with Hilbert polynomial
\begin{equation*}
p_L(t) = 1 + t.
\end{equation*}
Similarly,
a \emph{conic} (or an \emph{$H$-conic}) on $X$ is a subscheme $C \subset X$ with Hilbert polynomial
\begin{equation*}
p_C(t) = 1 + 2t.
\end{equation*}
We denote by
\begin{equation*}
\Sigma(X) = \Hilb^{p(t) = 1 + t}(X;H)
\end{equation*}
the Hilbert scheme of lines on $X$,
and by
\begin{equation*}
S(X) = \Hilb^{p(t) = 1 + 2t}(X;H)
\end{equation*}
the Hilbert scheme of conics on $X$.

\begin{lemma}\label{lemma:lines-conics}
Let $X$ be a projective variety with an ample divisor class $H$.

\begin{itemize}
\item[(i)] If $2H$ is very ample and $L \subset X$ is an $H$-line
then $L \cong \PP^1$ and $\CO_X(H)\vert_L \cong \CO_L(1)$.

\item[(ii)]
If $H$ is very ample and $C \subset X$ is an $H$-conic
then $C$ is purely
one-dimensional and
\begin{itemize}
\item[$\bullet$] either $C$ is a \emph{smooth conic}, i.e.\ $C \cong \PP^1$ and $\CO_X(H)\vert_C \cong \CO_C(2)$,
\item[$\bullet$] or $C$ is a \emph{reducible conic}, i.e.\ $C = L_1 \cup L_2$ for two distinct lines $L_1$ and $L_2$ on $X$ intersecting transversally at a point,
\item[$\bullet$] or $C$ is a \emph{non-reduced conic},
i.e. a non-reduced subscheme $C \subset X$ such that $C_{\reduced} = L$ is a line
and $I_{L}/I_C \cong \CO_{L}(-1)$.
\end{itemize}
\end{itemize}
\end{lemma}
\begin{proof}
First, assume that $H$ is very ample. Then we may assume that $X = \PP^n$,
and $H$ is the class of a hyperplane.

If $p_L(t) = 1 + t$ then all irreducible components of $L$ have dimension at most 1,
and the sum of the degrees (with multiplicities) of all one-dimensional components is 1.
Let~$L_0$ be the purely one-dimensional part of $L$ and let $\ell$ be the sum of the lengths
of all zero-dimensional components (including embedded ones). Then by the above
observation~$L_0$ is integral of degree 1, hence $L_0$ is $\PP^1$ linearly embedded into $\PP^n$.
In particular, one has~\mbox{$p_{L_0}(t) = 1 + t$}, hence~\mbox{$p_L(t) = 1 + \ell + t$}, which means $\ell = 0$
and so $L = L_0$.

Analogously, let $p_C(t) = 1 + 2t$. Then all irreducible components of $C$ have dimension at most 1,
and the sum of the degrees (with multiplicities) of all one-dimensional components is 2.
Let $C_0$ be the purely one-dimensional part of $C$ and let $\ell$ be the sum of the lengths
of all zero-dimensional components (including embedded ones). If $C_0$ is integral, then
it is contained in the linear span of any triple of its points. Thus $C_0$ is a divisor
of degree~2 on~$\PP^2$, so $C_0 \cong \PP^1$ and $\CO_X(H)\vert_{C_0} \cong \CO_{C_0}(2)$.
Furthermore, we have~\mbox{$p_{C_0}(t) = 1 + 2t$}, hence~\mbox{$p_C(t) = 1 + \ell + 2t$}
which means $\ell = 0$ and $C = C_0$.

If $C_0$ is not integral, then either it has two different irreducible components $L_1$ and~$L_2$ of degree~1,
or one irreducible component $L$ of degree~1 with multiplicity~2. In the first case,~$L_1$ and~$L_2$ are lines,
hence their scheme-theoretic intersection has length~\mbox{$\delta = 0$} or~\mbox{$\delta = 1$}.
It follows that
\begin{equation*}
p_C(t) = \ell + (1 + t) + (1 + t) - \delta,
\end{equation*}
which means
that $\ell = 0$ and $\delta = 1$. In other words, $L_1$ and $L_2$ meet at a point and~\mbox{$C = L_1 \cup L_2$}.

In the second case we have a canonical epimorphism $\CO_{C_0} \to \CO_L$
and its kernel is a line bundle on $L$, hence is isomorphic to $\CO_L(k)$ for some $k \in \Z$.
Then
\begin{equation*}
p_C(t) = \ell + (1+t) + (1+k+t),
\end{equation*}
which implies $k = - (1 + \ell)$.
On the other hand, it is easy to see that $\CO_L(k)$ is a quotient of $I_L/I_L^2$ which is the conormal
bundle of $L$ in $\PP^n$, hence is isomorphic to~\mbox{$\CO_L(-1)^{\oplus (n-1)}$}.
Therefore $k \ge -1$, so comparing with the previous observation, we see that $k = -1$ and~$\ell = 0$.
In other words, $C_\red = L$ and $I_L/I_C \cong \CO_L(-1)$.

Finally, assume that $H$ is not ample, but $2H$ is very ample and $p_L(t) = 1 + t$.
Then with respect to $2H$ the Hilbert polynomial of $L$ is $1 + 2t$, hence in the embedding
of $X$ given by the linear system~$2H$ it is a conic. But it can be neither reducible, nor non-reduced conic,
since $X$ contains no curves which have degree 1 with respect to $2H$. Thus $L \cong \PP^1$
and $\CO_X(2H)\vert_L \cong \CO_L(2)$, which implies the claim.
\end{proof}

\begin{remark}\label{remark:double-lines}
If for a line $L \subset X$ there is a non-reduced conic $C$ with $C_\red = L$, we will say that $L$ \emph{admits a structure of a non-reduced conic}.
It is worth noting that in contrast to the case of a projective space,
not every line admits such a structure.
Indeed, as we have seen in the proof of Lemma~\xref{lemma:lines-conics} above, a line $L$ admits a structure of a non-reduced conic if and only if there is an epimorphism
$\CN^\vee_L = I_L/I_L^2 \to \CO_L(-1)$.
In Remark~\xref{remark:dl-special} below we discuss for which lines on Fano threefolds this holds.
\end{remark}

\begin{remark}
It is easy to see that one cannot have the same results as in Lemma~\xref{lemma:lines-conics} under weaker assumptions.
Indeed, assume we consider $\Sigma(X)$ and only the divisor~$3H$, but not~$2H$, is very ample.
Then we can realize $\Sigma(X)$ as a subscheme in $\Hilb^{p(t) = 1 + 3t}(\PP^n; 3H)$.
The latter, however, has two irreducible components: one parameterizing normal rational cubic curves, and the other parameterizing plane cubics
plus a point (possibly an embedded one). Therefore the same is true in general for $\Sigma(X)$.
Similarly, assume we consider~\mbox{$S(X)$} and only~$2H$, but not~$H$, is very ample.
Then we can realize~$S(X)$ as a subscheme in~\mbox{$\Hilb^{p(t)=1+4t}(\PP^n; 2H)$}.
The latter Hilbert scheme also has several irreducible components,
some of which parameterize curves of other types
than those listed in Lemma~\xref{lemma:lines-conics}(ii).
\end{remark}

From now on we consider the Hilbert schemes of lines and conics on
Fano threefolds of Picard rank~1 and index~1 or~2
(with respect to the ample generator $H$ of the Picard group). We note that the Hilbert schemes $\Sigma(X)$ and $S(X)$ are
nonempty by \cite{Shokurov1979a} (see also \cite{Reid-lines}). As it was explained above to avoid pathologies
when considering $\Sigma(X)$ we should restrict to the case when $2H$ is very ample, i.e., to Fano threefolds of index~1 and genus $\g \ge 3$,
as well as Fano threefolds of index~2 with~\mbox{$\dd \ge 2$}.
Similarly, when considering~$S(X)$ we should restrict to the case when $H$ is very ample:
in the index~1 case
this means that either $\g \ge 4$, or $\g = 3$ and $X$ is a quartic threefold,
while in the index~2 case
this means $\dd \ge 3$.

Under our assumptions, by Lemma~\xref{lemma:lines-conics} both lines and conics are locally complete intersections, hence their conormal and normal sheaves are locally free.
We will need some facts about them. The first is quite standard.

\begin{lemma}[{\cite[Propositions~III.1.3(ii) and III.2.1(i), Lemma III.3.2]{Iskovskikh-1980-Anticanonical}}]
\label{lemma:normal-bundles}
If $L$ is a line and $C$ is a smooth conic on a Fano threefold $X$ of index $1$ then
\begin{equation*}
\CN_{L/X} \cong \CO_L(a) \oplus \CO_L(-1-a)
\qquad\text{and}\qquad
\CN_{C/X} \cong \CO_C(a) \oplus \CO_C(-a)
\end{equation*}
for some $a \ge 0$.

If $L$ is a line and $C$ is a smooth conic on a Fano threefold $Y$ of index $2$ then
\begin{equation*}
\CN_{L/Y} \cong \CO_L(a) \oplus \CO_L(-a)
\qquad\text{and}\qquad
\CN_{C/Y} \cong \CO_C(1+a) \oplus \CO_C(1-a)
\end{equation*}
for some $a \ge 0$.
\end{lemma}

It is a bit harder to deal with the normal bundle of a reducible or non-reduced conic~$C$
(see, however,~\S\xref{subsection:degenerate-conics}).

Recall that by~\cite{FGA} or \cite[Theorem I.2.8]{Kollar-1996-RC} the tangent space to the Hilbert scheme at a point corresponding to a locally complete intersection
subscheme $Z \subset X$ is~\mbox{$H^0(Z,\CN_{Z/X})$} and the obstruction space is $H^1(Z,\CN_{Z/X})$.
Therefore, the dimension of any irreducible
component of the Hilbert scheme is bounded from below by
the Euler characteristic~$\chi(\CN_{Z/X})$ of the normal bundle.
By Lemma~\xref{lemma:normal-bundles} and Corollaries~\xref{corollary:normal-reducible-conic-3fold}, and~\xref{corollary:normal-nonreduced-conic-3fold}
in the cases that are most relevant for us this gives.

\begin{corollary}\label{proposition:dimensions-hilbert-schemes}
The following assertions hold.
\begin{enumerate}
\item[(i)]
If $Y$ is a Fano threefold of index $2$, then the dimension of any component of $\Sigma(Y)$ is at least $2$.

\item[(ii)]
If $X$ is a Fano threefold of index $1$, then the dimension of any component of $\Sigma(X)$ is at least $1$.

\item[(iii)]
If $X$ is a Fano threefold of index $1$, then the dimension of any component of $S(X)$ is at least $2$.
\end{enumerate}
\end{corollary}

A bit later we will see that in all the cases listed in Corollary~\xref{proposition:dimensions-hilbert-schemes},
the Hilbert schemes are equidimensional of dimensions 2, 1, and 2 respectively (see Lemmas~\xref{lemma-hilb-x-1}, \xref{lemma-hilb-y-1}, \xref{lemma:exotic-component}, and~\xref{lemma-hilb-x-2}).

In what follows we will say that a line or a smooth conic is \emph{ordinary}, if in the notation of Lemma~\xref{lemma:normal-bundles} we have $a = 0$, and \emph{special}, if $a \ge 1$.
Furthermore, if $a=1$ we will say that the corresponding line (or conic) is \emph{$1$-special},
and if $a \ge 2$ we will say that it is $2$-special.

\begin{corollary}\label{corollary:hilbert-smoothness}
The following assertions hold.
\begin{enumerate}
\item[(i)] If $Y$ is a Fano threefold of index $2$, then the Hilbert scheme $\Sigma(Y)$ is smooth of dimension $2$ at points
corresponding to ordinary lines or $1$-special lines and singular at points corresponding to $2$-special lines.

\item[(ii)] If $X$ is a Fano threefold of index $1$, then the Hilbert scheme $\Sigma(X)$ is smooth of dimension $1$ at points
corresponding to ordinary lines,
and is singular at points corresponding to special lines.

\item[(iii)] If $X$ is a Fano threefold of index $1$, then
the Hilbert scheme $S(X)$ is smooth of dimension~$2$ at points
corresponding to smooth ordinary or smooth $1$-special conics,
and singular at points corresponding to smooth $2$-special conics.
\end{enumerate}
\end{corollary}
\begin{proof}
By Lemma~\xref{lemma:normal-bundles} in the cases claimed to be corresponding to smooth points the obstruction space $H^1(Z,\CN_{Z/X})$ vanishes,
and in the cases claimed to be corresponding to singular points the tangent space $H^0(Z,\CN_{Z/X})$ jumps.
\end{proof}

\begin{remark}\label{remark:dl-special}
Note that according to Remark~\xref{remark:double-lines},
only $1$-special lines on Fano threefolds (both of index 1 and 2)
admit a structure of a non-reduced conic,
and this structure is unique.
In particular, if $X$ is a Fano threefold of index 1 such that $\Sigma(X)$ is smooth, then~$X$ has no non-reduced conics.
\end{remark}

As we will see below, it is useful to know that Fano threefolds do not contain some special surfaces.
We check this in the next lemma.

\begin{lemma}\label{lemma:cones}
Let $X$ be a Fano threefold with $\rho(X)=1$ and $\iota(X)=1$,
and suppose that~\mbox{$-K_X$} is very ample.
Then the following assertions hold.

\begin{enumerate}
\item [(i)]
The threefold $X$ contains neither
the Veronese surface $v_2(\P^2)$, nor any of its linear projections.
\item [(ii)]
If $X$ contains a two-dimensional cone $Z$
then $X$ is a quartic in~$\PP^4$ and the base of~$Z$ is a smooth plane quartic curve.
\end{enumerate}

\end{lemma}
\begin{proof}
Assume that $Z \subset X$ is one of the surfaces listed in assertion~(i), so that in
particular~\mbox{$H^2\cdot Z \le 4$}.
Since $\rho(X) = 1$, we have~\mbox{$Z \sim rH$} for some positive integer $r$, hence
\begin{equation*}
4\ge H^2\cdot Z = rH^3 = r(2\g(X) - 2)
\end{equation*}
Since $-K_X$ is very ample, we have $\g(X) \ge 3$.
Hence the only possible case is when~$X$ is a quartic in~$\P^4$, $r = 1$, and $H^2\cdot Z=4$, so that $Z$ is a regular projection of the Veronese surface.
Moreover, we see that $Z$ is a hyperplane section of the smooth hypersurface~\mbox{$X\subset\P^4$},
so that $Z$ has at worst isolated singularities, and $Z$ is contained in~$\P^3$.
But the latter is impossible for a regular projection of a Veronese surface.
This gives assertion~(i).

If $Z \subset X$ is a two-dimensional cone with vertex at a point $P \in Z$ then $Z$ is contained
in the embedded tangent space to $X$ at $P$. Since $X$ is smooth, the embedded tangent space to $X$ at $P$ is $\PP^3$, so $Z$
is an irreducible component of a hyperplane section of $X$. But since~\mbox{$\Pic(X) =\Z\cdot H$},
it follows that $Z$ is a hyperplane section, so~\mbox{$\g(X) = 3$} and hence~$X$ is a quartic threefold.
The base of $Z$ is a quartic curve; it is smooth since $Z$, being a hyperplane section of a smooth hypersurface $X \subset \P^4$, can have at worst isolated singularities.
This gives assertion~(ii).
\end{proof}

An easy parameter count shows that a general quartic threefold in $\PP^4$ does not contain cones.
However, there are examples of quartic threefolds with cones.

\begin{example}[see \cite{Tennison1974}]\label{example:quartic-cone}
Consider the Fermat quartic threefold
\begin{equation*}
X = \{x_0^4 + x_1^4 + x_2^4 + x_3^4 + x_4^4 = 0\} \subset \PP^4.
\end{equation*}
Let $P \in X$ be a point with the last three coordinates equal to zero (there are four such points) and consider the plane $\Pi = \{x_0 = x_1 = 0\}$.
Consider the hyperplane~\mbox{$H(P,\Pi) \subset \PP^4$}
spanned by the point $P$ and the plane $\Pi$.
Then $X \cap H(P,\Pi)$ is the cone with vertex~$P$ and the base being the plane Fermat quartic $\Pi \cap X$.
Using the action of the automorphism group~\cite{Shioda-Fermat-curves}
\begin{equation*}
\Aut(X)\cong\mumu_4^4 \rtimes \SS_5,
\end{equation*}
we can construct $4\cdot 10 = 40$ such cones.
\end{example}

\subsection{Hilbert schemes of lines}
\label{subsection:lines-index-2}

Let $X$ be a smooth Fano threefold. Let $\Sigma_0$ be an irreducible component of the Hilbert scheme
of lines $\Sigma(X)$, and consider the reduced scheme
structure on~$\Sigma_0$.
Restricting to $\Sigma_0$ the universal family of lines,
we obtain a diagram
\begin{equation}\label{diagram-universal-line}
\vcenter{\xymatrix{
& \CL_0(X) \ar[dl]_q \ar[dr]^p \\
\Sigma_0 && X
}}
\end{equation}
The map $q\colon\CL_0(X) \to \Sigma_0$ is a $\P^1$-bundle.
Let $L \subset X$ be a line corresponding to a point~$[L]$ in the component $\Sigma_0$ of the Hilbert scheme.
The fiber $q^{-1}([L])$ is identified by the map~$p$ with the line $L$. Note that the normal bundle of $L$ in $\CL_0(X)$
is the trivial bundle of rank equal to the dimension of the tangent space to $\Sigma_0$ at $[L]$.
So,
the differential of $p$ is the map
\begin{equation}\label{equation-dp-lines}
dp \colon
\CN_{L/\CL_0(X)} = T_{[L]}\Sigma_0 \otimes \CO_L \hookrightarrow
T_{[L]} \Sigma(X) \otimes \CO_L =
H^0(L,\CN_{L/X}) \otimes \CO_L \to \CN_{L/X}
\end{equation}
with the last map being given by evaluation.
This is very useful for understanding the infinitesimal structure of the map $p$ along $L$.

\begin{lemma}
\label{lemma-hilb-x-1}
If $X$ is a Fano threefold with $\rho(X)=1$, $\iota(X)=1$, and very ample~$-K_X$, then
every irreducible component of the Hilbert scheme $\Sigma(X)$ of lines on $X$ is one-dimensional.
\end{lemma}
\begin{proof}
Let $\Sigma_0 \subset \Sigma(X)$ be an irreducible component of dimension $k \ge 2$,
and consider the reduced scheme structure on~$\Sigma_0$.
Consider the map~\eqref{equation-dp-lines}.
Its source is a trivial vector bundle,
and by Lemma~\xref{lemma:normal-bundles} its target is
\begin{equation*}
\CN_{L/X} \cong \CO_L(a) \oplus \CO_L(-1-a)
\end{equation*}
with $a \ge 0$.
Since the second summand has no global sections, the image of $dp$ is contained in the first summand, hence the rank of $dp$ does not exceed 1.
Moreover, since $\Sigma_0$ is reduced, so is~$\CL_0(X)$, and hence so is the general fiber of the map $p$.
This means that the map
\begin{equation*}
p\colon\CL_0(X) \to X
\end{equation*}
has fibers of dimension at least $k-1$,
hence the image $Z = p(\CL_0(X))$ has dimension at most
$k + 1 - (k - 1) = 2$.
Therefore, $Z \subset X$ is a surface with $\dim\Sigma(Z) \ge k \ge 2$.
By Corollary~\xref{corollary:lines-conics-plane-1} the surface $Z$ is a plane,
but by Lemma~\xref{lemma:cones}(i) the threefold $X$ contains no planes,
which is a contradiction.
\end{proof}

\begin{remark}
\label{remark:nonreduced-sigma}
On most Fano threefolds $X$ with $\rho(X)=1$ and $\iota(X)=1$ a general point of every irreducible component of $\Sigma(X)$ is an ordinary line.
However, there are some exceptions.
First, if~$X$ is the Mukai--Umemura threefold of genus 12 (see \cite[\S 6]{Mukai-Umemura-1983} or Theorem~\xref{theorem-mu-22} below)
then all lines on $X$ are special, and in fact $\Sigma(X)$ is everywhere non-reduced with~\mbox{$\Sigma(X)_{\operatorname{red}}\cong\PP^1$}
(see Proposition~\ref{proposition:lines-special-v22}).
In the opposite direction not that much is known. What we know is that the
Mukai--Umemura threefold is the only one with everywhere non-reduced $\Sigma(X)$ in genus~12,
and that in genus~10 and~9 there are no threefolds with~\mbox{$\Sigma(X)$} everywhere non-reduced
\cite{Prokhorov-1990b}, \cite{Gruson-Laytimi-Nagaraj}.

Another interesting example is a quartic $X$ in $\PP^4$ containing a cone (see Example~\xref{example:quartic-cone}) so that $\g(X) = 3$.
In this case each line $L$ on the cone has a structure of a non-reduced conic (obtained by
intersecting the cone with its tangent plane along $L$),
hence by Remark~\xref{remark:dl-special} each such~$L$ is 1-special, hence
the corresponding irreducible component of~$\Sigma(X)$ is everywhere non-reduced with the underlying reduced scheme being a smooth plane quartic.
For instance, if $X$ is the Fermat quartic of Example~\ref{example:quartic-cone},
then $\Sigma(X)$ is the union of 40 such non-reduced components, (see \cite[Example in \S 2]{Tennison1974}).
\end{remark}

\begin{remark}
Suppose that $X$ is a Fano threefold with $\rho(X)=1$ and $\iota(X)=1$.
If $X$ is general in the corresponding deformation family,
then $\Sigma(X)$ is a smooth curve, and its genus can be
computed in every case, see~\cite[Theorem~4.2.7]{Iskovskikh-Prokhorov-1999}.
\end{remark}

Now consider the Hilbert scheme of lines on threefolds $Y$ of index 2.

\begin{lemma}\label{lemma-hilb-y-1}
Let $Y$ be a Fano threefold with $\rho(Y)=1$ and $\iota(Y)=2$.
Suppose that the divisor $2H$ is very ample, i.\,e. that $\dd(Y)\ge 2$.
Then every irreducible component of the Hilbert scheme $\Sigma(Y)$ of lines on $Y$ is two-dimensional and its general point corresponds to an ordinary line.
In particular, every irreducible component of $\Sigma(Y)$ is generically smooth.
Moreover,
the map $p\colon\CL_0(Y) \to Y$ is surjective, generically finite, does not contract divisors, and is not birational.
\end{lemma}
\begin{proof}
Let $\Sigma_0 \subset \Sigma(Y)$ be an irreducible component, and consider the reduced scheme structure
on~$\Sigma_0$.
Assume that a line corresponding to a general point of $\Sigma_0$ is special.
By the argument of Lemma~\xref{lemma-hilb-x-1} the rank of the map $dp$ does not exceed 1, and the map
\begin{equation*}
p\colon\CL_0(Y) \to Y
\end{equation*}
has fibers of dimension at least $k-1$, where $k = \dim \Sigma_0$. Therefore, the image
\mbox{$Z = p(\CL_0(Y))$}
has dimension
at most $k + 1 - (k - 1) = 2$. Thus $Z \subset Y$ is a surface,
and by Corollary~\xref{proposition:dimensions-hilbert-schemes} one has $\dim\Sigma(Z) \ge \dim \Sigma_0 = k \ge 2$.
By our assumption the divisor~$2H$ is very ample, hence by Corollary~\xref{corollary:lines-conics-plane-1} the surface $Z$ is a plane.
But $Y$ cannot contain a plane by adjunction, which gives a contradiction.

Therefore a general point of~$\Sigma_0$ corresponds to an ordinary line $L$,
hence $\dim \Sigma_0 = 2$ by Corollary~\xref{corollary:hilbert-smoothness}.
Moreover, for such $L$ all the maps in~\eqref{equation-dp-lines} are isomorphisms, so
the map $dp$ is an isomorphism on $L$, hence the map~$p$ is dominant and unramified along $L$. Since $p$ is also proper, it is surjective.
Moreover, since~\mbox{$\dim \CL_0(Y) =3= \dim Y$}, the map $p$ is generically finite.

Now consider the ramification locus $R(p) \subset \CL_0(Y)$ of the map $p$.
Let $L$ be a line corresponding to an arbitrary point of $\Sigma_0$.
If $L$ is an ordinary line then we have already seen that~$p$ is unramified along $L$.
If, however, $L$ is special, the map~$dp$ is degenerate at all points of $L$.
Therefore the ramification locus $R(p) \subset \CL_0(Y)$ is just the preimage under~$q$ of the locus of special lines in $\Sigma_0$.

Assume that $D \subset \CL_0(Y)$ is an irreducible divisor contracted by $p$.
Then $D \subset R(p)$, hence $D$ is a union of fibers of $q$.
Therefore, $p(D)$ is a union of lines. Since $D$ is irreducible and~\mbox{$\dim p(D) < \dim D = 2$},
it is just one line $L$. But then $D \subset q^{-1}([L])$ is not a divisor.

Assume that $p$ is birational. Since $\rho(\CL_0(Y)) \ge 2$ and $\rho(Y) = 1$,
the morphism
$p$ cannot be an isomorphism. Since $Y$ is smooth, the exceptional locus of $p$
should be a divisor contracted by $p$ (see {\cite[\S2.3, Theorem 2]{Shafarevich1994a}}),
which contradicts the above conclusions.
\end{proof}

\begin{remark}
If $Y$ is a Fano threefold with $\rho(Y) = 1$, $\iota(Y) = 2$ and
$\dd(Y)=4$ or $5$ then the map $p$ is finite of degree $4$ and $3$, respectively.
This is no longer true in the cases~\mbox{$\dd(Y)=3$} and~\mbox{$\dd(Y)=2$}.
For a cubic threefold $Y\subset \PP^4$ the map $p$
has one-dimensional fibers exactly when~$Y$ contains \emph{generalized Eckardt points},
i.e. points $P$ such that the embedded tangent space at~$P$
cuts out a cone on~$Y$.
For example, the Fermat cubic contains 30 generalized Eckardt points.
Similarly, there are examples of double covers of $\PP^3$ branched in quartic surfaces
(for example, over Fermat quartic surfaces) that
contain points over which $p$ is not finite.
\end{remark}

The following result is well known
(see e.\,g.~\cite[Proposition~III.1.3(iii)]{Iskovskikh-1980-Anticanonical}).

\begin{proposition}\label{hilb-lines-smooth}
Let $Y$ be a Fano threefold with $\rho(Y) = 1$, $\iota(Y) = 2$, and $\dd(Y) \ge 3$.
Then the Hilbert scheme of lines $\Sigma(Y)$ is a smooth surface.
\end{proposition}
\begin{proof}
By Corollary~\xref{corollary:hilbert-smoothness} it is enough to show that there are no $2$-special lines.
Since~\mbox{$\dd(Y) \ge 3$},
the class $H$ is very ample and defines an embedding~\mbox{$Y \hookrightarrow \PP^n$}.
Consider the standard exact sequence
\begin{equation*}
0 \to \CN_{L/Y} \to \CN_{L/\PP^n} \to \CN_{Y/\PP^n}\vert_L \to 0.
\end{equation*}
Note that $\CN_{L/\PP^n} \cong \CO_L(1)^{\oplus (n-1)}$. Thus $\CN_{L/Y}$ is a subbundle in
the direct sum of $n-1$ copies of $\CO_L(1)$. This means that $\CN_{L/Y}$ cannot have a summand isomorphic to $\CO_L(a)$ with~\mbox{$a \ge 2$}.
Hence $L$ cannot be $2$-special.
\end{proof}

\begin{remark}
If $\dd(Y) = 2$, so that
$f\colon Y \to \P^3$ is a double cover of $\P^3$ branched in a smooth quartic surface $S$, the Hilbert scheme of lines $\Sigma(Y)$ is,
in fact, a double cover of the subscheme of $\Gr(2,4)$ parameterizing bitangent lines to the surface $S$,
branched in a finite number of points corresponding to lines contained in $S$.
Moreover, if $L_0 \subset S$ is such a line,
and $L= f^{-1}(L_0)_{\mathrm{red}}$, then $L$ is a $2$-special line on $Y$ and hence $\Sigma(Y)$ is singular at~$L$
(cf. \cite[Remark to Proposition III.1.3]{Iskovskikh-1980-Anticanonical}).
\end{remark}

For Fano threefolds of index 2 and degree at least 3 one can describe $\Sigma(Y)$ explicitly.

\begin{proposition}
\label{hilb-lines-explicit}
Let $Y$ be a \textup(smooth\textup) Fano threefold with $\rho(Y) = 1$, $\iota(Y) = 2$, and~\mbox{$\dd(Y)\ge 3$}.
Then $\Sigma(Y)$ is smooth and irreducible. Moreover
\begin{itemize}
\item[(i)] if $\dd(Y)=3$, then $\Sigma(Y)$ is a minimal surface of general type
with irregularity $5$, geometric genus $10$, and canonical degree~\mbox{$K_{\Sigma(Y)}^2=45$};
\item[(ii)] if $\dd(Y)=4$, then $\Sigma(Y)$ is an abelian surface;
\item[(iii)] if $\dd(Y)=5$, then $\Sigma(Y)\cong\P^2$.
\end{itemize}
\end{proposition}
\begin{proof}
If $\dd(Y)=3$, then $Y$ is a cubic hypersurface in $\P^4$,
and the assertion holds by~\mbox{\cite[\S1]{Altman1977}}.
If $\dd(Y)=4$, then $Y$ is a complete intersection of two quadrics in~$\P^5$, and the assertion holds
by~\cite[Theorem~5]{NarasimhanRamanan1969} (see also \cite[Theorem 4.8]{Reid1972}, \cite[Theorem~2]{DesaleRamanan}, \cite[\S6.3]{Griffiths-Harris-1978}).
If $\dd(Y)=5$, then $Y$ is isomorphic
to a linear section $\Gr(2,5)\cap\P^6$
of the Grassmannian $\Gr(2,5)\subset\P^9$,
and the assertion holds by~\mbox{\cite[Proposition~III.1.6]{Iskovskikh-1980-Anticanonical}}
or \cite{Furushima1989a}
(see also~\S\ref{subsection:v5} for an explicit description of lines and
Proposition~\ref{proposition-v22-v5} for an alternative approach).
\end{proof}

\begin{remark}
\label{remark:hyperelliptic-curve}
The abelian surface $\Sigma(Y)$ associated to a Fano threefold $Y$ with $\rho(Y) = 1$, $\iota(Y) = 2$, and~\mbox{$\dd(Y) = 4$} can be described as follows. Recall that such $Y$
is an intersection of two quadrics in $\P^5$.
The corresponding pencil contains precisely 6 degenerate quadrics (\cite[Proposition~2.1]{Reid1972}),
so one can consider the double cover $B(Y) \to \P^1$ branched in these six points.
This is a curve of genus 2. It can be regarded as a curve parameterizing the families of
planes in the quadrics of our pencil.
One can show that $\Sigma(Y)$ is isomorphic to the Jacobian of the curve~$B(Y)$,
see~\cite[Theorem~5]{NarasimhanRamanan1969}.
Moreover, the surface~$\Sigma(Y)$ is isomorphic to
the intermediate Jacobian of~$Y$ (as an abstract variety), see \cite[\S6.4]{Griffiths-Harris-1978}.
\end{remark}

\subsection{Hilbert schemes of conics}
\label{subsection:conics-index-1}

In this section we restrict to the case of smooth Fano threefolds $X$ with $\rho(X)=1$ and
$\iota(X)=1$ and their Hilbert schemes of conics $S(X)$.
Let~$S_0$ be an irreducible component of $S(X)$, and consider the reduced scheme
structure on~$S_0$.
Restricting to $S_0$ the universal family of conics, we obtain a diagram
\begin{equation}\label{diagram-universal-conic}
\vcenter{
\xymatrix{
& \CC_0(X) \ar[dl]_q \ar[dr]^p \\
S_0 && X
}
}
\end{equation}
The map $q\colon\CC_0(X) \to S_0$ is a conic bundle.
Let $C \subset X$ be a conic corresponding to a point $[C]$ in the component $S_0$ of the Hilbert scheme.
The fiber $q^{-1}([C])$ is identified by the map $p$ with the conic $C$. The normal bundle of $C$ in $\CC_0(X)$
is the trivial bundle of rank equal to the dimension of the tangent space to $[C]$ at $S_0$.
Like in~\eqref{equation-dp-lines},
the differential
of $p$ is the map
\begin{equation}\label{equation-dp-conics}
dp \colon
\CN_{C/\CC_0(X)} = T_{[C]} S_0 \otimes \CO_C \hookrightarrow
T_{[C]} S(X) \otimes \CO_C =
H^0(C,\CN_{C/X}) \otimes \CO_C \to \CN_{C/X}
\end{equation}
with the last map being given by evaluation.

We will call an irreducible component of $S(X)$ \emph{exotic}
if it does not contain smooth conics.
The next lemma shows that exotic components appear only for quartics with cones and describes them explicitly.

\begin{lemma}\label{lemma:exotic-component}
Let $X$ be a Fano threefold with $\rho(X)=1$ and $\iota(X)=1$,
and suppose that~\mbox{$-K_X$} is very ample.
Let $S_0 \subset S(X)$ be an exotic component.
Then $X$ is a quartic with a cone, and $S_0 \cong \Hilb^2(B)$, where $B$ is a smooth curve, which is the base of the cone.
In particular, one has~\mbox{$\dim S_0 = 2$}.
Moreover, the irreducible component of $S(X)$ underlying $S_0$ is everywhere non-reduced.
\end{lemma}
\begin{proof}
By Corollary~\xref{proposition:dimensions-hilbert-schemes} we have $\dim S_0 \ge 2$.
On the other hand, by Lemma~\xref{lemma-hilb-x-1} every irreducible component of $\Sigma(X)$ is one-dimensional.
Since a line $L$ on $X$ admits at most one structure of a non-reduced conic (see Remark~\xref{remark:dl-special}),
it follows that a conic corresponding to a general point of~$S_0$ is a union of two distinct lines.
Since $X$ does not contain projections of the Veronese surface by Lemma~\xref{lemma:cones}(i)
(in particular, $X$ does not contain smooth quadric surfaces),
we deduce from Lemma~\xref{lemma:many-reducible-conics} that~$X$ contains a two-dimensional cone with base $B$
such that $S_0$ is the set of conics formed by unions of rulings of the cone.
In other words, one has
\begin{equation*}
S_0 = \{ L_{b_1} \cup L_{b_2} \mid (b_1,b_2) \in \Sym^2(B) \}.
\end{equation*}
Moreover, $X$ is a quartic threefold and $B$ is a smooth curve by Lemma~\xref{lemma:cones}(ii).
In particular, one has $\Sym^2(B) \cong \Hilb^2(B)$.

As we already mentioned in Remark~\ref{remark:nonreduced-sigma} the component of $\Sigma(X)$ underlying $B$ is everywhere non-reduced.
A similar argument shows that the component of $S(X)$ underlying~$\Hilb^2(B)$ is also everywhere non-reduced.
Indeed, if $C = L_{b_1} \cap L_{b_2}$ spans a plane $\Pi$, the corresponding double plane provides $C$ with a non-reduced structure
corresponding to a surjective map $\CN^\vee_{C/X} \to \CO_X(-1)\vert_C$.
By duality this gives an embedding $\CO_X(1)\vert_C \to \CN_{C/X}$, hence $\dim H^0(C,\CN_{C/X}) \ge \dim H^0(C,\CO_X(1)\vert_C) = 3$.
\end{proof}

\begin{lemma}\label{lemma-hilb-x-2}
If $X$ is a Fano threefold with $\rho(X)=1$, $\iota(X)=1$ and $-K_X$ very ample then
every irreducible component $S_0$ of the Hilbert scheme $S(X)$ of conics on $X$ is two-dimensional.
If $S_0$ is not exotic, then the map $p\colon\CC_0(X) \to X$ is surjective, generically finite,
does not contract divisors, and is not birational;
moreover, the natural scheme structure on $S_0$ is generically reduced.
If $S_0$ is exotic, then~$X$ is a quartic and~\mbox{$p(\CC_0(X))$} is a cone over a
smooth curve.
\end{lemma}
\begin{proof}
First, we note that $p$ is surjective unless $S_0$ is exotic.
Indeed, if the image of~$p$ is a surface $Z \subset X$ then $\dim S(Z) \ge \dim S_0 \ge 2$ by Corollary~\xref{proposition:dimensions-hilbert-schemes},
hence by Lemma~\xref{lemma:lines-conics-plane} the surface~$Z$
is a linear projection of the Veronese surface, which contradicts Lemma~\xref{lemma:cones}(i),
or~$Z$ is a cone. In the latter case clearly $S_0$ is an exotic component.

Assume that $\dim S_0=k \ge 3$.
By Lemma~\xref{lemma:exotic-component} the component $S_0$ is not exotic, and by Corollary~\xref{corollary:hilbert-smoothness}
a general point of~$S_0$ corresponds to a smooth special (and even 2-special) conic.
The differential $dp$ on such a conic has rank at most 1 everywhere,
therefore the fibers of the map $p$ have dimension~\mbox{$k-1$}, hence the image $p(\CC_0(X))$ has dimension~\mbox{$k + 1 - (k-1) = 2$}.
In particular, $p$ is not surjective, which contradicts the above conclusions.
Thus we have $\dim S_0 = 2$.

Assume that $S_0$ is not exotic and $\dim S_0 = 2$.
Then $p$ is surjective, and since $\dim\CC_0(X) =3= \dim X$, the morphism~$p$ is generically finite.
Consider the ramification locus $R(p)$ of the map~$p$.
Let $C$ be a conic corresponding to a smooth point of $S_0$.
Then both~$\CN_{C/\CC_0(X)}$ and~$\CN_{C/X}$ are vector bundles on $C$ of rank 2 and Euler characteristic~2,
see Lemma~\xref{lemma:normal-bundles} and Corollaries~\xref{corollary:normal-reducible-conic-3fold}, and~\xref{corollary:normal-nonreduced-conic-3fold}.
Hence the kernel and the cokernel of the map $dp$ have the same rank and Euler characteristic. If the rank
of the cokernel is $0$, then so is the rank of the kernel. Since $\CN_{C/\CC_0(X)}$ is a trivial vector bundle,
it is torsion free, so it follows that the kernel is zero. But then the cokernel is zero as well.
This means that either the cokernel of the map $dp$ is zero, hence the map $p$ is unramified along $C$,
or the support of the cokernel is either $C$, or if $C = L_1 \cup L_2$ is reducible, one of the lines $L_i$.
This shows that away of the $q$-preimage of the singular locus of $S(X)$
the ramification locus $R(p)$ is the union of (irreducible components of) fibers of~$q$.
Arguing as in Lemma~\xref{lemma-hilb-y-1} (with obvious modifications), we conclude that $p$ cannot contract divisors and cannot be birational.

The above arguments also show that a generic point of $S_0$ corresponds to a smooth ordinary conic $C$.
Therefore, the tangent space to $S(X)$ at $C$ is 2-dimensional, hence~$S(X)$ is generically reduced along $S_0$.

Finally, if $S_0$ is an exotic component, a description of Lemma~\xref{lemma:exotic-component} shows that $X$ is a quartic,
and $p(\CC_0(X))$ is a cone over a smooth curve~$B$.
\end{proof}

Our next goal, as before, is a proof of smoothness of $S(X)$ and its explicit description for some $X$.
The direct proof of smoothness is very complicated, since there are three types of conics
and it is much more difficult to analyze the tangent space to the Hilbert scheme
at a point corresponding to a reducible or non-reduced conic, and the corresponding obstruction
space (see Appendix~\xref{subsection:degenerate-conics}, or~\cite[\S3.2]{iliev2011fano}).
Typically such considerations work only for general Fano threefolds.

So, instead of using the above straightforward approach, we use the ideas of~\cite{kuznetsov2009derived},
where it was argued that the geometry of Fano threefolds of index 1 and even genus $\g$ is related to the geometry
of Fano threefolds of index 2 and degree $\dd = \g/2 - 1$. The reason for this is a similarity between the structure
of their derived categories. Using this idea we will prove the following result.

\begin{theorem}\label{theorem:S-vs-Sigma}
Let $X$ be a \textup(smooth\textup) Fano threefold with $\rho(X) = 1$, $\iota(X) = 1$, and
\begin{equation*}
\g(X) \in \{8,10,12\}.
\end{equation*}
Then there is a smooth Fano threefold $Y$ with $\rho(Y) = 1$, $\iota(Y) = 2$, and
\begin{equation*}
\dd(Y) = \frac{\g(X)}{2}-1
\end{equation*}
such that $S(X) \cong \Sigma(Y)$.
\end{theorem}

Since the proof of this result uses a completely different technique, we moved it to Appendix~\xref{section-Lines-and-conics}.
Actually, for $\g(X) = 10$ and $\g(X) = 12$ we identify the Hilbert schemes explicitly and show that these identifications match up,
while
for~\mbox{$\g(X)=8$} we construct a direct isomorphism of $S(X)$ and $\Sigma(Y)$.

Combining Theorem~\xref{theorem:S-vs-Sigma} with a description of Hilbert schemes of lines of index 2 threefolds and
with some other results, we can state now the following proposition.
Recall that a vector bundle $\CE$ is called \emph{simple}, if $\Hom(\CE,\CE) = \Bbbk$.

\begin{proposition}\label{proposition:conics}
Let $X$ be a \textup(smooth\textup) Fano threefold of index $1$ and
genus $\g(X)\ge 7$.
Then $S(X)$ is a smooth irreducible surface and
\begin{itemize}
\item[(i)] if $\g(X)=7$, then $S(X)$ is the symmetric square of a
smooth curve of genus $7$;
\item[(ii)] if $\g(X)=8$, then $S(X)$ is a minimal surface of general type
with irregularity $5$, geometric genus $10$, and canonical
degree~\mbox{$K_{S(X)}^2=45$};
\item[(iii)] if $\g(X)=9$, then $S(X)$ is a ruled surface
that is a projectivization of a simple rank~$2$ vector bundle on a smooth curve of genus $3$;
\item[(iv)] if $\g(X)=10$, then $S(X)$ is an abelian surface;
\item[(v)] if $\g(X)=12$, then $S(X)\cong\P^2$.
\end{itemize}
\end{proposition}
\begin{proof}
If $\g(X)=7$, the assertion holds by~\cite[Theorem~6.3]{Kuznetsov-V12}.
If $\g(X)=8$, the assertion holds by Theorem~\xref{theorem:S-vs-Sigma} and Proposition~\xref{hilb-lines-explicit}(i).
If $\g(X)=9$, the surface $S(X)$ is ruled by~\cite[Proposition~3.10]{BrambillaFaenzi},
and the simplicity of the corresponding vector bundle is proved in Lemma~\xref{lemma:v-simple}.
If $\g(X)=10$, the assertion holds by Proposition~\ref{proposition:S-Pic-iso}, or equivalently by
Theorem~\xref{theorem:S-vs-Sigma} and Proposition~\xref{hilb-lines-explicit}(ii).
If $\g(X)=12$, the assertion holds by~\mbox{\cite[Theorem~2.4]{Kollar2004b}}
(alternatively, one can apply Proposition~\ref{proposition-v22-v5}, or equivalently
Theorem~\xref{theorem:S-vs-Sigma} and Proposition~\xref{hilb-lines-explicit}(iii)).
Smoothness of~$S(X)$ is clear from the above case-by-case analysis.
\end{proof}

Propositions~\xref{hilb-lines-explicit} and~\xref{proposition:conics}
together give Theorem~\xref{theo:main1}.

\begin{remark}
\label{remark:hilb-conics-reducible}
Note that if $X$ is a Fano
threefold of genus~\mbox{$\g(X)\le 6$}, then
the surface~$S(X)$ may
be singular and even reducible.
For example, let $\pi\colon X\to Y$ be a double cover of a smooth Fano variety $Y$
with $\rho(Y)=1$ and $\iota(Y)=2$ branched in a smooth anticanonical divisor.
Then $X$ is a smooth Fano threefold with $\rho(X) = 1$, $\iota(X) = 1$, and $\g(X) = \dd(Y) + 1$ (see Lemma~\xref{lemma:trivial-action-conics} below), and
$S(X)$ is a union of two irreducible components; one of them is identified with
the Hilbert scheme of lines in~$Y$,
and the other is a double cover of the subvariety of the Hilbert scheme of conics $S(Y)$ bitangent
to the branch divisor (see \cite[Proposition 2.1.2]{Iliev1994a} for the case~\mbox{$\g(X)=6$}).
\end{remark}

For the sake of completeness, we conclude this section by a discussion of Hilbert
schemes of conics on some Fano threefolds of Picard rank $1$ and index~$2$.
These results, of course, are well known to experts, however, we do not know a good reference for them
except for the case~\mbox{$\dd(Y)=5$}.
According to our conventions, we consider Hilbert schemes of conics only on those Fano threefolds whose ample generator of the Picard group is very ample.
In the case of index 2 this means that~\mbox{$\dd(Y) \ge 3$}.
Recall the description of the Hilbert scheme of lines~$\Sigma(Y)$ for these threefolds from Proposition~\xref{hilb-lines-explicit}.
Also, recall that there is a curve $B(Y)$ of genus 2 associated to a Fano threefold $Y$ of index 2 and degree 4,
see Remark~\xref{remark:hyperelliptic-curve}.

\begin{proposition}
Let $Y$ be a \textup(smooth\textup) Fano threefold with $\rho(Y)=1$, $\iota(Y)=2$, and~\mbox{$\dd(Y) \ge 3$}.
Then $S(Y)$ is a smooth fourfold, and
\begin{itemize}
\item[(i)] if $\dd(Y)=3$, then $S(Y)$ is a $\PP^2$-bundle over the surface $\Sigma(Y)$.

\item[(ii)] if $\dd(Y)=4$, then $S(Y)$ is a $\PP^3$-bundle over the curve $B(Y)$.

\item[(iii)] if $\dd(Y)=5$, then $S(Y)\cong\P^4$.
\end{itemize}
\end{proposition}
\begin{proof}
Let us start with assertion~(i).
Let $Y \subset \P^4$ be a smooth cubic hypersurface.
The linear span of a conic $C \subset Y$ is a plane~\mbox{$\langle C \rangle \cong \P^2$}.
This plane is not contained in $Y$, because the Picard group of $Y$ is generated by a hyperplane section by Lefschetz theorem.
Hence the intersection~\mbox{$\langle C \rangle \cap Y$} is a plane cubic curve, containing the conic $C$.
This means that
\begin{equation*}
\langle C \rangle \cap Y = C \cup L(C),
\end{equation*}
where $L(C)$ is a line (usually called the residual line of $C$).
It is easy to see that the association $C \mapsto L(C)$ defines a regular map $S(Y) \to \Sigma(Y)$.

The fiber of the map over a point $[L] \in \Sigma(Y)$ is the space of all planes in $\P^4$ containing~$L$ (hence is isomorphic to $\P^2$).
Indeed, if $\Pi$ is such a plane then $\Pi \cap Y = L \cup C(\Pi)$ with~$C(\Pi)$ a conic, and conversely, every conic whose residual line is $L$ spans a plane containing $L$.
Altogether, this shows that
$$
S(Y) \cong \Sigma(Y) \times_{\Gr(2,5)} \Fl(2,3;5),
$$
where $\Fl(2,3;5)$ is the flag variety. Thus, we proved assertion~(i).

The idea of the proof of assertion~(ii) is similar to that for assertion~(i).
Let $Y \subset \P^5$ be a complete intersection of two quadrics.
Given a conic~\mbox{$C \subset Y$}, we consider its linear span~\mbox{$\langle C \rangle \subset \P^5$}.
The restriction to $\langle C \rangle$ of the pencil of quadrics defining $Y$ is a pencil of conics containing $C$.
This means that there is a unique quadric $Q(C)$ in the pencil defining~$Y$ that contains $\langle C \rangle$
(again, because the plane $\langle C \rangle \cong \P^2$ is not contained in $Y$).
In other words, the association
$C \mapsto (Q(C),\langle C\rangle)$
defines a regular map
\begin{equation*}
S(Y) \to \Hilb^{p(t) = (1+t)(2+t)/2}(\CQ/\P^1)
\end{equation*}
into the relative Hilbert scheme of
planes in the divisor~\mbox{$\CQ \subset \P^5\times\P^1$}
(defined by the pencil of quadrics) over $\P^1$.
For a smooth quadric $Q \subset \P^5$ the Hilbert scheme of planes~\mbox{$\Hilb^{p(t) = (1+t)(2+t)/2}(Q)$} is isomorphic to a union of two copies of~$\P^3$,
while for a cone over a smooth quadric in $\P^4$, it is isomorphic to~$\P^3$.
Altogether, this means that the Stein factorization for the canonical map
$\Hilb^{p(t) = (1+t)(2+t)/2}(\CQ/\P^1) \to \P^1$ is a composition of a $\P^3$-bundle with a double cover
$B(Y) \to \P^1$, branched in the points of $\P^1$ corresponding to singular quadrics in the pencil.
This proves assertion~(ii).

For assertion~(iii) see \cite[Proposition~2.32]{Sanna} or~\cite[Proposition~1.2.2]{Iliev1994a}.
\end{proof}

\section{Automorphism groups}
\label{section:automorphisms}

In this section we remind some general results on automorphism groups of projective varieties,
in particular showing that under appropriate conditions they are linear algebraic groups.
We also discuss some general approaches to finiteness of automorphism groups.
Throughout the section we work
under rather general assumptions.

\subsection{Actions on linear systems}
\label{subsection:actions}

Let $X$ be a normal projective variety and let $A$ be a Weil divisor on $X$.
If the linear system $|A|$ is not empty, we denote by
\begin{equation*}
\varphi_{|A|}\colon X \dashrightarrow \P\big(H^0(X, \CO_X(A))^{\vee}\big)
\end{equation*}
the corresponding rational map.
If the class $[A]$ of $A$ in $\Cl(X)$ is invariant with respect to a subgroup $\Gamma \subset \Aut(X)$,
then there is a natural action of $\Gamma$ on $\PP(H^0(X, \CO_X(A))^{\vee})$ and the map $\varphi_{|A|}$ is $\Gamma$-equivariant.
Note also that the $\Gamma$-action on $\PP(H^0(X, \CO_X(A))^{\vee})$ is induced by an action on $H^0(X,\CO_X(A))^\vee$ of a central extension
\begin{equation*}
1 \to \mumu_m \to \widetilde\Gamma \to \Gamma \to 1,
\end{equation*}
where $m = \dim H^0(X,\CO_X(A))$. Indeed, the above exact sequence is the pullback of
\begin{equation*}
1 \to \mumu_m \to \SL(H^0(X, \CO_X(A))^{\vee}) \to \PGL(H^0(X, \CO_X(A))^{\vee}) \to 1
\end{equation*}
via the map $\Gamma \to \Aut(\PP(H^0(X, \CO_X(A))^{\vee})) \cong \PGL(H^0(X, \CO_X(A))^{\vee})$.
The induced map~\mbox{$\widetilde\Gamma \to \SL(H^0(X, \CO_X(A))^{\vee})$} gives a $\widetilde\Gamma$-action on $H^0(X, \CO_X(A))^{\vee}$.

\begin{remark}\label{remark:need-extension}
Note that the action of $\Gamma$ on $\PP(H^0(X, \CO_X(A))^{\vee})$ may not be induced by the action on $H^0(X, \CO_X(A))^\vee$
of $\Gamma$ itself, i.\,e. passing to a central extension above is indeed necessary.
On the other hand, if the sheaf $\CO_X(A)$ admits a $\Gamma$-linearization
(that is, if the action of $\Gamma$ on $X$ lifts to its action on~$\CO_X(A)$),
then the map~\mbox{$\Gamma \to \PGL(H^0(X, \CO_X(A))^{\vee})$} lifts to a map~\mbox{$\Gamma \to \GL(H^0(X, \CO_X(A))^{\vee})$}.
\end{remark}

The following lemma is easy and well known.

\begin{lemma}\label{lemma:anticanonical-ring0}
Let $X$ be a normal projective variety and $A$ be a Weil divisor on $X$.
Let $\Aut(X; [A])$ be the stabilizer of the class $[A]\in\Cl(X)$ in $\Aut(X)$.
If the map $\varphi_{|A|}$
is birational onto its image then the action of $\Aut(X; [A])$ on $\PP(H^0(X, \CO_X(A))^{\vee})$ is faithful.
In particular, in this case $\Aut(X; [A])$ is a linear algebraic group.
\end{lemma}
\begin{proof}
If some element $g \in\Aut(X; [A])$ acts trivially on $\PP((H^0(X, \CO_X(A))^{\vee}))$, then by assumption it also acts trivially
on an open dense subset of $X$, hence on the whole $X$.
\end{proof}

Note that any multiple of the canonical class is invariant under the automorphism group~\mbox{$\Aut(X)$}
and, moreover, has a natural $\Aut(X)$-linearization. Applying Lemma~\xref{lemma:anticanonical-ring0}
and taking into account Remark~\xref{remark:need-extension},
we obtain the following result.

\begin{corollary}\label{corollary:anticanonical-ring}
Let $X$ be a normal projective variety.
Suppose that for some $m\in \Z$ \textup(either positive or negative\textup)
the map~\mbox{$\varphi_{|mK_X|}$} is birational onto its image. Then the action of the group $\Aut(X)$ on
$\P\big(H^0(X,\O_X(mK_X))^{\vee}\big)$
is faithful and lifts to an embedding
\begin{equation*}
\Aut(X) \hookrightarrow \GL(H^0(X,\CO_X(mK_X))^\vee).
\end{equation*}
In particular, $\Aut(X)$ is a linear algebraic group.
\end{corollary}

\begin{corollary}\label{corollary:hypersurface}
Let $X\subset\P^N$ be a normal complete intersection
of dimension~\mbox{$\dim X\ge 3$} that is not contained in a hyperplane in~$\P^N$.
Then there is a natural embedding~\mbox{$\Aut(X)\hookrightarrow\PGL_{N+1}(\Bbbk)$}.
\end{corollary}
\begin{proof}
By Lefschetz theorem one has $\Pic(X)=\Z\cdot H$, where $H$ is a hyperplane
section (see e.\,g. \cite[Corollary~IV.3.2]{Hartshorne1970}).
Thus, the embedding $X \hookrightarrow \P^N$ is given by an invariant linear system
$|H|$, so the assertion follows from Lemma~\xref{lemma:anticanonical-ring0}.
\end{proof}

\subsection{Finiteness results}
\label{subsection:known-finiteness}

Let us recall several easy finiteness results for
automorphism groups of algebraic varieties.

\begin{lemma}
\label{lemma:non-uniruled}
If a linear algebraic group $G$ acts faithfully on a variety $X$ which is not ruled, then $G$ is finite.
\end{lemma}
\begin{proof}
If $G$ is not finite, it contains a subgroup isomorphic to $\GGm$ or $\GGa$.
An open subset of $X$ is covered by one-dimensional orbits of this subgroup, hence $X$ is ruled, which is a contradiction.
\end{proof}

\begin{corollary}\label{corollary:nef-K}
Let $X$ be a variety of Kodaira dimension $\kappa(X)\ge 0$.
Suppose that a linear algebraic group $G$ acts faithfully on $X$.
Then $G$ is finite.
\end{corollary}
\begin{proof}
Since the linear system $|nK_X|$ is not empty for some $n>0$,
the variety $X$ is not uniruled (see \cite[Theorem 1]{MiyaokaMori}).
Thus we can apply Lemma~\xref{lemma:non-uniruled}.
\end{proof}

\begin{corollary}\label{corollary:general-type}
Let $X$ be a variety of general type. Then the group $\Aut(X)$ is finite.
\end{corollary}
\begin{proof}
Apply Corollaries~\xref{corollary:anticanonical-ring}
and~\xref{corollary:nef-K}.
\end{proof}

\begin{remark}
Actually, even the group of birational selfmaps of a variety of
general type is finite, since it coincides with the automorphism group of its canonical model.
\end{remark}

Another collection of finiteness results concerns hypersurfaces and complete intersections.

\begin{theorem}[{see \cite{MatsumuraMonsky}}]
\label{theorem:MatsumuraMonsky}
Let $X$ be a smooth hypersurface of degree $d\ge 3$ in $\P^N$, where
$N\ge 2$. Then the automorphism group of $X$ is finite
unless either $N=2$ and $d=3$, or $N=3$ and $d=4$.
\end{theorem}

There are many classification results
on automorphism groups of hypersurfaces of small degree,
in particular, cubic hypersurfaces (see \cite{Hosoh97},
\cite{Hosoh02}, \cite[\S9.5]{Dolgachev-book},
\cite{Adler78}, \cite{GAL11}, \cite{OY15}).
Also, Theorem~\xref{theorem:MatsumuraMonsky} has the following
recent generalization.

\begin{theorem}[{\cite[Theorem~3.1]{Benoist2013}}, see also \cite{ChenPanZhang-2015}]
\label{theorem:Benoist}
If $X \subset \PP^N$ is a smooth complete intersection of
dimension $\dim X\ge 3$
and codimension $\operatorname{codim}(X) \ge 2$ not contained in a hyperplane
in~$\P^N$, then $\Aut(X)$ is finite.
\end{theorem}

Another well-known finiteness result that we will need is as follows.
Recall that for any morphism $\phi\colon Y\to X$ there is a subgroup
$\Aut(Y/X)\subset \Aut(Y)$ that consists of all automorphisms whose action is fiberwise with respect to $\phi$;
we will refer to this group as \emph{the group of automorphisms of $Y$ over $X$}.

\begin{lemma}
\label{lemma:projectivization-stable-bundle}
Let $\CE$ be a simple vector bundle on a projective scheme~$X$. Then the group~\mbox{$\Aut(\PP_X(\CE)/X)$}
of the automorphisms of the projectivization $\PP_X(\CE)$ over $X$ is finite.
\end{lemma}
\begin{proof}
This is Corollary to Proposition~2
in~\cite{Grothendieck1958geometrie}
(note also that the group denoted by~$\Gamma$ in~\cite{Grothendieck1958geometrie} is a subgroup in the $2$-torsion subgroup of $\Pic(X)$, hence is finite).
\end{proof}

\begin{corollary}
If $\CE$ is a simple vector bundle on a smooth curve $C$ of genus $g > 1$, then the
group $\Aut(\PP_C(\CE))$ is finite.
\end{corollary}
\begin{proof}
Indeed, the morphism $\PP_C(\CE) \to C$ is canonical, hence there is an exact sequence
\begin{equation*}
1 \to \Aut(\PP_C(\CE)/C) \to \Aut(\PP_C(\CE)) \to \Aut(C).
\end{equation*}
The term on the left is finite by Lemma~\xref{lemma:projectivization-stable-bundle}, and the term on the right is finite since~\mbox{$g>1$}, see Corollary~\xref{corollary:general-type}.
Therefore $\Aut(\PP_C(\CE))$ is finite.
\end{proof}

\section{Finiteness for Fano threefolds}
\label{section:finiteness-for-Fanos}

In this section we prove finiteness of automorphism groups for most of smooth Fano threefolds of Picard rank 1.

\subsection{Faithfulness of action on a family of curves}
\label{subsection:faithfulness-general}

In this subsection we prove a general result on faithfulness of an automorphism group action on a Hilbert scheme of
curves of degree 2 with respect to the anticanonical class.
In the next subsections we apply it to Hilbert schemes of lines on Fano threefolds of index 2 and Hilbert schemes of conics on Fano threefolds of index 1.

Let $X$ be a smooth projective variety (of any dimension).
Let $S$ be an irreducible and reduced projective subscheme in a Hilbert scheme of curves on $X$, let
$\CC \subset S \times X$ be the corresponding family of curves, and let
\begin{equation*}
\xymatrix{
& \CC \ar[dl]_q \ar[dr]^p \\
S && X
}
\end{equation*}
be the corresponding diagram of projections.

\begin{theorem}\label{theorem:faithful}
Assume that $X$ is a smooth Fano variety of any dimension greater than~$1$ with~\mbox{$\Pic(X)\cong\mathbb{Z}$} and~$-K_X$ very ample.
Assume that
\begin{itemize}
\item
for general $s \in S$ the fiber $\CC_s$ of $q$ is a smooth rational curve and $\CC_s\cdot(-K_X) = 2$;
\item
the morphism
$p$ is dominant, not birational, and does not contract divisors.
\end{itemize}
Let $G \subset \Aut(X)$ be a nontrivial algebraic
subgroup that acts trivially on $S$.
Then the group $G$ is cyclic of order~$2$,
and a curve $C$ corresponding to a general point of $S$ is the preimage of a curve $C' \subset X' = X/G$ under the quotient map $\pi \colon X \to X'$.

Furthermore, if $\dim X$ is odd then $X'$ is smooth, and
if $\dim X$ is even, then $X'$ is either smooth or has one singular point of type $\frac12(1,\ldots,1)$.
In both cases the branch locus of $\pi$ in $X'$ is the union of a smooth anticanonical divisor $B$ and $\Sing(X')$.

Finally, $X'$ is a Fano variety with $\Pic(X') \cong \Z$, and the divisor $K_{X'}$ is divisible by $2$ in
the class group~\mbox{$\Cl(X')$}.
\end{theorem}
\begin{proof}
First let us show that $G$ is finite.
Since $G$ is an algebraic group, it is enough to show that the connected component $G^0$ of identity in $G$ is trivial.
Since $\CC_s$ is smooth for general $s \in S$ and $p$ is generically finite, a general point $x \in X$ does not lie on a reducible curve from the family.
Since $p$ is dominant and not birational, the fiber $p^{-1}(x)$ over a general point $x \in X$ consists of more than one point.
Thus for a general $x \in X$ there are two distinct irreducible curves $C_1$ and $C_2$ in the family $\CC$ that pass through~$x$.
Since the action of~$G^0$ on $S$ is trivial, both curves $C_i$ are $G^0$-invariant.
Hence~\mbox{$G^0\cdot x \subset C_1 \cap C_2$}.
Since~$C_1$ and~$C_2$ are distinct and irreducible, it follows that $G^0\cdot x$ is finite.
But $G^0$ is connected, hence~\mbox{$G^0\cdot x = x$}.
Thus, a general point of $X$ is fixed by $G^0 \subset \Aut(X)$, hence~$G^0$ is trivial.
This means that the group~$G$ is finite.

Now let $G_0 \subset G$ be a cyclic subgroup of order~\mbox{$n>1$}.
The action of~$G_0$ on~$\CC$ is fiberwise over~$S$.
Therefore, it has two fixed points on a general smooth
fiber of $q$, so
the fixed locus of $G_0$ in $\CC$ contains a divisor which intersects a general smooth fiber of $q$ at two distinct points.
Since the morphism $p \colon \CC \to X$ contracts no divisors, the fixed locus of~$G_0$ in~$X$ contains a divisor $F$ which intersects a general smooth curve from $\CC$ at two points.
Since~$\Pic (X)\cong \Z$, this means that $F \sim -K_X$. Put
\begin{equation*}
V = H^0(X,\CO(-K_X))^\vee,
\end{equation*}
so that $X \subset \P(V)$ is the anticanonical embedding.
The action of $G_0$ on $X$ induces an action on $V$ by Corollary~\xref{corollary:anticanonical-ring}.
The fixed divisor $F$ generates a hyperplane $V_0 \subset V$ and we have a direct sum decomposition
\begin{equation*}
V = V_0 \oplus V_1,
\end{equation*}
where $V_0$ and $V_1$ are eigenspaces for (a generator of) $G_0\cong\mumu_n$, and
$V_1$ is one-dimensional.
It follows that the fixed locus of $G_0$ on $\P(V)$ is $\P(V_0) \sqcup \P(V_1)$, and its fixed locus on $X$
is either $F = X \cap \P(V_0)$, or the union of $F$ with the point $P=\P(V_1) \in \P(V)$ corresponding to the one-dimensional
eigenspace $V_1 \subset V$ (if the point $P$ lies on $X$).

Let $X' = X/G_0$ be the quotient with $\pi\colon X \to X'$ being the projection.
If $P \in X$ then~\mbox{$P' = \pi(P)$} is a quotient singularity of type $\frac1n(1,\ldots,1)$ on $X'$ and $X'_0 = X' \setminus P'$ is smooth;
otherwise $X'$ is smooth and we set $X'_0 = X'$.
Put $X_0 = \pi^{-1}(X'_0)$ and $\pi_0 = \pi\vert_{X_0}$.
Since $\pi$ is a finite morphism, for any Weil divisor $R$ on $X'$ the pull-back $\pi^*R$ is a
well-defined $G_0$-invariant Weil divisor
(the closure of $\pi_0^{-1}(R\vert_{X'_0})$).
Furthermore, one has (see, e.g. \cite[1.7.5]{Fulton-Intersection-theory})
\begin{equation}
\label{equation-faithful-action-conics}
\Cl(X') = \Cl(X'_0) = \Pic(X'_0),\qquad \Cl(X')\otimes \Q=(\Cl(X)\otimes \Q)^{G_0}.
\end{equation}
Since $\pi_0\colon X_0 \to X'_0$ is a cyclic degree $n$ cover with ramification divisor $F$,
the class of the branch divisor~\mbox{$B=\pi_0(F) \subset X'_0$} is divisible by $n$ in $\Pic(X'_0)$,
see e.g. \cite[Theorem~1.2]{Wavrik1968}, so that $B \sim nD$ for some~\mbox{$D \in \Pic(X'_0)$} with $F \sim \pi_0^*D$.

Let $C$ be a smooth curve corresponding to a general point of $\CC$.
Then $C$ does not pass through $P$, since otherwise~\mbox{$p^{-1}(P) \subset \CC$} would be a divisor contracted by $p$.
Thus~\mbox{$C = \pi_0^*(\pi_0(C))$} and hence
\begin{equation}
\label{equation-faithful-action-conics-LD}
2  = C \cdot (-K_X) = C\cdot F = \pi_0^*(\pi_0(C)) \cdot \pi_0^*(D) = n\pi_0(C)\cdot D,
\end{equation}
so $n$ divides 2. Since $n>1$ by our assumption, we have $n = 2$ and $G_0 \cong \mumu_2$.
Furthermore, since $n = 2$ by Hurwitz formula we have
\begin{equation*}
\pi_0^*K_{X'_0} \sim K_{X_0} - F \sim -2F.
\end{equation*}
Applying $\pi_{0*}$ we obtain
\begin{equation}
\label{eq:k-divisible}
K_{X'_0} \sim -2D.
\end{equation}
Thus the divisor $K_{X'}$ is divisible by~2 in $\Cl(X')$.

If $\dim X$ is odd and~\mbox{$P \in X$}, the image of $K_{X'}$ in the local class group
$\Cl(X',P') \cong \mathbb{Z}/2\mathbb{Z}$ is the generator.
This gives a contradiction with~\eqref{eq:k-divisible} and thus shows that $P \not\in X$ when~$\dim X$ is odd, and hence $X' = X'_0$ is smooth.

Since $\pi\colon X \to X'$ is a double cover and $\pi^*K_{X'} \sim 2K_X$,
it follows that $-K_{X'}$ is ample, i.e.~$X'$ is a Fano variety.
By \eqref{equation-faithful-action-conics} we have $\rho(X') = 1$.

Finally, it remains to show that $G = G_0 \cong \mumu_2$.
We have already shown that any
nontrivial element in $G_0$ has order $2$ and acts on $V$ as an involution with eigenspaces of dimension $1$ and \mbox{$\dim V - 1$}.
It follows that
$G\cong \mumu_2^{r}$ for some $r$, and it remains to show that~\mbox{$r=1$}.
Suppose that $r>1$. Take two different involutions $\tau_1,\tau_2\in G$.
The action of the abelian group~\mbox{$G\cong \mumu_2^{r}$} on $V = H^0(X,\CO(-K_X))^\vee$ is diagonalizable.
Thus we may assume that the action of $\tau_1$ (respectively,~$\tau_2$) on $V$ is given by $\operatorname{\diag}(-1,1,\ldots,1)$
(respectively,~\mbox{$\operatorname{\diag}(1,-1,1,\ldots,1)$}).
Then the action of $\tau_1\circ \tau_2$ is given by $\operatorname{\diag}(-1,-1,1,\ldots,1)$.
If~$\dim V > 3$ the dimension of both eigenspaces is greater than 1, which contradicts the above observation.
On the other hand, the case $\dim V = 3$ is impossible, since then~$X \cong \P^2$ and has no curves of degree $2$ with respect to the anticanonical class.
\end{proof}

\subsection{Action on lines}
\label{subsection:act-on-lines}

We start with studying an action of the automorphism group of a smooth Fano threefold $Y$
of Picard rank $1$ and index~$2$ on its Hilbert scheme of lines.
Recall the notation introduced in \S\xref{subsection:lines-and-conics}. In particular, for an irreducible
component~\mbox{$\Sigma_0 \subset \Sigma(Y)$} of the Hilbert
scheme of lines on~$Y$ we denote by~$\CL_0(Y)$ the corresponding component of the universal line, so that we have the diagram~\eqref{diagram-universal-line}
Note that every component of~$\Sigma(Y)$ is generically reduced by Lemma~\ref{lemma-hilb-y-1}.

As we explained earlier, we are interested in proving faithfulness of the action of the automorphism group $\Aut(Y)$ on the Hilbert scheme $\Sigma(Y)$.
In the next lemma we consider an irreducible component $\Sigma_0$ of $\Sigma(Y)$ and the subgroup $\Aut_{\Sigma_0}(Y) \subset \Aut(Y)$ stabilizing~it.
Although for $\dd(Y) \ge 3$ the Hilbert scheme $\Sigma(Y)$ is irreducible by Proposition~\xref{hilb-lines-explicit},
for~\mbox{$\dd(Y) = 2$} the Hilbert scheme $\Sigma(Y)$ might a priori have several components. This is why we formulate the lemma in this form.

\begin{lemma}\label{lemma:Fano-lines}
Let $Y$ be a Fano threefold with $\rho(Y) = 1$, $\iota(Y) = 2$, and
assume that~$-K_Y$ is very ample, i.e.\ $2 \le \dd(Y) \le 5$.
Then the action of $\Aut_{\Sigma_0}(Y)$ on an irreducible component $\Sigma_0$ of the Hilbert scheme of lines on $Y$ is faithful.
In particular, the action of~\mbox{$\Aut(Y)$} on $\Sigma(Y)$ is faithful.
\end{lemma}
\begin{proof}
Let~\mbox{$G \subset \Aut_{\Sigma_0}(Y)$} be the kernel of the action of the group $\Aut_{\Sigma_0}(Y)$ on $\Sigma_0$.
Suppose that $G$ is nontrivial.
Since $-K_Y$ is ample, the group $G$ is a linear algebraic group by Corollary~\xref{corollary:anticanonical-ring}.
By Lemma~\xref{lemma-hilb-y-1} the conditions of Theorem~\xref{theorem:faithful} are satisfied.
Since~\mbox{$\dim Y = 3$} is odd, we conclude that there is a double cover $\pi \colon Y \to Y'$ over a smooth Fano variety $Y'$ with $\rho(Y') = 1$.
Since the branch divisor $B \subset Y'$ is anticanonical, it follows that $K_Y \sim\frac12\pi^*K_{Y'}$, hence
\begin{equation*}
K_Y^3 = \frac18 \cdot 2\cdot K_{Y'}^3 = \frac14 K_{Y'}^3.
\end{equation*}
In particular,
one has
\begin{equation*}
-K_{Y'}^3 = -4K_Y^3 = 32\dd(Y) \ge 64.
\end{equation*}
This is only possible if $Y' \cong \P^3$ (see Table~\xref{table:Fanos-i-ge-2}).
On the other hand, by Theorem~\xref{theorem:faithful} for a line $L$ corresponding to a general point of $\Sigma_0$ we have $L = \pi^*L'$
for a curve $L' \subset Y'$, hence
\begin{equation}\label{eq:2-KY-L}
2 = -K_Y \cdot L = -\frac12 \pi^*K_{Y'} \cdot \pi^*L' = -K_{Y'} \cdot L'.
\end{equation}
The right side of~\eqref{eq:2-KY-L} is divisible by $\iota(Y') = \iota(\P^3) = 4$.
This contradiction shows that~$G$ is trivial, hence the action of $\Aut_{\Sigma_0}(Y)$ on $\Sigma_0$ is faithful.
\end{proof}

\begin{corollary}\label{corollary-aut-index-2}
Let $Y$ be a \textup(smooth\textup) Fano threefold with $\rho(Y) = 1$, $\iota(Y) = 2$, and~\mbox{$\dd(Y)\ge 3$}.
The following assertions hold:
\begin{itemize}
\item[(i)]
if $\dd(Y) = 3$ then
$\Aut(Y)\subset\Aut(S)$
for a smooth minimal surface $S$ of general type with irregularity $5$, geometric genus $10$, and canonical degree~\mbox{$K_{\Sigma(Y)}^2=45$};
\item[(ii)]
if $\dd(Y) = 4$ then
$\Aut(Y)\subset\Aut(S)$
for an abelian surface $S$;
\item[(iii)]
if $\dd(Y) = 5$ then
$\Aut(Y) \subset \Aut(\PP^2) \cong\PGL_3(\Bbbk)$.
\end{itemize}
In particular, for $\dd(Y) = 3$ and $\dd(Y) = 4$, the group $\Aut(Y)$ is finite.
\end{corollary}
\begin{proof}
Assertions (i), (ii) and (iii) follow from Proposition~\xref{hilb-lines-explicit}
and Lemma~\xref{lemma:Fano-lines}.
Finiteness for $\dd(Y)=3$
follows from assertion~(i) and Corollary~\xref{corollary:general-type},
while finiteness for~\mbox{$\dd(Y)=4$} follows from assertion~(ii) and
Corollaries~\xref{corollary:anticanonical-ring}
and~\xref{corollary:nef-K}.
\end{proof}

An alternative proof of finiteness of the automorphism
group $\Aut(Y)$ for $\dd(Y)=3$ and~\mbox{$\dd(Y)=4$}
is by applying Theorems~\xref{theorem:MatsumuraMonsky} and~\xref{theorem:Benoist}.
For~\mbox{$\dd(Y)=5$} one actually has
\begin{equation*}
\Aut(Y)\cong\PGL_2(\Bbbk),
\end{equation*}
see Theorem~\xref{theorem-mu-5} below.

\begin{remark}
Besides its action on $\Sigma(Y)$, the automorphism group $\Aut(Y)$ also acts on the intermediate Jacobian $J(Y)$ of $Y$.
For $\dd(Y) = 3$ one can check that this action is faithful.
Indeed, by~\cite[\S5]{Beauville1982singularities} the intermediate Jacobian $J(Y)$ contains an $\Aut(Y)$-invariant theta divisor $\Theta \subset J(Y)$
which is equal to the image of the canonical Abel--Jacobi map
\begin{equation*}
\Sigma(Y) \times \Sigma(Y) \to J(Y),
\qquad
(L_1,L_2) \mapsto [L_1] - [L_2].
\end{equation*}
Moreover, $\Theta$ has a unique singular point $P \in \Theta$ (the image of the diagonal in $\Sigma(Y) \times \Sigma(Y)$)
and the exceptional divisor of the blow up of $\Theta$ at $P$ is $\Aut(Y)$-equivariantly isomorphic to $Y$.
The faithfulness of the action on $J(Y)$ follows immediately.

For $\dd(Y) = 4$ one can show that the group $\Aut(Y)$ contains a subgroup~$\Gamma$ of order $32$ that acts trivially on the corresponding pencil
of quadrics, and $\Gamma$ contains a subgroup~$\Gamma_0$ of order $16$
that acts trivially on the associated curve $B(Y)$ of genus $2$
mentioned in Remark~\xref{remark:hyperelliptic-curve}, and also on the intermediate Jacobian
of~$Y$.

Finally, for $\dd(Y) = 5$ the intermediate Jacobian of $Y$ is zero (see for instance~\cite[\S12.2]{Iskovskikh-Prokhorov-1999}).
\end{remark}

\subsection{Action on conics}
\label{subsection:act-on-conics}

Now we will analyze the action of the automorphism group of a smooth Fano threefold
$X$ of Picard rank $1$ and index $1$ on the Hilbert scheme $S(X)$ of conics on~$X$. We will
assume that $H \sim -K_X$ is very ample; in particular, this means that~\mbox{$\g(X) \ge 3$}.

As in the case of lines, we are interested in proving faithfulness of the action of the automorphism group $\Aut(X)$ on $S(X)$,
but by the same reason as in \S\xref{subsection:act-on-lines} we consider the action of the subgroup $\Aut_{S_0}(X)$ stabilizing an irreducible component $S_0$ of~$S(X)$.
Note that we know irreducibility of $S(X)$ for $\g(X) \ge 7$ (see Proposition~\xref{proposition:conics}),
but already for~$\g(X) = 6$ the scheme $S(X)$ can be reducible (see Remark~\xref{remark:hilb-conics-reducible}).

We start by discussing some cases
when the action of the subgroup $\Aut_{S_0}(X) \subset \Aut(X)$ on an irreducible component $S_0$ of $S(X)$
is not faithful.

\begin{lemma}\label{lemma:trivial-action-conics}
Let $Y$ be a Fano threefold with $\rho(Y) = 1$, $\iota(Y) = 2$, and~\mbox{$\dd(Y) \ge 2$}.
Let~\mbox{$\pi\colon X \to Y$} be a double cover branched in a smooth anticanonical divisor $B \subset Y$.
Then~$X$ is a smooth Fano threefold with $\rho(X)=1$, $\iota(X) = 1$, \mbox{$\g(X) = \dd(Y) + 1$},
and~$S(X)$ has an irreducible component $S_0$ such that the action of $\Aut_{S_0}(X)$ on it is not faithful.
\end{lemma}
\begin{proof}
By the Hurwitz formula one has
\begin{equation}\label{equationHurwitz}
-K_X\sim \pi^* H_Y,
\end{equation}
where $H_Y$ is the ample generator of the Picard group of $Y$ (so that $-K_Y\sim 2H_Y$).
Hence~$X$ is a (smooth) Fano threefold.
Furthermore, by \cite{Cornalba1981}
the pullback morphism \mbox{$\pi^* \colon H^2(Y,\Z) \to H^2(X,\Z)$} is an isomorphism, hence $\rho(X) = 1$ and $\Pic(X)$ is generated by $K_X$, i.e., $\iota(X) = 1$.
It follows easily that $\g(X) = \dd(Y) + 1$.

For every line $L \subset Y$ its preimage~\mbox{$\pi^{-1}(L) \subset X$} is a conic.
This defines a morphism $\Sigma(Y) \to S(X)$, whose image is a union of components of $S(X)$.
The Galois involution of the double cover is an automorphism of $X$ which acts trivially on
any such component $S_0$ of $S(X)$,
hence is contained in the kernel of $\Aut_{S_0}(X)$-action on $S_0$.
\end{proof}

\begin{example}\label{example:trivial-action-exotic-component}
Assume that $X$ is a quartic threefold with cones and $S_0 \subset S(X)$ is an exotic component (see Lemma~\xref{lemma:exotic-component}).
Then there may be a nontrivial subgroup in $\Aut(X)$ acting trivially on $S_0$. For example, assume that $X \subset \PP^4$ is the Fermat quartic,
consider a cone on $X$ described in Example~\xref{example:quartic-cone} and consider the action of the group $\mumu_4$ on~$\PP^4$
by the primitive character on the first two coordinates, and trivial on the last three coordinates. The equation of $X$ is preserved
by this action, hence $\mumu_4$ acts (faithfully) on~$X$. On the other hand, its action on the base $B$ of the cone is trivial,
hence so is its action on~\mbox{$S_0=\Hilb^2(B)$}.
However, it may act nontrivially on nilpotents (recall that the scheme structure on the component of $S(X)$ underlying $S_0$ is everywhere non-reduced).
\end{example}

In the next lemma we show that the situations of
Lemma~\xref{lemma:trivial-action-conics} and Example~\xref{example:trivial-action-exotic-component} are the only ones
when the action of the automorphism group $\Aut_{S_0}(X)$ on an irreducible component $S_0$ of $S(X)$ is not faithful,
at least in the case of a very ample~$-K_X$.

\begin{lemma}\label{lemma:Fano-conics}
Let $X$ be a Fano threefold with $\rho(X)=1$ and $\iota(X)=1$ such that~\mbox{$-K_X$} is very ample,
i.e. either $\g(X)\ge 4$, or $X$ is a quartic threefold.
If the action of $\Aut_{S_0}(X)$ on a non-exotic irreducible component $S_0$ of $S(X)$ is not faithful,
then~$X$ is a double cover of a smooth Fano threefold $Y$ with $\rho(Y) = 1$, $\iota(Y) = 2$,
and~\mbox{$\dd(Y)=\g(X)-1\ge 2$},
the irreducible component $S_0$ comes from $\Sigma(Y)$ as in Lemma~\xref{lemma:trivial-action-conics},
and the kernel of the action of $\Aut_{S_0}(X)$ on $S_0$ is generated by the Galois involution of the double cover.

In particular, for $\g(X) \ge 7$ the action of $\Aut(X)$ on $S(X)$ is faithful.
\end{lemma}

\begin{proof}
Let~\mbox{$G \subset \Aut_{S_0}(X)$} be the kernel of the action of the group $\Aut_{S_0}(X)$ on $S_0$.
Suppose that $G$ is nontrivial.
The group $G$ is a linear algebraic group by Corollary~\xref{corollary:anticanonical-ring}.

If~$S_0$ is not exotic then by Lemma~\xref{lemma-hilb-x-2} the conditions of Theorem~\xref{theorem:faithful} are satisfied.
Since~\mbox{$\dim X = 3$} is odd, we conclude that there is a double cover $\pi \colon X \to Y$ over a smooth Fano
variety~$Y$ with $\rho(Y) = 1$ (which corresponds to the variety $X'$ of Theorem~\xref{theorem:faithful}).
The index of $Y$ is even and $Y \not\cong \P^3$ (by the same reason as in the proof of Lemma~\xref{lemma:Fano-lines},
where we had a similar situation with the threefold $Y'$ instead of~$Y$),
hence~\mbox{$\iota(Y) = 2$}.
Since the branch divisor $B \subset Y$ is anticanonical, it follows
that~\mbox{$-K_X \sim \pi^*H_Y$}, where~$H_Y$ is the ample generator of the Picard group of $Y$ (so that $-K_Y\sim 2H_Y$).
Hence
\begin{equation*}
K_X^3 = \frac18 \cdot 2\cdot K_Y^3 = \frac14 K_Y^3,\quad \dd(Y) = -\frac18K_Y^3 = -\frac12K_X^3 = \g(X) - 1.
\end{equation*}
Since $\g(X) \ge 3$ it follows that $\dd(Y) \ge 2$.
On the other hand, since $\dd(Y) \le 5$, we have~\mbox{$\g(X) \le 6$}, so for $\g(X) \ge 7$ the action is faithful.
\end{proof}

We think that for $\g(X) \le 6$ the action of $\Aut(X)$ on $S(X)$ is still faithful.

\begin{corollary}\label{corollary:high-genus-Fanos-Aut}
Let $X$ be a \textup(smooth\textup) Fano threefold with $\rho(X) = 1$, $\iota(X) = 1$, and
genus $\g(X)\ge 7$.
The following assertions hold:
\begin{itemize}
\item[(i)]
if $\g(X) = 7$ then $\Aut(X) \subset \Aut(C)$ for a smooth irreducible curve $C$ of genus~$7$;
\item[(ii)]
if $\g(X) = 8$ then
$\Aut(X)\subset\Aut(S)$ for a minimal surface of general type
with irregularity $5$, geometric genus $10$, and canonical
degree~\mbox{$K_{S(X)}^2=45$};
\item[(iii)]
if $\g(X) = 9$ then
$\Aut(X) \subset \Aut(S)$ for a surface $S$ isomorphic to a projectivization of a simple rank $2$ vector bundle on a smooth irreducible curve of genus~$3$;
\item[(iv)]
if $\g(X) = 10$ then
$\Aut(X)\subset\Aut(S)$ for an abelian surface $S$;
\item[(v)]
if $\g(X) = 12$ then
$\Aut(X) \subset \Aut(\PP^2) \cong \PGL_3(\Bbbk)$.
\end{itemize}
In particular, if $7 \le \g(X) \le 10$, then $\Aut(X)$ is finite.
\end{corollary}
\begin{proof}
By Lemma~\xref{lemma:Fano-conics} the action of $\Aut(X)$ on the Hilbert scheme
$S(X)$ of conics on~$X$ is faithful. Thus assertions~(ii), (iii), (iv) and~(v)
are implied by assertions~(ii), (iii), (iv) and~(v) of Proposition~\xref{proposition:conics},
respectively. In case of assertion~(i) we also take into account an isomorphism
$\Aut\big(\Sym^2(C)\big)\cong\Aut(C)$,
see~\cite{Ran1986}.

Keeping in mind Corollary~\xref{corollary:general-type}, we see that finiteness
for $\g(X)=7$ and~$8$ is implied by assertions~(i) and~(ii)
respectively. If $\g(X)=9$, finiteness is implied by Lemma~\xref{lemma:projectivization-stable-bundle}
and assertion~(iii).
Finally, keeping in mind Corollary~\xref{corollary:anticanonical-ring} and
Lemma~\xref{lemma:non-uniruled} we see that finiteness for $\g(X)=10$ is implied by assertion~(iv).
\end{proof}

As we will see in~\S\xref{section:infinite} some Fano threefolds of index 1 with
$\g(X) = 12$ actually have an infinite automorphism group.

\subsection{Small degree and genus}
\label{subsection:finite-Fano}

As we have shown in Corollaries~\xref{corollary-aut-index-2}
and~\xref{corollary:high-genus-Fanos-Aut}, for Fano threefolds of index 2 and degree $3 \le \dd(Y) \le 4$, and
for Fano threefolds of index 1 and genus~\mbox{$7 \le \g(X) \le 10$} the automorphism groups are always finite.
In this subsection we show the same for smaller values of degree and genus.

We start with the cases when $H$ is not very ample.

\begin{lemma}\label{lemma:double-cover-finite}
Let $X$ be a Fano threefold of index $1$ or $2$ with $\rho(X) = 1$. If the ample generator~$H$ of~\mbox{$\Pic(X)$}
is not very ample then the group $\Aut(X)$ is finite.
\end{lemma}
\begin{proof}
According to
Tables~\xref{table:Fanos-i-ge-2}
and~\xref{table:Fanos-i-1} all such varieties are double covers
\begin{equation*}
\varphi \colon X \to X',
\end{equation*}
where $X'$ is a Fano threefold with $\Cl(X')=\mathbb{Z}\cdot H'$ for some ample divisor $H'$ on $X'$.
In Table~\xref{table:double-covers} we list all possible situations.
The first column of Table~\xref{table:double-covers} lists the invariants of~$X$.
The second column describes $X'$; here
$\P(1,1,1,2)$ is the weighted projective space and~$Q$ is a smooth three-dimensional quadric.
The third column specifies the class in~\mbox{$\Pic(X')$} of the branch divisor $B \subset X'$ of the double cover $\varphi \colon X \to X'$
(note that in the case~\mbox{$X' = \P(1,1,1,2)$} the double cover is also branched over the singular point of $X'$).
\begin{table}[h]
\caption{Double covers}\label{table:double-covers}
\begin{tabular}{|c|c|c|}
\hline
invariants of $X$ & $X'$ & $B$ \\
\hline
$\iota(X) = 2,\ \dd(X) = 1$ &
$\P(1,1,1,2)$ & $6H'$ \\
\hline
$\iota(X) = 2,\ \dd(X) = 2$ & $\PP^3$ & $4H'$ \\
\hline
$\iota(X) = 1,\ \g(X) = 2$ & $\PP^3$ & $6H'$ \\
\hline
$\iota(X) = 1,\ \g(X) = 3$ & $Q$ & $4H'$ \\
\hline
\end{tabular}
\end{table}

The map $\varphi$ is anticanonical in all cases except when $\iota(X)=2$ and $\dd(X)=2$;
in the latter case it is defined by the complete linear system of the divisor $H$ such that~$2H\sim -K_X$.
In particular, $\varphi$ is equivariant with respect to the whole automorphism group $\Aut(X)$, hence we have a natural map
\begin{equation}\label{eq:AutXB}
\Aut(X) \to \Aut(X'; B)
\end{equation}
into the group of automorphisms of $X'$ preserving $B$.
Moreover, we have
\begin{equation*}
X \cong \operatorname{Spec}_{X'}\left(\CO_{X'} \oplus \CO_{X'}\left(-\textstyle{\frac12}B\right)\right),
\end{equation*}
where $\CO_{X'}\left(-\textstyle{\frac12}B\right)$ is the reflexive sheaf corresponding to the Weil divisor class $-\textstyle{\frac12}B$,
and the algebra structure is determined by the composition
\begin{equation*}
\CO_{X'}\left(-\textstyle{\frac12}B\right) \otimes \CO_{X'}\left(-\textstyle{\frac12}B\right) \xrightarrow{\ \ \ }
\CO_{X'}(-B) \xrightarrow{\ B\ } \CO_{X'}
\end{equation*}
with the canonical first map and with the second map given by the divisor $B$.
In particular, every automorphism of $X'$ that fixes $B$ induces an automorphism of $X$, hence
the map~\eqref{eq:AutXB}
is surjective.
Its kernel is clearly generated by the Galois involution of the double cover~$\varphi$,
hence is isomorphic to $\mumu_2$.

On the other hand, it is clear from Table~\ref{table:double-covers} that the divisor $2H'$ is very ample in all cases, so that
$X' \subset \P(V)$, where $V = H^0(X,\CO_{X'}(2H'))^\vee$, and $\Aut(X') \subset \PGL(V)$.
Furthermore,
$B$ is not contained in a hyperplane in $\P(V)$, hence
the natural map
\begin{equation*}
\Aut(X'; B) \to \Aut(B; [2H'])
\end{equation*}
into the group of automorphisms of $B$ preserving (the class in $\Pic(B)$ of) the restriction of $2H'$ to $B$
is injective.
Thus we have an exact sequence
\begin{equation*}
1 \to \mumu_2 \to \Aut(X) \to \Aut(B; [2H']).
\end{equation*}
It remains to notice that $B$ is smooth (as the fixed locus of an involution on a smooth variety $X$) and its canonical bundle
is nef by adjunction formula, hence $\Aut(B; [2H'])$ is finite
by Lemma~\xref{lemma:anticanonical-ring0} and Corollary~\xref{corollary:nef-K}.
\end{proof}

Now we will combine the above results with Theorems~\xref{theorem:MatsumuraMonsky} and Theorem~\xref{theorem:Benoist}.

\begin{proposition}\label{proposition:nearly-done}
If $Y$ is a \textup(smooth\textup) Fano threefold with $\rho(Y) = 1$, $\iota(Y) = 2$, and~\mbox{$\dd(Y) \le 4$}
then the group $\Aut(Y)$ is finite.
If $X$ is a \textup(smooth\textup) Fano threefold with~$\rho(X) = 1$, $\iota(X) = 1$, and~\mbox{$\g(X) \le 10$}
then the group $\Aut(X)$ is finite.
\end{proposition}
\begin{proof}
First, let $Y$ be a Fano threefold with $\rho(Y) = 1$ and $\iota(Y) = 2$.
If~\mbox{$\dd(Y) \in \{1,2\}$}, then
the ample generator of $\Pic(Y)$ is not very ample, and
we apply Lemma~\xref{lemma:double-cover-finite}.
If~\mbox{$\dd(Y) \in \{3,4\}$}, then we apply Corollary~\xref{corollary-aut-index-2}.

Second, let $X$ be a Fano threefold with $\rho(X) = 1$ and $\iota(X) = 1$.
If $\g(X) = 2$ or~\mbox{$\g(X) = 3$} and~$-K_X$ is not very ample,
we apply Lemma~\xref{lemma:double-cover-finite}.
If $\g(X) = 3$ and~$-K_X$ is very ample then $X$ is a quartic in $\P^4$
and we apply Theorem~\xref{theorem:MatsumuraMonsky}.
If~\mbox{$\g(X)=4$} or~\mbox{$\g(X)=5$}, then $X$ is a complete intersection in a projective space
of multidegree~\mbox{$(2,3)$} and~\mbox{$(2,2,2)$} respectively,
and we apply Theorem~\xref{theorem:Benoist}.
If $\g(X) = 6$ we refer to~\cite[Proposition~3.21(c)]{Debarre-Kuznetsov2015}.
Finally, if~\mbox{$7 \le \g(X) \le 10$}, we apply Corollary~\xref{corollary:high-genus-Fanos-Aut}.
\end{proof}

To complete the proof of Theorem~\xref{theorem:Prokhorov}, we need to describe the automorphism groups of Fano threefolds of index 2 and degree 5,
and of index 1 and genus 12. This is done in the next section.

\section{Infinite automorphism groups}
\label{section:infinite}

We already know from Proposition~\xref{proposition:nearly-done} that the only Fano threefolds of Picard rank~$1$ and index 1 or 2
which can have infinite automorphism groups are the threefold $Y$ with~$\iota(Y) = 2$ and $\dd(Y) = 5$
(such threefold is actually unique up to isomorphism, see \cite[Theorem~II.1.1]{Iskovskikh-1980-Anticanonical} or \cite[3.3.1--3.3.2]{Iskovskikh-Prokhorov-1999}),
and some of the threefolds $X$ with~$\iota(X) = 1$ and $\g(X) = 12$.

\subsection{Fano threefolds of index 2 and degree 5}
\label{subsection:v5}

We start with a detailed description of the Fano threefold $Y$ with $\iota(Y) = 2$ and $\dd(Y) = 5$.
From the classification it is known that $Y$ is isomorphic to a linear section of the Grassmannian $\Gr(2,5)\subset \P^9$
by a subspace~\mbox{$\P^6\subset\P^9$} (see Table~\xref{table:Fanos-i-ge-2}).
For our purposes, however, the description of $Y$ suggested by Mukai and Umemura~\cite{Mukai-Umemura-1983} is more convenient.

Let
\begin{equation*}
M_d=\Sym^d (\Bbbk^{2})^\vee
\end{equation*}
be the space of binary forms of degree $d$.
We denote by $x$ and $y$ the elements of the standard basis of the vector space~\mbox{$(\Bbbk^2)^\vee$},
so that elements of $M_d$ are polynomials of degree~$d$ in~$x$ and~$y$.
The group $\GL_2(\Bbbk)$ acts naturally on the space $M_d$ by the rule
\begin{equation*}
\left(\begin{smallmatrix} a & b \\ c & d \end{smallmatrix}\right) \colon
x \mapsto ax + cy,
\qquad
y \mapsto bx + dy,
\end{equation*}
and induces an action of $\PGL_2(\Bbbk)$ on the projective space $\P(M_d)$. Consider the form
\begin{equation*}
\phi_6(x,y)=xy(x^4-y^4)\in M_6
\end{equation*}
and the corresponding point $[\phi_6]\in\P(M_6)\cong\P^6$.

\begin{theorem}[{\cite{Mukai-Umemura-1983}, see also~\cite[Proposition~7.1.10]{CheltsovShramov2016}}]
\label{theorem-mu-5}
The stabilizer of $[\phi_6]$ is the octahedral group
$$\Oct\cong\SS_4 \subset \PGL_2(\Bbbk),$$
and the closure of its orbit
\begin{equation*}
Y= \overline{\PGL_2(\Bbbk)\cdot [\phi_6]}\subset\P^6
\end{equation*}
is the smooth Fano threefold with $\rho(Y)=1$, $\iota(Y)=2$, and $\dd(Y)=5$
embedded by the ample generator of $\Pic(Y)$.
The automorphism group of $Y$ is $\Aut(Y) \cong \PGL_2(\Bbbk)$.
\end{theorem}

We will need a description of the $\PGL_2(\Bbbk)$-orbits on $Y$ (see \cite[Lemma 1.5]{Mukai-Umemura-1983}).
For this we need to introduce notation for the standard connected subgroups in $\PGL_2(\Bbbk)$.
We denote by
\begin{itemize}
\item $B_2 \subset \PGL_2(\Bbbk)$ the standard Borel subgroup (upper triangular matrices),
\item $U_2 \subset \PGL_2(\Bbbk)$ the standard unipotent subgroup (upper triangular matrices with units on the diagonal), and
\item $T_2 \subset \PGL_2(\Bbbk)$ the standard torus (diagonal matrices).
\end{itemize}

The orbit decomposition of $Y$ is
\begin{equation*}
Y = \Orb_3(Y) \sqcup \Orb_2(Y) \sqcup \Orb_1(Y)
\end{equation*}
with $\Orb_k(Y)$ standing for the unique $\PGL_2(\Bbbk)$-orbit of dimension $k$; explicitly
\begin{alignat*}{2}
\Orb_3(Y) & = \PGL_2(\Bbbk)\cdot [\phi_6] && \cong \PGL_2(\Bbbk)/\Oct,\\
\Orb_2(Y) & = \PGL_2(\Bbbk)\cdot [xy^5] && \cong \PGL_2(\Bbbk)/T_2,\\
\Orb_1(Y) & = \PGL_2(\Bbbk)\cdot [x^6] && \cong \PGL_2(\Bbbk)/B_2.
\end{alignat*}
It is clear from this description that
$\Orb_{1}(Y)$ is a normal rational sextic curve, and that
\begin{equation*}
\oOrb_2(Y) = \Orb_2(Y) \sqcup \Orb_1(Y)
\end{equation*}
is the image of $\PP^1 \times \PP^1 = \PP(M_1) \times \PP(M_1)$ under the map
\begin{equation*}
\nu \colon \PP(M_1) \times \PP(M_1) \xrightarrow{\qquad} \PP(M_6),
\qquad (f,g) \mapsto f^5g.
\end{equation*}
Geometrically, $\oOrb_2(Y)$ is the tangential scroll of $\Orb_1(Y)$, i.e., the surface swept by tangent lines to the twisted sextic curve $\Orb_1(Y)$,
and $\nu$ is its normalization morphism (for more details, see for instance~\cite[Lemma~7.2.2]{CheltsovShramov2016}).

\begin{remark}
\label{remark:curves-oorb2}
It follows from the above description that any irreducible curve of degree at most~5 contained in $\oOrb_2(Y)$ is
either a line~$\nu(\{f\} \times \PP(M_1))$ (as we show below these are special lines on~$Y$, as defined in~\S\xref{subsection:lines-and-conics}),
or is the image of the normal rational quintic
\begin{equation}\label{eq:z:mu}
Z_{\mathrm{MU}} = \nu(\PP(M_1) \times \{x\}) = \overline{B_2 \cdot [xy^5] } \subset Y
\end{equation}
under the $\PGL_2(\Bbbk)$-action.
The reason for the notation in~\eqref{eq:z:mu} will become clear later.
And meanwhile, just note that $Z_{\mathrm{MU}}$ is preserved by the Borel subgroup $B_2 \subset \PGL_2(\Bbbk)$.
\end{remark}

Recall that by Proposition~\ref{hilb-lines-explicit}(iii) the Hilbert scheme of lines $\Sigma(Y)$ is isomorphic to~$\P^2$.
In fact, we have a $\PGL_2(\Bbbk)$-equivariant isomorphism
\begin{equation*}
\Sigma(Y) \cong \P(M_2)
\end{equation*}%
see~\cite[Theorem~I]{Furushima1989a} or \cite[Proposition~2.20]{Sanna}.
Below we describe explicitly lines on~$Y$ corresponding to points of $\P(M_2)$.
Note that any pair of points $f,g \in \P(M_1)$ gives a point $fg \in \P(M_2)$,
and, if $f \ne g$, a point $fg(f^4 - g^4) \in \Orb_3(Y) \subset Y$.

\begin{lemma}
\label{lemma:lines-v5-explicit}
Every line on $Y$ can be written in one of the following two forms:
\begin{equation*}
\begin{aligned}
L_{fg} &= \{ fg(s_1f^4 - s_2g^4) \in Y \}_{(s_1:s_2) \in \P^1}, \qquad&& \text{for $f,g\in\P(M_1)$, $f \ne g$},\\
L_{f^2} &= \{ f^5(s_1x + s_2y) \}_{(s_1:s_2) \in \P^1}, && \text{for $f\in\P(M_1)$}.
\end{aligned}
\end{equation*}
\end{lemma}
\begin{proof}
It is clear that both $L_{fg}$ and $L_{f^2}$ are lines on $Y$.
The first intersects $\oOrb_2(Y)$ at two points $fg^5$ and $f^5g$, while the second is contained in $\oOrb_2(Y)$.
Thus, they correspond to points of different $\PGL_2(\Bbbk)$-orbits in $\Sigma(Y)$.
It remains to recall that $\Sigma(Y) \cong \P(M_2)$
contains just two $\PGL_2(\Bbbk)$-orbits, and to notice that the images of the lines $L_{fg}$ and $L_{f^2}$
under the action of $\PGL_2(\Bbbk)$ are lines of the same form.
\end{proof}

From the description of Lemma~\ref{lemma:lines-v5-explicit} it is easy to obtain the following result.
Recall that $\CL(Y)$ denotes the universal line on $Y$, and $q \colon \CL(Y) \to \Sigma(Y)$, $p \colon \CL(Y) \to Y$ denote its natural projections.

\begin{corollary}[{cf. \cite[1.2.1(3)]{Iliev1994a} and \cite[Corollary~2.24]{Sanna}}]
\label{corollary:V5-lines-description}
The set $q(p^{-1}([\varphi]))$ of lines on $Y$ passing through a point $[\varphi] \in Y$
can be described as follows:
\begin{equation*}
q(p^{-1}([\varphi])) =
\begin{cases}
\{[fg], [f^2 - g^2], [f^2 + g^2]\}, & \text{if $\varphi = fg(f^4 - g^4) \in \Orb_3(Y)$,}\\
\{[fg], [f^2]\}, & \text{if $\varphi = f^5g \in \Orb_2(Y)$,}\\
\{[f^2]\}, & \text{if $\varphi = f^6 \in \Orb_1(Y)$,}\\
\end{cases}
\end{equation*}
The ramification divisor of the map $p \colon \CL(Y) \to Y$ is the union of lines $L_{f^2}$ for $f \in \P(M_1)$.
In other words, any line $L_{fg}$ with $f \ne g$ is ordinary and any line $L_{f^2}$ is special.
\end{corollary}

The three points $[fg]$, $[f^2 - g^2]$, and $[f^2 + g^2]$ parameterizing three lines through a general point of $Y$
correspond to the three axes of an octahedron.

\subsection{Fano threefolds of index 1 and genus 12}
\label{subsection:v22}

There is a similar example of a Fano threefold $X$ with $\rho(X) = 1$, $\iota(X) = 1$, and~\mbox{$\g(X) = 12$},
that was also found by Mukai and Umemura.
Consider the form
\begin{equation*}
\phi_{12}(x,y)=xy(x^{10} +11 x^5y^5+y^{10}) \in M_{12}
\end{equation*}
and the point
\begin{equation*}
\upsilon= (\phi_{12}, 1) \in \PP(M_{12}\oplus M_0) \cong \PP^{13}.
\end{equation*}

\begin{theorem}[\cite{Mukai-Umemura-1983}]
\label{theorem-mu-22}
The stabilizer of $[\upsilon]$ is the icosahedral group
$$\Icos\cong\A_5 \subset \PGL_2(\Bbbk),$$
and the closure of its orbit
\begin{equation*}
X^{\mathrm{MU}} = \overline{\PGL_2(\Bbbk)\cdot [\upsilon]}\subset \P^{13}
\end{equation*}
is a smooth anticanonically embedded Fano threefold with $\rho(X^{\mathrm{MU}})=1$, $\iota(X^{\mathrm{MU}})=1$, and~\mbox{$\g(X^{\mathrm{MU}})=12$}.
The automorphism group of $X^{\mathrm{MU}}$ is $\Aut(X^{\mathrm{MU}}) \cong \PGL_2(\Bbbk)$.
\end{theorem}

Note, however, that $X^{\mathrm{MU}}$ is just a single variety from a six-dimensional family of Fano threefolds of this type.
One of descriptions of other Fano threefolds of index 1 and genus~12 is based on the double projection method.

\begin{theorem}[{\cite{Iskovskikh1989}}, {\cite{Prokhorov-1992b}}, {\cite[Theorem 4.3.7]{Iskovskikh-Prokhorov-1999}}]
\label{theorem:double-projection}
The following assertions hold:
\begin{itemize}
\item[(i)]
Let $X$ be a smooth Fano threefold with $\rho(X) = 1$, $\iota(X) = 1$, and $\g(X) = 12$, and let $L \subset X$ be a line.
Then the linear system $|H_X-2L|$, where $H_X$ is the ample generator of $\Pic(X)$, defines a birational map of $X$
onto the smooth Fano threefold~$Y$ with $\rho(Y)=1$, $\iota(Y)=2$, and $\dd(Y)=5$.
\item[(ii)]
Let $Y$ be the smooth Fano threefold with $\rho(Y) = 1$, $\iota(Y)=2$, and $\dd(Y)=5$, and let $Z\subset Y \subset \P^6$ be a normal rational quintic curve.
Then the linear system~\mbox{$|3H_Y-2Z|$}, where $H_Y$ is the ample generator of $\Pic(Y)$, defines a
birational map of~$Y$ onto a smooth Fano threefold $X$ with $\rho(X)=1$, $\iota(X)=1$, and~\mbox{$\g(X)=12$}.
\end{itemize}
The constructions of~{\rm(i)} and~{\rm(ii)} are mutually inverse, and the corresponding birational transformation between $X$ and $Y$ can be described by a diagram
\begin{equation}\label{equation(1)}
\vcenter{
\xymatrix{
&&
X' \ar[dl]_{\sigma_X} \ar@{-->}[rr]^-{\upchi} && Y' \ar[dr]^{\sigma_Y}
\\
L \ar@{^{(}->}[r] &
X \ar@{-->}[rrrr]^-\xi &&&& Y &
Z \ar@{_{(}->}[l]
}
}
\end{equation}
where the morphism $\sigma_X$ is the blow up of $L$, the morphism $\sigma_Y$ is the blow up of $Z$,
and the upper dashed arrow $\upchi$ is a flop.
\end{theorem}

\begin{remark}\label{remark:contracted-divisors}
In the above diagram, the map $\xi \colon X\dashrightarrow Y$ contracts a divisor which is a unique member of the linear system $|H_X-3L|$.
Similarly, the map $\xi^{-1} \colon Y\dashrightarrow X$ contracts a divisor which is a unique member of the linear system $|H_Y-Z|$.
\end{remark}

We denote by $E_L \subset X'$ and $E_Z \subset Y'$ the exceptional divisors of the blowups $\sigma_X$ and~$\sigma_Y$.

\begin{lemma}
\label{lemma:flopping-locus}
The flopping locus of the map $\upchi$ is the union of strict transforms of lines on $X$ intersecting $L$,
and of the exceptional section of the divisor $E_L$ if the line $L$ is special.
The flopping locus of the map $\upchi^{-1}$ is the union of strict transforms of bisecants of $Z$ on $Y$.
\end{lemma}

\begin{proof}
The first assertion can be found in~\cite[Proposition~4.3.1]{Iskovskikh-Prokhorov-1999}.
For the second assume that $C \subset Y'$ is a flopping curve of $\upchi^{-1}$ and let $C_X \subset X'$ be the corresponding flopped curve.
Then either $\sigma_X(C_X)$ is a line meeting $L$ or $C_X$ is the exceptional section of $E_L$.
Therefore, one has
\begin{equation*}
(\sigma_X^* H_X-2 E_L)\cdot C_X= -1.
\end{equation*}
By the construction of flops \cite{Kollar-flops} we have $\sigma_Y^* H_Y \cdot C = 1$.
Therefore, $\sigma_Y(C)$ is a line on~$Y$.
Since $K_{Y'}\cdot C = 0$, it is a bisecant of $Z$.
\end{proof}

One can also show that the flopping curves of the map $\upchi$ are disjoint and have normal bundles
of the form $\O_{\P^1}(-1)\oplus\O_{\P^1}(-1)$ or $\O_{\P^1}\oplus\O_{\P^1}(-2)$, see \cite{Cutkosky1989},
and therefore, near each flopping curve, the flop $\upchi$ is given by Reid's pagoda \cite{Reid1983}.

\begin{remark}
\label{remark:functoriality}
The construction of Theorem~\ref{theorem:double-projection} is functorial: an isomorphism between pairs $(X_1,L_1)$ and $(X_2,L_2)$ induces
an isomorphism of the associated diagrams~\eqref{equation(1)}, and hence an isomorphism of the corresponding pairs $(Y,Z_1)$ and $(Y,Z_2)$.
Conversely, an isomorphism of pairs $(Y,Z_1)$ and $(Y,Z_2)$ induces in the same way an isomorphism of pairs $(X_1,L_1)$ and $(X_2,L_2)$ associated with them.
In particular, if the pair $(Y,Z)$ corresponds to a pair $(X,L)$ then
\begin{equation}
\label{eq:aut-xl-yz}
\Aut(X; L) \cong \Aut(Y; Z),
\end{equation}
where $\Aut(X; L) \subset \Aut(X)$ and $\Aut(Y; Z) \subset \Aut(Y)$ are the subgroups preserving $L$ and~$Z$ respectively.
In particular, if $G \subset \Aut(X)$ and $L$ is $G$-invariant, then $G \subset \PGL_2(\Bbbk)$ and $Z$ is $G$-invariant.
Conversely, if $Z$ is stabilized by a subgroup $G \subset \PGL_2(\Bbbk)$ then $G$ acts faithfully on $X$ and preserves the line $L$.
\end{remark}

Denote by $\Sigma^0_L(X) \subset \Sigma(X)$ the open subscheme of the Hilbert scheme of lines on $X$ that parameterizes lines
which intersect neither $L$, nor any other line intersecting $L$.
Similarly, denote by
\begin{equation*}
\Sigma_Z(Y) = q(p^{-1}(Z)) \subset \Sigma(Y)
\end{equation*}
the closed subscheme of $\Sigma(Y)$ parameterizing lines intersecting a normal rational quintic~$Z$, and let
$\Sigma^0_Z(Y) \subset \Sigma_Z(Y)$ be its open subscheme that parameterizes lines which are
neither bisecants of $Z$, nor intersect any bisecant of $Z$.

\begin{lemma}
\label{lemma:sigma-0-iso}
The scheme $\Sigma^0_L(X)$ is an open dense subscheme of $\Sigma(X)$, and
the map $L' \mapsto \xi(L')$ is a rational map $\Sigma(X) \dashrightarrow \Sigma_Z(Y)$
inducing an isomorphism $\Sigma^0_L(X) \cong \Sigma^0_Z(Y)$.
\end{lemma}
\begin{proof}
By Lemma~\ref{lemma:flopping-locus}
any line intersecting a given line $L_0$ on $X$ is a flopping line for the double projection from $L_0$,
hence the number of such lines is finite.
Since any component of $\Sigma(X)$ is one-dimensional (see Lemma~\ref{lemma-hilb-x-1}),
it follows that $\Sigma^0_L(X) \subset \Sigma(X)$ is dense.

If $L'$ corresponds to a point of $\Sigma^0_L(X)$, the map $\xi$ is regular on $L'$.
Since
\begin{equation*}
H_Y \cdot \xi(L') = (H_X - 2E_L) \cdot L' = 1
\qquad\text{and}\qquad
E_Z \cdot \xi(L') = (H_X - 3E_L) \cdot L' = 1,
\end{equation*}
it follows that $\xi(L')$ is a line, and intersects $Z$ at one point.
Moreover, $\xi(L')$ does not intersect bisecants of~$Z$, since $L'$ does not intersect flopping lines.
Hence $\xi(L')$ corresponds to a point of $\Sigma^0_Z(Y)$.
Thus, the map $\xi$ is well defined on an open subscheme $\Sigma^0_L(X)$ as a map $\Sigma^0_L(X) \to \Sigma^0_Z(Y)$.

Conversely, if $L'$ corresponds to a point of $\Sigma^0_Z(Y)$, the map $\xi^{-1}$ is regular on $L'$,
and a computation similar to the above shows that $\xi^{-1}(L')$ is a line on $X$.
This defines a morphism $\Sigma^0_Z(Y) \to \Sigma_L^0(X)$.
The two morphisms are evidently mutually inverse.
\end{proof}

In the above lemma we do not claim that $\Sigma^0_Z(Y)$ is dense in $\Sigma_Z(Y)$.
In fact, $\Sigma_Z(Y)$ can have a component consisting of lines
meeting both $Z$ and a bisecant of $Z$ (cf.\ Lemma~\ref{lemma:sigma-z-y})
and then $\Sigma^0_Z(Y)$ is contained in the complement of this component.

\begin{corollary}
\label{corollary:sigma-x-gorenstein}
The Hilbert scheme of lines $\Sigma(X)$
is a Gorenstein curve.
\end{corollary}
\begin{proof}
As it was mentioned in the proof of Lemma~\ref{lemma:sigma-0-iso}, the number of lines in $X$ intersecting a given line $L$ is finite.
This means that the open subschemes $\Sigma^0_L(X)$ form an open covering of $\Sigma(X)$.
So, by Lemma~\ref{lemma:sigma-0-iso} it is enough to prove that $\Sigma^0_Z(Y)$ is Gorenstein.

Since $Z$ is a smooth curve, it is a locally complete intersection in $Y$.
Since the map $p\colon \CL(Y) \to Y$ is finite, the scheme $p^{-1}(Z) \subset \CL(Y)$ is also a locally complete intersection,
and since $\CL(Y)$ is smooth, we conclude that $p^{-1}(Z)$ is a Gorenstein curve.
It remains to notice that the map $q\colon p^{-1}(Z) \to \Sigma_Z(Y)$ is an isomorphism over $\Sigma^0_Z(Y)$, hence the latter is also Gorenstein.
\end{proof}

The above argument also shows that the curve $\Sigma(X)$ has only planar singularities.

\subsection{Special Fano threefolds of genus $12$}
\label{subsection:v22-special}

In this section we construct some examples of Fano threefolds $X$ of genus 12 with infinite automorphism groups, and after that
we show that all $X$ with infinite automorphism groups are covered be these examples.

By Remark~\ref{remark:functoriality} to produce an example of such $X$,
it is enough to find a normal rational quintic $Z$ stabilized by an infinite subgroup of $\Aut(Y) \cong \PGL_2(\Bbbk)$.
Recall the notation for subgroups $B_2$, $U_2$, and $T_2$ of $\PGL_2(\Bbbk)$ introduced in~\S\ref{subsection:v5}.

\begin{example}
\label{example-MU}
Let $Z = Z_{\mathrm{MU}} \subset Y$ be the quintic of Remark~\ref{remark:curves-oorb2}.
The corresponding Fano threefold of index 1 and genus 12 has a faithful action of the subgroup $B_2 \subset \PGL_2(\Bbbk)$.
In Theorem~\ref{theorem:Prokhorov-g-12-finite} we prove it is the Mukai--Umemura threefold $X^{\mathrm{MU}}$ of Theorem~\ref{theorem-mu-22}.
\end{example}

\begin{example}[{\cite{Prokhorov-1990c}}] \label{example-V22-a}
The curve
\begin{equation}\label{eq:z:a}
Z_{\mathrm{a}} = \overline{ U_2 \cdot [\phi_6]}\subset Y\subset\P(M_6)
\end{equation}
is a normal rational quintic curve preserved by the subgroup $U_2 \subset \PGL_2(\Bbbk)$.
We have
\begin{equation*}
Z_{\mathrm{a}} \cap \Orb_3(Y) = U_2 \cdot [\phi_6] \cong \mathbb{A}^1,
\qquad
Z_{\mathrm{a}} \cap \Orb_2(Y) = \varnothing,
\qquad
Z_{\mathrm{a}} \cap \Orb_1(Y) = [x^6].
\end{equation*}
We denote by $X^{\mathrm{a}}$ the Fano threefold of index 1 and genus 12
corresponding to the quintic~$Z_{\mathrm{a}}$ via the construction of Theorem~\xref{theorem:double-projection}.
\end{example}

\begin{example}[{\cite{Prokhorov-1990c}}]\label{example-V22-m}
For every parameter $u \in \Bbbk$ put
\begin{equation}
\label{eq:normal-quintic}
\phi_{6,u}(x,y) =
\left(\begin{smallmatrix}
1 & u \\ 0 & 1
\end{smallmatrix}\right)
\cdot \phi_6 =
x(ux+y)(x^4 - (ux+y)^4).
\end{equation}
Clearly, one has
\begin{equation*}
[\phi_{6,u}] \in U_2 \cdot[\phi_6] \subset \Orb_3(Y) \subset Y.
\end{equation*}
Expanding the right side of~\eqref{eq:normal-quintic} we get
\begin{equation*}
\phi_{6,u}(x,y) = u(1-u^4)x^6 + (1 - 5u^4)x^5y - 10u^3x^4y^2 - 10u^2x^3y^3 - 5ux^2y^4 - xy^5.
\end{equation*}
If all the coefficients of this polynomial are non-zero, i.e., if
\begin{equation}\label{eq:uu15u1}
u(u^4-1)(5u^4-1) \ne 0,
\end{equation}
the closure of the $T_2$-orbit of $\phi_{6,u}$
\begin{equation}\label{eq:z:m}
Z_{\mathrm{m}}(u) = \overline{T_2 \cdot [\phi_{6,u}]}
\end{equation}
is a normal rational quintic curve preserved by the subgroup $T_2 \subset \PGL_2(\Bbbk)$.
If~\eqref{eq:uu15u1} fails the orbit closure is either the line $L_{xy}$ (if $u = 0$),
or a normal rational quartic curve (if $u^4 = 1$), or a singular rational quintic curve (if $u^4 = 1/5$).

We have
\begin{align*}
Z_{\mathrm{m}}(u) \cap \Orb_3(Y) &= T_2 \cdot [\phi_{6,u}] \cong \mathbb{A}^1 \setminus \{0\},\\
Z_{\mathrm{m}}(u) \cap \Orb_2(Y) &= [xy^5],\\
Z_{\mathrm{m}}(u) \cap \Orb_1(Y) &= [x^6].
\end{align*}%
We denote by $X^{\mathrm{m}}(u)$ the Fano threefold of index 1 and genus 12 corresponding to
the quintic $Z_{\mathrm m}(u)$ via the construction of Theorem~\xref{theorem:double-projection}.
\end{example}

In what follows we refer to varieties $X^{\mathrm{MU}}$, $X^{\mathrm a}$, and $X^{\mathrm m}(u)$ defined in Theorem~\ref{theorem-mu-22}
and Examples~\ref{example-V22-a} and~\ref{example-V22-m} as \emph{special} Fano threefolds of genus $12$.
According to Remark~\ref{remark:functoriality} and in view of the construction of the curves $Z_{\mathrm{MU}}$, $Z_{\mathrm{a}}$, and $Z_{\mathrm{m}}(u)$
we have $B_2 \subset \Aut(X^{\mathrm{MU}})$, $U_2 \subset \Aut(X^{\mathrm{a}})$, and $T_2 \subset \Aut(X^{\mathrm{m}}(u))$.
We already know that $\Aut(X^{\mathrm{MU}})$ is actually much bigger.
In~\S\ref{subsection:aut-explicit} we will show that the other two groups are slightly bigger as well.

\begin{remark}
The construction of varieties $X^{\mathrm{MU}}$ and $X^{\mathrm{a}}$ does not depend on any parameter, so these are single varieties.
On the contrary, the construction of $X^{\mathrm{m}}$ depends on the parameter~$u$.
In fact, this is not quite precise.
Indeed, let $\upzeta$ be a primitive fourth root of unity.
It is easy to see that the polynomials $\phi_{6,u}(x,y)$ and $\phi_{6,\upzeta u}(x,\upzeta y)$ are proportional,
hence the $T_2$-orbits of $[\phi_{6,u}]$ and $[\phi_{6,\upzeta u}]$ coincide.
Thus $Z_{\mathrm{m}}(u) = Z_{\mathrm{m}}(\upzeta u)$ and~\mbox{$X^{\mathrm{m}}(u) \cong X^{\mathrm{m}}(\upzeta u)$}.
So the space of parameters for the family of varieties $X^{\mathrm{m}}(u)$ is
\begin{equation*}
\Big(\PP_u^1 \setminus \{0, \sqrt[4]{1}, \sqrt[4]{1/5}, \infty\} \Big)/\mumu_4 =
\PP_{u^4}^1 \setminus \{0,1,1/5,\infty\},
\end{equation*}
with $u^4$ being a coordinate.
\end{remark}

The next lemma shows that the quintics $Z_{\mathrm{MU}}$, $Z_{\mathrm{a}}$, and $Z_{\mathrm{m}}(u)$
exhaust all rational normal quintic curves in $Y$ with an infinite stabilizer inside $\PGL_2(\Bbbk)$.

\begin{lemma}
\label{lemma:z:classification}
Assume that $Z \subset Y$ is a rational normal quintic curve, invariant with respect to a non-trivial connected solvable algebraic group $B \subset \PGL_2(\Bbbk)$.
Then $Z$ is conjugate under the action of $\Aut(Y) = \PGL_2(\Bbbk)$ to one of the curves $Z_{\mathrm{MU}}$, $Z_{\mathrm{a}}$, and~$Z_{\mathrm{m}}(u)$
described by~\eqref{eq:z:mu}, \eqref{eq:z:a}, or~\eqref{eq:z:m}.
\end{lemma}
\begin{proof}
Since the subgroup $B \subset \PGL_2(\Bbbk)$ is conjugate to one of the subgroups $B_2$, $T_2$, or $U_2$
discussed in~\S\xref{subsection:v5},
we can assume without loss of generality that $B$ is one of these subgroups. Let us consider these cases one-by-one.

First, assume that $B = B_2$. Since $Z \cong \PP^1$, every point of $Z$
has a nontrivial one-dimensional stabilizer in $B$, hence $Z \subset \oOrb_2(Y)$.
By Remark~\ref{remark:curves-oorb2} $Z$ is conjugate to~$Z_{\mathrm{MU}}$.

Moreover, the quintics conjugate to $Z_{\mathrm{MU}}$ are the only smooth rational quintics contained in~\mbox{$\oOrb_2(Y)$},
so from now on we may assume that $Z \not\subset \oOrb_2(Y)$.

An arbitrary point of the open orbit~\mbox{$\Orb_3(U)$} can be written as $[\varphi]$, where
\begin{equation*}
\varphi = fg(f^4 - g^4),
\end{equation*}
and $f$, $g$ are linear forms, so we may assume that $Z$ is the closure of the $B$-orbit
of~$[\varphi]$.

Now, assume that $B = U_2$. For general $f$ and $g$ the closure of the $U_2$-orbit of $[\varphi]$ is a curve of degree $6$.
For it to have degree 5, it is necessary for $\varphi$ to be divisible by $x$.
Conjugating by an element of $\Aut(Y)_\varphi \cong \Oct$, we may assume that $f = x$
(up to a scalar multiple). But then (again up to a scalar multiple) $\varphi$ should be equal to
\begin{equation*}
\phi_{6,u,v}(x,y) = x(ux+vy)(x^4 - (ux+vy)^4)
 = \left(\begin{smallmatrix}
1 & u \\ 0 & v
\end{smallmatrix}\right)
\cdot \phi_6,
\qquad u \in \Bbbk,\ v \in \Bbbk^\times.
\end{equation*}
Such point is obtained from $[\phi_6]$ by a $B_2$-action. But the group $B_2$ normalizes the
subgroup~\mbox{$U_2 \subset \PGL_2(\Bbbk)$},
hence $Z$ is conjugate to the closure of the $U_2$-orbit of $[\phi_6]$, i.e.
to~$Z_{\mathrm{a}}$.

Finally, assume that $B = T_2$.
Again, for general $f$ and $g$ the closure of the $T_2$-orbit of $[\varphi]$ has degree 6, and the degree is smaller if and only if $\varphi$ is divisible by $x$ or $y$.
Conjugating by an element of $\Aut(Y)_\varphi \cong \Oct$ we again may assume that $f= x$, i.e.\ $\varphi = \phi_{6,u,v}$.
But this point is in the $T_2$-orbit of $[\phi_{6,u}]$, hence $Z$ is conjugate to~$Z_{\mathrm{m}}(u)$.
\end{proof}

Now finally, we can classify all Fano threefolds of index 1 and genus 12 with infinite automorphism groups,
which is the first main result of this section.

\begin{theorem}[{see \cite{Prokhorov-1990c}}]
\label{theorem:Prokhorov-g-12-finite}
Let $X$ be a \textup(smooth\textup) Fano threefold with $\rho(X) = 1$, $\iota(X) = 1$, and $\g(X) = 12$.
Then the automorphism group of $X$ is finite unless $X$ is a special Fano threefold of genus~$12$.
More precisely, if $X$ admits a faithful action of the group $B_2$, then $X \cong X^{\mathrm{MU}}$ is
the Mukai--Umemura threefold described in Theorem~\xref{theorem-mu-22} and mentioned in Example~\xref{example-MU}.
Otherwise, either $X \cong X^{\mathrm{a}}$ or $X \cong X^{\mathrm{m}}(u)$, see Examples~\xref{example-V22-a} and~\xref{example-V22-m}.
%
\end{theorem}

\begin{proof}
Let $B$ denote a maximal solvable (Borel) subgroup of the connected component~\mbox{$\Aut^0(X)$} of identity in the group~\mbox{$\Aut(X)$}.
The group $\Aut(X)$ is finite if and only if $B$ is trivial.
The group $B$ acts on the Hilbert scheme $\Sigma(X)$ of lines on $X$, and by the fixed-point theorem~\mbox{\cite[Theorem~VIII.21.2]{Humphreys1975}}
there exists a $B$-invariant line~\mbox{$L\subset X$}.
Then by Remark~\ref{remark:functoriality}
the associated curve~$Z$ is $B$-invariant and by Lemma~\xref{lemma:z:classification} the curve~$Z$ is $\Aut(Y)$-conjugate
to one of the curves $Z_{\mathrm{MU}}$, $Z_{\mathrm{a}}$, or $Z_{\mathrm{m}}(u)$, hence $X$ is isomorphic
to one of the special threefolds of Example~\ref{example-MU}, \ref{example-V22-a}, or~\ref{example-V22-m}.

Moreover, if $X$ is the Mukai--Umemura threefold of Theorem~\ref{theorem-mu-22} then $B = B_2$,
hence the corresponding quintic is $B_2$-invariant, hence is conjugate to $Z_{\mathrm{MU}}$.
Therefore, the threefold of Example~\ref{example-MU} is the Mukai--Umemura threefold.
%
\end{proof}

To complete the proof of Theorem~\ref{theorem:Prokhorov} it remains to describe explicitly
the automorphism groups of the threefolds $X^{\mathrm{a}}$ and $X^{\mathrm{m}}(u)$.
We do this in the next subsection.

\begin{remark}
One can also use the approach of Theorem~\xref{theorem:Prokhorov-g-12-finite} to establish finiteness of the automorphism group
of an arbitrary smooth Fano threefold $X$ with $\rho(X)=1$, $\iota(X)=1$, and $7\le \g(X)\le 10$,
see \cite{Prokhorov-1990c} for details.
\end{remark}

\subsection{Explicit automorphisms groups}
\label{subsection:aut-explicit}

The main ingredient in the explicit description of the automorphisms groups of $X = X^{\mathrm{a}}$ and $X = X^{\mathrm{m}}(u)$
is the description of the Hilbert scheme $\Sigma(X)$ of lines on $X$.
For this
we use Lemma~\ref{lemma:sigma-0-iso} relating it to $\Sigma_Z(Y)$,
where $Z$ is the corresponding quintic curve.
Accordingly, we start by describing $\Sigma_Z(Y)$.
We include the case $Z = Z_{\mathrm{MU}}$ for completeness.

\begin{lemma}
\label{lemma:sigma-z-y}
If $Z = Z_{\mathrm{MU}}$, $Z = Z_{\mathrm{a}}$, or $Z = Z_{\mathrm{m}}(u)$,
then the curve $\Sigma_Z(Y)$ is a plane quintic curve which can be described by the following picture:
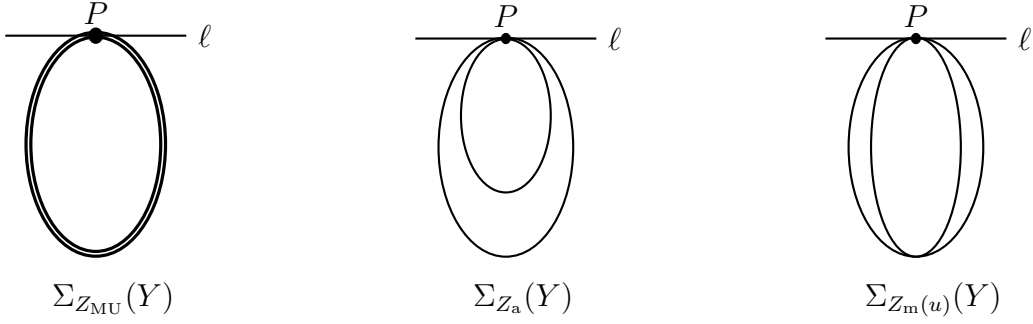
\begin{equation*}
\arraycolsep = 3em
\begin{array}{ccc}
\begin{tikzpicture}[xscale = .7, yscale = .8]
\draw[very thick, double] (0,0) ellipse (3em and 10ex);
\draw[thick] (-1.7,1.8) -- (0,1.8) node [above] {$P$} --(1.7,1.8) node [right] {$\ell$};
\draw[fill] (0,1.8) circle (.3em);
\end{tikzpicture}
&
\begin{tikzpicture}[xscale = .7, yscale = .8]
\draw[thick] (0,0) ellipse (3em and 10ex);
\draw[thick] (0,.52) ellipse (2em and 7ex);
\draw[thick] (-1.7,1.8) -- (0,1.8) node [above] {$P$} --(1.7,1.8) node [right] {$\ell$};
\draw[fill] (0,1.8) circle (.2em);
\end{tikzpicture}
&
\begin{tikzpicture}[xscale = .7, yscale = .8]
\draw[thick] (0,0) ellipse (3em and 10ex);
\draw[thick] (0,0) ellipse (2em and 10ex);
\draw[thick] (-1.7,1.8) -- (0,1.8) node [above] {$P$} --(1.7,1.8) node [right] {$\ell$};
\draw[fill] (0,1.8) circle (.2em);
\end{tikzpicture}
\\
\Sigma_{Z_{\mathrm{MU}}}(Y) &
\Sigma_{Z_{\mathrm{a}}}(Y) &
\Sigma_{Z_{\mathrm{m}}(u)}(Y)
\end{array}
\end{equation*}
In other words, $\Sigma_Z(Y)$ is the union of a line $\ell$ and two conics \textup{(}or a double conic, in the Mukai--Umemura case\textup{)}
tangent to $\ell$ at a certain point $P \in \Sigma(Y)$.
\end{lemma}

\begin{proof}
To describe $\Sigma_Z(Y)$ we use consecutively Corollary~\ref{corollary:V5-lines-description}.

First, put $Z = Z_{\mathrm{a}} = \overline{U_2 \cdot [\phi_6]}$.
Clearly, one has
\begin{equation*}
q(p^{-1}([\phi_6])) = \{ [xy], [x^2-y^2], [x^2+y^2] \},
\end{equation*}
Hence
\begin{align*}
q(p^{-1}(U_2 \cdot [\phi_6])) &= (U_2 \cdot [xy]) \cup (U_2 \cdot [x^2 - y^2]) \cup (U_2 \cdot [x^2 + y^2]) =
\\
&= \{x(sx + y)\} \cup \{ x^2 - (sx + y)^2\} \cup \{ x^2 + (sx + y)^2\},
\end{align*}
where $s \in \Bbbk$.
The point at the boundary of $Z_{\mathrm{a}}$ is $[x^6]$ and $q(p^{-1}([x^6])) = [x^2]$.
We see that~$\Sigma_{Z_{\mathrm{a}}}(Y) = q(p^{-1}(Z_{\mathrm{a}}))$ is the union of a line
\begin{equation*}
\ell = \{ x(s_1x + s_2 y) \},
\end{equation*}
and two conics
\begin{equation*}
\gamma'_{\mathrm{a}} = \{ (s_1^2 - s_2^2)x^2 - 2s_1s_2xy - s_1^2y^2 \},
\qquad
\gamma''_{\mathrm{a}} = \{ (s_1^2 + s_2^2)x^2 + 2s_1s_2xy + s_1^2y^2 \},
\end{equation*}
tangent to it (and tangent to each other with multiplicity~4) at the point $P = [x^2]$.

If $Z = Z_{\mathrm{m}}(u) = \overline{T_2 \cdot [\phi_{6,u}]}$, then
\begin{equation*}
q(p^{-1}([\phi_{6,u}])) = \{ [x(ux + y)], [x^2-(ux + y)^2], [x^2+(ux + y)^2] \},
\end{equation*}
hence
\begin{align*}
q(p^{-1}(T_2 \cdot [\phi_{6,u}])) &= (T_2 \cdot [x(ux + y)]) \cup (T_2 \cdot [x^2 - (ux + y)^2]) \cup (T_2 \cdot [(x^2 + (ux + y)^2])=
\\
&= \{x(ux + ty)\} \cup \{ x^2 - (ux + ty)^2\} \cup \{ x^2 + (ux + ty)^2\},
\end{align*}
where $t \in \Bbbk^\times$.
The points at the boundary of $Z_{\mathrm{m}}(u)$ are $[x^6]$ and $[xy^5]$, and we have~\mbox{$q(p^{-1}([x^6])) = [x^2]$} and $q(p^{-1}([xy^5])) = \{[xy], [y^2]\}$.
Thus, $\Sigma_{Z_{\mathrm{m}}}(Y) = q(p^{-1}(Z_{\mathrm{m}}(u)))$ is the union of the line $\ell$ (the same line as in the case of $Z = Z_{\mathrm{a}}$)
and two conics
\begin{equation*}
\gamma'_{\mathrm{m}}(u) = \{ s_1^2(1 - u^2)x^2 - 2s_1s_2uxy - s_2^2y^2 \},
\quad
\gamma''_{\mathrm{m}}(u) = \{ s_1^2(1 + u^2)x^2 + 2s_1s_2uxy + s_2^2y^2 \},
\end{equation*}
tangent to~$\ell$ at the point $P = [x^2]$, and also tangent to each other with multiplicity~2 at the points~$[x^2]$ and~$[y^2]$ respectively.

Finally, if $Z = Z_{\mathrm{MU}} = \{x(s_1x + s_2y)^5\}_{(s_1:s_2) \in \P^1}$, then
\begin{equation*}
\Sigma_{Z_{\mathrm{MU}}(Y)} = q(p^{-1}(Z_{\mathrm{MU}})) = \{ x(s_1x + s_2y) \} \cup \{ (s_1x + s_2y)^2 \}.
\end{equation*}
Thus, $\Sigma_{Z_{\mathrm{MU}}}(Y)$ is the union of the line $\ell$ (the same line again) and of the conic
\begin{equation*}
\gamma_{\mathrm{MU}} = \{ (s_1x + s_2y)^2 \}
\end{equation*}
tangent to~$\ell$ at the point $P = [x^2]$.
Since the lines parameterized by this conic are special, its preimage $q^{-1}(\gamma_{\mathrm{MU}})$ is the ramification divisor of $p \colon \CL(Y) \to Y$,
hence the component of $\Sigma_{Z_{\mathrm{MU}}}(Y)$ underlying the conic $\gamma_{\mathrm{MU}}$ is everywhere non-reduced.

Summarizing, we can write
\begin{equation*}
\Sigma_{Z}(Y) =
\begin{cases}
\ell \cup 2\gamma_{\mathrm{MU}}, 					& \text{if $Z = Z_{\mathrm{MU}}$},\\
\ell \cup \gamma'_{\mathrm{a}} \cup \gamma''_{\mathrm{a}}, 		& \text{if $Z = Z_{\mathrm{a}}$},\\
\ell \cup \gamma'_{\mathrm{m}}(u) \cup \gamma''_{\mathrm{m}}(u), 	& \text{if $Z = Z_{\mathrm{m}}(u)$}.
\end{cases}
\end{equation*}
This completes the proof of the lemma.
\end{proof}

Another observation that we need is the following.
Denote by $\langle Z  \rangle$ the linear span of the quintic $Z$.
It is a hyperplane in $\P(M_6) = \P^6$.

\begin{lemma}
\label{lemma:bisecant}
Let $Z = Z_{\mathrm{MU}}$, $Z = Z_{\mathrm{a}}$, or $Z = Z_{\mathrm{m}}(u)$ and
\begin{equation*}
F = Y \cap \langle Z \rangle.
\end{equation*}
Then $F$ is a non-normal quintic surface whose normalization is the Hirzebruch surface~$\F_3$.
The normalization map $\F_3 \to F$
glues the exceptional section with one fiber of $\F_3$ into the line $L_{x^2} = \Sing(F)$.
The line $L_{x^2}$ is the unique bisecant of $Z$ and corresponds to the distinguished point $P \in \Sigma_Z(Y)$.
Any line on $Y$ intersecting both $Z$ and $L_{x^2}$ is the image of the fiber of $\F_3$; these lines are parameterized by the component $\ell \subset \Sigma_Z(Y)$.
\end{lemma}
\begin{proof}
In all three cases the linear span $\langle Z \rangle$ is the hyperplane
\begin{equation*}
\langle x^6, x^5y, x^4y^2,x^3y^3, x^2y^4, xy^5 \rangle \subset \P(M_6),
\end{equation*}
so~$F$ is the corresponding hyperplane section of $Y$.
In particular, it is a quintic surface.
The line $L_{x^2}$ is contained in the hyperplane, hence also in $F$, and the same is true for any line~$L_{x(s_1x + s_2y)}$ parameterized by $\ell \subset \Sigma_Z(Y)$.
The lines $L_{x^2}$ and $L_{x(s_1x + x_2y)}$ meet at the point~$[x^5(s_1x+s_2y)]$, so the surface $F$ is swept out by secants of $L_{x^2}$.
Therefore, applying the main result of~\cite{FurushimaTada1989} we conclude that $F$ is non-normal and its normalization is a Hirzebruch surface.
Moreover, since the line~$L_{x^2}$ is special (see Corollary~\ref{corollary:V5-lines-description}), the normalization of~$F$ is~$\F_3$ and the normalization map
glues the exceptional section of $\F_3$ with a fiber.

Let $s$ denote the class of the exceptional section of $\F_3$, and $f$ the class of a fiber.
Since the fibers of $\F_3$ are mapped to lines on $Y$, and since the image of $\F_3$ is a quintic surface, the map $\F_3 \to F \to \langle Z \rangle \cong \P^5$
is given by an incomplete linear subsystem in $|s + 4f|$.
Let us check that the curve $Z$ is also the image of a member of the same linear system.
Indeed, $Z$ is a smooth quintic curve, and $|s + 4f|$ is the only linear system that contains integral curves of degree 5 with respect to $s + 4f$.

Now, we can check the last two assertions of the lemma.
Any bisecant of $Z$ is contained in the linear span $\langle Z \rangle$, hence in the surface $F$.
Therefore, it is the image of a fiber of $\F_3$.
Since $(s+4f) \cdot f = 1$, the image of a fiber intersects $Z$ in a single point and the intersection is transversal,
unless this is the fiber that is glued with the exceptional section.
This shows that $L_{x^2}$ is the unique bisecant of $Z$.
Finally, any line intersecting both $Z$ and $L_{x^2}$ is also contained in $\langle Z \rangle$, hence lies on $F$, hence is the image of a fiber of $\F_3$.
And as we have seen above, these lines are parameterized by $\ell$.
\end{proof}

Lemma~\ref{lemma:bisecant} allows to describe the flopping locus of the birational transformation $\upchi$
in the diagram~\eqref{equation(1)}.

\begin{proposition}\label{proposition:flopping-loci}
Let $X = X^{\mathrm{MU}}$, $X = X^{\mathrm{a}}$, or $X = X^{\mathrm{m}}(u)$ and $Z \subset Y$ is the corresponding quintic curve.
The flopping loci of the birational transformation $\upchi$
in the diagram~\eqref{equation(1)} is the exceptional section of the exceptional divisor $E_L \subset  X'$,
and the strict transform of the unique bisecant $L_{x^2}$ of $Z$ in $Y'$.
In particular, the line $L \subset X$ is special and does not intersect any other line on $X$.
\end{proposition}
\begin{proof}
By Lemma~\ref{lemma:flopping-locus} the flopping locus in $Y'$ consists of strict transforms of bisecants of~$Z$.
So, by Lemma~\ref{lemma:bisecant} there is a unique flopping curve for $\upchi^{-1}$.
Consequently, the same is true for the map~$\upchi$.

On the other hand, the surface $F$ is the image of the exceptional divisor $E_L$.
Since its normalization is isomorphic to $\F_3$, it follows that $L$ is a special line on $X$.
Indeed, otherwise $E_L\cong \F_1$ and the map $\sigma_Y \circ \upchi$ is regular near a general fiber
of $E_L$. Then the image of fibers must be irreducible conics on $F$.
On the other hand, $F$ does not contain irreducible conics.
Hence, the exceptional section of the exceptional divisor is in the flopping locus.
But as we have already shown, the flopping locus consists of a single curve.
This means that no other line on $X$ intersects $L$.
\end{proof}

Combining the above assertions we obtain the following

\begin{proposition}
\label{proposition:lines-special-v22}
The Hilbert scheme of lines on a special Fano threefold $X$ of genus~$12$ has the following description:
\begin{itemize}
\item If $X = X^{\mathrm{MU}}$ then $\Sigma(X)$ is a smooth rational curve with an non-reduced scheme structure.
\item If $X = X^{\mathrm{a}}$ then $\Sigma(X)$ is the union of two rational curves glued at a point $P$, such that $\Sing(\Sigma(X)) = P$.
\item If $X = X^{\mathrm{m}}(u)$ then $\Sigma(X)$ is the union of two smooth rational curves glued at two simple tangency points $P$ and $P'$.
\end{itemize}
\end{proposition}

\begin{warning}
One can actually prove that in the case $X = X^{\mathrm{a}}$ the components of~$\Sigma(X)$ are smooth and the point $P$ is their tangency point of multiplicity 4.
Moreover, in fact in all cases considered in Proposition~\ref{proposition:lines-special-v22} the Hilbert scheme $\Sigma(X)$
has a natural structure of a plane quartic (see also Remark~\ref{remark-V22-VSP}),
which is either a double conic, or a union of two conics with a single common point, or a union of two conics with two tangency points,
so the picture below is adequate.
However, we do not need all these facts, and the proof that we have in mind requires going in too much details, so we skip it.
\end{warning}

The next picture shows how $\Sigma(X)$ looks:
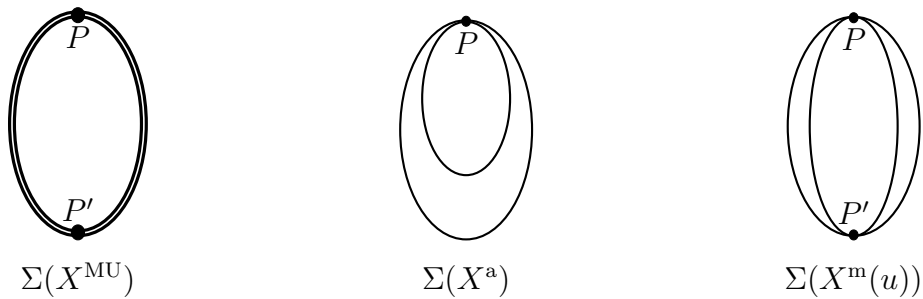
\begin{equation*}
\arraycolsep = 4em
\begin{array}{ccc}
\begin{tikzpicture}[xscale = .7, yscale = .8]
\draw[very thick, double] (0,0) ellipse (3em and 10ex);
\draw (0,1.8) node [below] {$P$};
\draw (0,-1.8) node [above] {$P'$};
\draw[fill] (0,1.8) circle (.3em);
\draw[fill] (0,-1.8) circle (.3em);
\end{tikzpicture}
&
\begin{tikzpicture}[xscale = .7, yscale = .8]
\draw[thick] (0,0) ellipse (3em and 10ex);
\draw[thick] (0,.52) ellipse (2em and 7ex);
\draw (0,1.8) node [below] {$P$};
\draw[fill] (0,1.8) circle (.2em);
\end{tikzpicture}
&
\begin{tikzpicture}[xscale = .7, yscale = .8]
\draw[thick] (0,0) ellipse (3em and 10ex);
\draw[thick] (0,0) ellipse (2em and 10ex);
\draw (0,1.8) node [below] {$P$};
\draw (0,-1.8) node [above] {$P'$};
\draw[fill] (0,1.8) circle (.2em);
\draw[fill] (0,-1.8) circle (.2em);
\end{tikzpicture}
\\
\Sigma(X^{\mathrm{MU}}) &
\Sigma(X^{\mathrm{a}}) &
\Sigma(X^{\mathrm{m}}(u))
\end{array}
\end{equation*}

\begin{proof}
Let $L$ be the line on $X$ obtained from the pair $(Y,Z)$ by the construction of Theorem~\ref{theorem:double-projection}(ii).
By Proposition~\ref{proposition:flopping-loci} and Lemmas~\ref{lemma:sigma-0-iso} and~\ref{lemma:bisecant} we have an isomorphism
\begin{equation*}
\Sigma(X) \setminus \{[L]\} = \Sigma_L^0(X) \cong \Sigma_Z^0(Y) = \Sigma_Z(Y) \setminus \ell.
\end{equation*}
Thus, $\Sigma(X)$ is a one-point compactification of $\Sigma_Z(Y) \setminus \ell$.
Using the description of~$\Sigma_Z(Y)$ given in Lemma~\ref{lemma:sigma-z-y} we deduce all assertions of the proposition,
except for the local description of $\Sigma(X)$ at the point $P$ (corresponding to the line $L$) in the Mukai--Umemura and multiplicative cases.
For this we can argue as follows.
First, replace the line $L$ with the line $L'$ corresponding to the point $P'$ (any other point in the Mukai--Umemura case,
and the other singular point in the multiplicative case) and consider the quintic curve~$Z' \subset Y$ associated with the pair $(X,L')$.
By Lemma~\ref{lemma:z:classification} we conclude that~$Z'$ is conjugate to~$Z_{\mathrm{m}}(u')$,
for some $u'$ possibly different from $u$, with respect to the $\Aut(Y)$-action,
so it follows that
the local behavior of $\Sigma(X)$ at $P$ is the same as at $P'$.
\end{proof}

Now we are ready to prove the second main result of this section.

\begin{proposition}
\label{proposition:automorphisms-special-v22}
The automorphism groups of special Fano threefolds of genus~$12$ are the following:
\begin{equation*}
\Aut(X^{\mathrm{MU}}) \cong \PGL_2(\Bbbk),
\qquad
\Aut(X^{\mathrm{a}}) \cong \GG_{\mathrm a} \rtimes \mumu_4,
\qquad
\Aut(X^{\mathrm{m}}(u)) \cong \GG_{\mathrm m} \rtimes \mumu_2.
\end{equation*}
\end{proposition}
\begin{proof}
The first isomorphism is given by Theorem~\ref{theorem-mu-22}, so we concentrate on the other two.
The group~$\Aut(X)$ acts on the Hilbert scheme $\Sigma(X)$, and as a consequence on the set~$\Sing(\Sigma(X))$,
which by Proposition~\ref{proposition:lines-special-v22} is a single point in the case $X = X^{\mathrm a}$ and a two-point set in the case~$X = X^{\mathrm m}(u)$.
Let $L \subset X$ be the line corresponding to the singular point~$P$ of $\Sigma(X)$,
and let $\Aut(X; L) \subset \Aut(X)$ be the subgroup that preserves~$L$.
Then we have an equality
\begin{equation*}
\Aut(X^{\mathrm a}) = \Aut(X^{\mathrm a}; L),
\end{equation*}
and an exact sequence
\begin{equation}
\label{eq:aut-xm-sequence}
1 \to \Aut(X^{\mathrm m}(u); L) \to \Aut(X^{\mathrm m}(u)) \to \mumu_2,
\end{equation}
where the group $\mumu_2$ is considered as the group of permutations of the set $\{P,P'\}$.

On the other hand, we have an isomorphism~\eqref{eq:aut-xl-yz}.
A simple computation shows that
\begin{equation*}
\Aut(X^{\mathrm a}; L) \cong {} \Aut(Y; Z_{\mathrm a}) = U_2 \rtimes \mumu_4 \cong \GG_{\mathrm a} \rtimes \mumu_4.
\end{equation*}
Indeed, if $g \in \Aut(Y)$ preserves $Z_{\mathrm a} = \overline{U_2 \cdot [\phi_6]}$,
then $g([\phi_6]) =[\phi_{6,u}]$ for some $u \in U_2$,
hence~$u^{-1}g$ is an element of the stabilizer $\Oct$ of the point $[\phi_6]$ that preserves $U_2$.
Therefore, we have $u^{-1}g \in \mumu_4$, where $\mumu_4$ is the subgroup of the octahedral group fixing one of the octahedron axes.
It is generated by the element
\begin{equation*}
\tau=
\left(\begin{smallmatrix}
\upzeta & 0\\ 0&1
\end{smallmatrix}\right)
\in \PGL_2(\Bbbk),
\end{equation*}
where $\upzeta$ is a fourth root of unity, and we finally get the required description of $\Aut(X^{\mathrm{a}})$.

Similarly,
\begin{equation*}
\Aut(X^{\mathrm m}(u); L) \cong {} \Aut(Y; Z_{\mathrm m}(u)) = T_2 \cong \GG_{\mathrm m}.
\end{equation*}
Indeed, if $g \in \Aut(Y)$ preserves $Z_{\mathrm m}(u) = \overline{T_2 \cdot [\phi_{6,u}]}$, then $g([\phi_{6,u}]) = t([\phi_{6,u}])$ for some~\mbox{$t \in T_2$},
hence $t^{-1}g$ is an element of the stabilizer $u \cdot \Oct \cdot u^{-1}$ of the point $[\phi_{6,u}]$ (where we consider $u$ as an element of $U_2$) that preserves $T_2$.
In other words,
\begin{equation*}
t^{-1} g \in (u \cdot \Oct \cdot u^{-1}) \cap (T_2 \cup w(T_2)),
\end{equation*}
where $w$ is the nontrivial element of the Weyl group ${\mathfrak S}_2$ of $\PGL_2(\Bbbk)$.
But for $u$ satisfying the inequality of Example~\ref{example-V22-m} the intersection on the
right hand side is trivial, and we finally see that $g \in T_2$.

To conclude
we note that by~\cite[Proposition~5.1]{Dinew-Kapustka-Kapustka-2015}
the group $\Aut(X^{\mathrm{m}}(u))$ contains an extra involution
hence~\mbox{$\Aut(X^{\mathrm{m}}(u)) \ne \Aut(X^{\mathrm{m}}(u); L)$},
so the second map in~\eqref{eq:aut-xm-sequence} is surjective, and we get the required description of $\Aut(X^{\mathrm{m}}(u))$.
\end{proof}

\begin{remark}
\label{remark-V22-VSP}
According to S.\,Mukai \cite{mukai1989biregular} (see also \cite{Mukai-1992}, \cite{Schreyer2001}) any
Fano threefold $X$ with $\rho(X)=1$, $\iota(X)=1$, and $\g(X)=12$
can be realized as a variety of sums of powers that parameterizes
polar hexagons of a plane quartic curve $\mathfrak{C}$ (see \cite[\S5]{Mukai-1992} for a definition and~\cite{Dinew-Kapustka-Kapustka-2015} for some details),
and the Hilbert scheme $\Sigma(X)$ is the \emph{Scorza transform} of
$\mathfrak{C}$, (see~\cite[\S7]{DolgachevKanev} for a definition).
Unfortunately, a complete proof of these facts is not yet published,
while the construction of the extra involution in~\cite{Dinew-Kapustka-Kapustka-2015} relies on them.

To establish the existence of an involution in $\Aut(X^{\mathrm{m}}(u))\setminus \GG_{\mathrm m}$,
independent of the above Mukai's results,
one can use another equivariant Sarkisov link similar to~\eqref{equation(1)},
see~\cite{Kuznetsov-Prokhorov-Gm} for details.
\end{remark}

Propositions~\xref{proposition:nearly-done} and~\xref{proposition:automorphisms-special-v22}
together with Theorem~\ref{theorem:Prokhorov-g-12-finite} (and a classification of smooth Fano threefolds of Picard rank $1$, see~\cite[\S12.2]{Iskovskikh-Prokhorov-1999})
give a proof of Theorem~\xref{theorem:Prokhorov}.


\appendix

\section{Some standard results on conics}
\label{appendix:new}

In this section we collect some well-known results about conics. We refer to~\S\xref{subsection:lines-and-conics} for our notation and conventions.

\subsection{Conics on surfaces}

For any variety $Z\subset\P^N$ we denote by $\Sigma(Z)$ and $S(Z)$ the Hilbert schemes of lines
and conics contained in $Z$, respectively, see~\S\xref{subsection:lines-and-conics}.

\begin{lemma}
\label{lemma:many-reducible-conics}
Let $Z\subset\P^N$ be an integral variety such that $\dim\Sigma(Z)\le 1$.
Suppose that there is an irreducible two-dimensional closed subset $S_0\subset S(Z)$ such that
a general point of $S_0$ corresponds to a reducible \textup(reduced\textup) conic.
Then either $Z$ contains a cone over a curve $B$ and~\mbox{$S_0\cong\Sym^2(B)$},
or $Z$ contains a smooth quadric surface.
\end{lemma}
\begin{proof}
Since a general point of $S_0$ corresponds to a reducible conic, one of the two possibilities occur:
either $\Sigma(Z)$ has a one-dimensional irreducible component $\Sigma_0$ such that any two lines corresponding to
its points meet each other, or $\Sigma(Z)$ has two one-dimensional irreducible components
$\Sigma_1$ and $\Sigma_2$ such that every line corresponding to a point of $\Sigma_1$ meets
every line corresponding to a point of $\Sigma_2$.

Suppose that the first possibility occurs.
Take two different lines $L'_0$ and $L''_0$ corresponding to the points of $\Sigma_0$.
They intersect at a point, say~$P$, and span a plane, say $\Pi$.
A general line corresponding to a point of $\Sigma_0$ intersects both $L'_0$ and $L''_0$.
Therefore either it passes through~$P$ (hence these lines
sweep a cone that gives the first option listed in the assertion of the lemma with $B = \Sigma_0$),
or it is contained in $\Pi$ (hence these lines sweep the plane $\Pi$, in which case $\Sigma(Z)$
is two-dimensional, which is a contradiction).

Now suppose that the second possibility occurs.
Take two general lines $L'_1$ and $L''_1$ corresponding to points of $\Sigma_1$.
We may assume that they do not intersect each other (otherwise we are in the situation considered above).
Lines in~$\P^N$ that intersect both $L'_1$ and~$L''_1$ are then parameterized by $L'_1 \times L''_1$,
so we can consider the curve $\Sigma_2$ as a curve in~$L'_1 \times L''_1$.
But since as before we may assume that the lines parameterized by $\Sigma_2$ do not intersect each other,
therefore the projections of the curve $\Sigma_2$ to the factors of $L'_1 \times L''_1$ are bijective,
hence lines in $\Sigma_2$ sweep a smooth quadric surface.
\end{proof}

The following classically known result whose proof can be obtained by
combining the results of~\cite{Castelnuovo1894} and~\cite{Segre21}
is very useful.

\begin{lemma}
\label{lemma:lines-conics-plane}
Let $Z \subset \PP^N$ be an integral surface.
Assume that $\dim S(Z) \ge 2$ and $Z$ is not a cone.
Then $Z$ is the Veronese surface
$$v_2(\P^2)\subset\P^5,$$
or its linear \textup(regular or rational\textup) projection.
In particular, one has
$\deg Z\le 4$.

Moreover, if $\dim \Sigma(Z) \ge 1$
then $Z$ is a cubic scroll
\begin{equation*}
\PP_{\PP^1}\big(\CO(-1) \oplus \CO(-2)\big)\subset\P^4,
\end{equation*}
or its linear projection, so that
$\deg Z\le 3$.
\end{lemma}

\begin{proof}
We use the ideas and methods from the proofs of \cite[Theorem~3.4.1]{Russo2016}
and~\mbox{\cite[Theorem~3.4.4]{Russo2016}}.

First, consider the Hilbert scheme of lines $\Sigma(Z)$ and the diagram~\eqref{diagram-universal-line}.
Note that the map $p \colon \CL(Z) \to Z$ cannot have fibers of positive dimension, since otherwise $Z$ would be
a cone. In particular, this implies that $\dim \Sigma(Z) \le 1$.

Let $S_0$ be an irreducible component of $S(Z)$ such that $\dim S_0 \ge 2$.
Note that a smooth quadric can be represented as a linear projection
of a cubic scroll, and thus also as a projection of the Veronese surface.
Therefore, by Lemma~\xref{lemma:many-reducible-conics}
we can assume that a general point of $S_0$ corresponds to a smooth conic.

Suppose that $\dim\Sigma(Z) = 1$.
Let $\Sigma_0\subset \Sigma(Z)$ be a one-dimensional irreducible component and let
$q\colon\CL_0(Z)\to \Sigma_0$ be the corresponding family of lines.
Thus we have a finite surjective morphism
\begin{equation*}
p\colon \CL_0(Z)\to Z\subset \PP^N.
\end{equation*}
If $p$ is not birational then there is a two-dimensional family of pairs of intersecting lines in $\Sigma_0$, which means that
general lines $L_1$ and $L_2$ corresponding to points of $\Sigma_0$ meet each other.
This gives on $Z$ a two-dimensional family of reducible conics and by Lemma~\ref{lemma:many-reducible-conics} implies that $Z$ is a smooth quadric.
So, we may assume that $p$ is birational.

Let $B$ be the normalization of $\Sigma_0$ and $Z_0 = B \times_{\Sigma_0} \CL_0(Z)$ the pullback to $B$ of the universal family over $\Sigma_0$.
Then $Z_0$ is a ruled surface over $B$.
Since $p$ is birational, the preimage in $Z_0$ of a general conic in $Z$ is a rational curve that projects nontrivially to $B$. Hence $B$ is rational and so
\begin{equation*}
Z_0 \cong \PP_{\PP^1}(\CO \oplus \CO(-e)),
\end{equation*}
for some $e \ge 0$.
Denote by $s$ the class of the exceptional section of $Z_0$ and by $f$ the class of the fiber.
The map $Z_0 \to Z \hookrightarrow \PP^N$ is given by a subsystem of the linear
system~\mbox{$|s + nf|$} for some integer $n$.
Note that $n \ge e$
(otherwise the linear system has base points on the
exceptional section), and if $n=e$ then the image of $Z_0$
is a cone. Thus one has $n>e$.
Assume further that
(the preimage on $Z_0$ of) the class of a non-degenerate conic $C\subset Z$
corresponding to a general point of $S_0$
is $as + bf$ for some~\mbox{$a,b \in \Z$}. Then we have
\begin{equation*}
2 = (as + bf)(s+nf) = -ae + b + an = b + (n-e)a.
\end{equation*}
On the other hand, since $C$ is irreducible and movable, we see that~\mbox{$b\ge ea\ge 0$}.
Moreover, we have $a\neq 0$ and $b\neq 0$ because $\dim S(Z) \ge 2$.
Taking all this into account we get
\begin{equation*}
a = b=1,\quad 0 \le e \le 1,\quad n = e + 1.
\end{equation*}
If $e = 0$ we conclude that $Z$ is a linear projection of a quadric~\mbox{$\PP^1\times \PP^1 \subset \PP^3$},
and if $e = 1$ we conclude that $Z$ is a linear projection of the cubic scroll, i.e., of the surface
\begin{equation*}
\PP_{\PP^1}\big(\CO \oplus \CO(-1)\big)\cong \PP_{\PP^1}\big(\CO(-1) \oplus \CO(-2)\big)
\end{equation*}
embedded into $\PP^4$ via the linear system~\mbox{$|s+2f|$}.
Both can also be represented as linear projections of the Veronese surface.
This finishes the proof in the case when~\mbox{$\dim \Sigma(Z) \ge 1$}.

From now on we assume that $\Sigma(Z)$ is at most finite.
Choose a sufficiently general point~\mbox{$z \in Z$}.
Then $z$ is a smooth point of $Z$ that does not lie on a line.
Therefore, all conics passing through $z$ are irreducible,
and in particular smooth at $z$.

Let $S_z$ be the subscheme of $S(Z)$ parameterizing the conics that pass through the point~$z$, and $q_z\colon \CC_z \to S_z$ be the corresponding universal family.
Since all conics of $S_z$ are smooth at $z$, the point $z$
defines a section $\Pi\subset\CC_z$ of the fibration~$q_z$.

Let $p_z\colon \CC_z\to Z$ be the tautological morphism.
By~\cite[Proposition~V.3.7.5]{Kollar-1996-RC} this morphism is birational.
Clearly, $\Pi$ is the scheme-theoretical preimage of $z$.
Since~$\Pi$ is a Cartier divisor on $\CC_z$, the map $p_z$ lifts to a map $\tilde{p}_z \colon \CC_z \to \tilde{Z}$, where
$\sigma\colon\tilde{Z}\to Z$ is the blow up of the point $z$.
Thus we have a commutative diagram
\begin{equation*}
\xymatrix{
& \CC_z \ar@{->}[r]^{\tilde{p}_z} \ar@{->}[rd]_{p_z} \ar[dl]^{q_z} &
\tilde{Z}\ar@{->}[d]^{\sigma}\\
S_z && Z
}
\end{equation*}
cf.\ the proof of~\cite[Theorem~2.1]{Fusi15}.

Since the morphism $\tilde{p}_z$ is birational, and the surface $\tilde{Z}$ is normal (and even smooth) in a neighborhood of the exceptional divisor $E$ of $\sigma$, the morphism
\begin{equation*}
\tilde{p}_z\vert_{\Pi}\colon\Pi\to E
\end{equation*}
is birational as well. Since $E\cong\P^1$ is a smooth curve,
this implies that $\tilde{p}_z\vert_{\Pi}$ is actually an isomorphism,
which means that a conic contained in $Z$ and passing through $z$
is uniquely defined by its tangent direction.
Now~\cite[Theorem~2.5]{Fusi15}
implies that there is a birational (possibly biregular) map
\begin{equation*}
\zeta\colon \P^2\dasharrow Z\subset\P^N
\end{equation*}
defined by a linear subsystem of $|\CO_{\P^2}(2)|$.
This means that $Z$ is a Veronese surface or its linear projection.
\end{proof}

\begin{corollary}
\label{corollary:lines-conics-plane-1}
Let $Z$ be an integral surface and let $H$ be a Cartier divisor on $Z$ such that~\mbox{$2H$} is very ample.
Let $\Sigma(Z)$ be the Hilbert scheme of $H$-lines.
Suppose that $\dim \Sigma(Z) \ge 2$. Then $Z \cong \P^2$ and $H$ is the class of a line.
\end{corollary}
\begin{proof}
Clearly, one has $\Sigma(Z) \subset S(v_2(Z))$, where $v_2$ is the embedding of $Z$ given by the linear system $|2H|$, hence $\dim S(v_2(Z)) \ge 2$.
Moreover, $v_2(Z)$ contains no lines, hence is not a cone. Therefore, by Lemma~\xref{lemma:lines-conics-plane} it is an isomorphic projection of $v_2(\P^2)$.
This means that $Z \cong \P^2$ and $H$ is the class of a line.
\end{proof}

\subsection{Normal bundles of degenerate conics}
\label{subsection:degenerate-conics}

Below we describe a relation between a normal bundle of a reducible or non-reduced conic $C$ on a smooth projective variety $X$,
and normal bundles of the irreducible components of $C$, or the normal bundle of~$C_{\red}$, respectively.
Then we apply this description to conics on Fano threefolds of index~1.

\begin{lemma}\label{lemma:normal-reducible-conic-general}
Assume that $C = L_1 \cup L_2$ is a reducible conic on $X$, i.e.\ $L_1$ and $L_2$ are two distinct lines intersecting at a point $P$.
Then there are exact sequences
\begin{equation*}
0 \to \CN_{C/X} \to \CN_{C/X}\vert_{L_1} \oplus \CN_{C/X}\vert_{L_2} \to \CN_{C/X}\vert_P \to 0,
\end{equation*}
and
\begin{equation*}
0 \to \CN_{L_i/X} \to \CN_{C/X}\vert_{L_i} \to \CO_P \to 0, \quad i=1,2.
\end{equation*}
\end{lemma}
\begin{proof}
The first exact sequence can be written for any vector bundle on $C$;
it is obtained by taking the tensor product (over $\CO_C$) of this bundle with the canonical exact sequence
\begin{equation*}
0 \to \CO_C \to \CO_{L_1} \oplus \CO_{L_2} \to \CO_P \to 0,
\end{equation*}
where all maps are just restrictions.

To establish the second sequence recall that
\begin{equation*}
\Coker(I_C \to I_{L_1}) \cong \Ker(\CO_C \to \CO_{L_1}) \cong \CO_{L_2}(-P)
\end{equation*}
by the Snake Lemma. Thus
the natural embedding of the ideal sheaves $I_C \subset I_{L_1}$ extends to an exact sequence
\begin{equation}\label{eq:sequence-reducible-conic}
0 \to I_C \to I_{L_1} \to \CO_{L_2}(-P) \to 0.
\end{equation}
Note that
\begin{equation*}
I_C \otimes \CO_{L_1} \cong (I_C \otimes \CO_C) \otimes_{\CO_C} \CO_{L_1} \cong \CN_{C/X}^\vee\vert_{L_1},
\qquad
\CO_{L_1} \otimes \CO_{L_2}(-P) \cong \CO_P,
\end{equation*}
and $\Tor_1(\CO_{L_1},\CO_{L_2}(-P))$ is a torsion sheaf on $L_1$.
Therefore, tensoring the exact sequence~\eqref{eq:sequence-reducible-conic}
with~$\CO_{L_1}$, we deduce an exact sequence
\begin{equation*}
\ldots \to \Tor_1(\CO_{L_1},\CO_{L_2}(-P)) \to \CN_{C/X}^\vee\vert_{L_1} \to \CN_{L_1/X}^\vee \to \CO_P \to 0.
\end{equation*}
The sheaf $\CN_{C/X}^\vee$ as well as its restriction to $L_1$ is locally free,
hence any morphism to it from a torsion sheaf $\Tor_1(\CO_{L_1},\CO_{L_2}(-P))$ is zero.
Thus we have an exact triple
\begin{equation*}
0 \to \CN_{C/X}^\vee\vert_{L_1} \to \CN_{L_1/X}^\vee \to \CO_P \to 0.
\end{equation*}
Dualizing it, we obtain the required exact sequence for $L_1$.
The sequence for $L_2$ can be obtained in a similar way.
\end{proof}

\begin{corollary}\label{corollary:normal-reducible-conic-3fold}
If $X$ is a Fano threefold of index $1$ and $C \subset X$ is a reducible conic, the Euler characteristic of the normal bundle equals $\chi(\CN_{C/X}) = 2$.
\end{corollary}
\begin{proof}
We use the sequences of Lemma~\ref{lemma:normal-reducible-conic-general}.
By Lemma~\xref{lemma:normal-bundles} the Euler characteristic of~$\CN_{L_i/X}$ equals 1.
Since the Euler characteristic of $\CO_P$ is also 1, and that of $\CN_{C/X}\vert_P \cong \CO_P^{\oplus 2}$ equals 2,
we deduce that the Euler characteristic of $\CN_{C/X}$ equals $(1 + 1) + (1 + 1) - 2 = 2$.
\end{proof}

The case of a non-reduced conic is a bit more complicated.

\begin{lemma}\label{lemma:normal-nonreduced-conic-general}
Assume that $C$ is a non-reduced conic on a smooth projective variety $X$ and $C_\red = L$. Then there are exact sequences
\begin{equation*}
0 \to \CN_{C/X}\vert_L(-1) \to \CN_{C/X} \to \CN_{C/X}\vert_L \to 0,
\end{equation*}
and
\begin{equation*}
0 \to \CO_L(1) \to \CN_{L/X} \to \CN_{C/X}\vert_L \to \CO_L(2) \to 0.
\end{equation*}
\end{lemma}
\begin{proof}
Again, the first exact sequence can be written for any vector bundle on $C$:
it is obtained by taking the tensor product (over $\CO_C$) of this bundle with the canonical exact sequence
\begin{equation*}
0 \to \CO_L(-1) \to \CO_C \to \CO_L \to 0.
\end{equation*}
To establish the second sequence note that analogously to the reducible case the natural embedding of the ideal sheaves $I_C \subset I_L$ extends to an exact sequence
\begin{equation*}
0 \to I_C \to I_L \to \CO_L(-1) \to 0
\end{equation*}
Tensoring it with $\CO_L$ and taking into account that
\begin{align*}
\Tor_p(I_C,\CO_L) \cong \Lambda^{p+1}\CN_{C/X}^\vee\vert_L,&\qquad
\Tor_p(I_L,\CO_L) \cong \Lambda^{p+1}\CN_{L/X}^\vee,\\
\Tor_p(\CO_L(-1),\CO_L) &\cong \Lambda^{p}\CN_{L/X}^\vee(-1)
\end{align*}
for all $p$, we deduce an exact sequence
\begin{multline*}
0 \to \Lambda^{n-1}\CN_{L/X}^\vee(-1) \to \Lambda^{n-1}\CN_{C/X}^\vee\vert_L \to \Lambda^{n-1}\CN_{L/X}^\vee \to \dots \\
\ldots \to \Lambda^2\CN_{L/X}^\vee \to \CN_{L/X}^\vee(-1) \to \CN_{C/X}^\vee\vert_L \to \CN_{L/X}^\vee \to \CO_L(-1) \to 0,
\end{multline*}
where $n = \dim X$. Since all conormal sheaves are locally free of rank $n-1$, it follows that the first map in the sequence is an isomorphism, hence
\begin{equation*}
\det(\CN_{C/X}^\vee\vert_L) \cong \det(\CN_{L/X}^\vee)(-1).
\end{equation*}
Therefore, the kernel of the map $\CN_{C/X}^\vee\vert_L \to \CN_{L/X}^\vee$, which is a line bundle on $L$, is isomorphic to
\begin{equation*}
\det(\CN_{C/X}^\vee\vert_L) \otimes \det(\CN_{L/X}) \otimes \CO_L(-1) \cong \CO_L(-2).
\end{equation*}
As a result we get an exact sequence
\begin{equation*}
0 \to \CO_L(-2) \to \CN_{C/X}^\vee\vert_L \to \CN_{L/X}^\vee \to \CO_L(-1) \to 0.
\end{equation*}
The required exact sequence is obtained by dualization.
\end{proof}

\begin{corollary}\label{corollary:normal-nonreduced-conic-3fold}
If $X$ is a Fano threefold of index $1$ and $C \subset X$ is a non-reduced conic, the Euler characteristic of the normal bundle equals $\chi(\CN_{C/X}) = 2$.
\end{corollary}
\begin{proof}
We use the sequences of Lemma~\ref{lemma:normal-nonreduced-conic-general}.
By Lemma~\xref{lemma:normal-bundles} the Euler characteristic of $\CN_{L/X}$ equals 1.
Since the Euler characteristic of $\CO_L(1)$ and $\CO_L(2)$ equals 2 and 3 respectively, it follows that $\chi(\CN_{C/X}\vert_L) = 1 - 2 + 3 = 2$.
Since $\CN_{C/X}\vert_L$ is a vector bundle of rank 2, we have $\chi(\CN_{C/X}\vert_L(-1)) = \chi(\CN_{C/X}\vert_L) - 2 = 0$.
Therefore, the Euler characteristic of $\CN_{C/X}$ equals $2 + 0 = 2$.
\end{proof}

\section{Lines and conics on Fano threefolds}
\label{section-Lines-and-conics}

Throughout this section $X$ is a Fano threefold of index 1 and $Y$ is a Fano threefold of index 2
(both with Picard rank 1). We denote by $H_X$ the ample generator of~\mbox{$\Pic(X)$}, and by $L_X$ and $P_X$
the classes of a line and a point on $X$ in the corresponding Chow or cohomology groups.
The bounded derived category of coherent sheaves on $X$ is denoted
by~\mbox{$\D^b(X)$}, see~\cite{kuznetsov2014icm} for a recent survey.

The main goal of this section is to prove Theorem~\xref{theorem:S-vs-Sigma}.
Besides, we give more detailed proofs of some facts, used in~\cite{kuznetsov2009derived}.
As before, we use notation and conventions of~\S\xref{subsection:lines-and-conics}.

We start by recalling Mukai's results that describe Fano threefolds of index 1
and genus~\mbox{$g \ge 6$}
as complete intersections in homogeneous varieties.
Then we discuss the structure of derived categories of some
Fano threefolds, and define subcategories~\mbox{$\CA_X \subset \D^b(X)$}
and~\mbox{$\CB_Y \subset \D^b(Y)$} that contain the most essential geometric information about $X$ and $Y$.
In particular,
we show that the Hilbert schemes of lines $\Sigma(Y)$ and conics $S(X)$ can be
identified with certain moduli spaces of objects in these categories. Finally,
we show that this identification gives an isomorphism $S(X) \cong \Sigma(Y)$
for appropriate pairs $(X,Y)$.

\subsection{Fano threefolds as complete intersections in homogeneous varieties}

For each smooth Fano
threefold $X$ with $\rho(X)=1$, $\iota(X)=1$,
and even genus~\mbox{$\g(X) \ge 6$} Mukai constructed in~\cite{mukai1989biregular} a stable vector bundle $E$ on $X$ of rank~2 with the following properties.

\begin{theorem}[\cite{mukai1989biregular,Mukai-1992}]\label{theorem:mukai}
Let $X$ be a Fano threefold with $\rho(X)=1$, $\iota(X)=1$, and even genus~\mbox{$\g \ge 6$}.
Then there is a stable globally generated vector bundle $E$
of rank~$2$ on~$X$ with
\begin{equation*}
c_1(E) = H_X,
\qquad
c_2(E) = \left(1 + \frac{\g}{2}\right)L_X,
\end{equation*}
and
\begin{equation}\label{eq:h-ex}
\dim H^0(X,E) = 2 + \frac{\g}{2},
\qquad
H^1(X,E) = H^2(X,E) = H^3(X,E) = 0.
\end{equation}
Moreover, if $S \subset X$ is a very general anticanonical divisor,
then the restriction~\mbox{$E_S = E\vert_S$} is stable and globally generated with
\begin{equation}\label{eq:h-es}
\dim H^0(S,E_S) = 2 + \frac{\g}{2},
\qquad
H^1(S,E_S) = H^2(S,E_S) = 0.
\end{equation}
\end{theorem}
\begin{proof}
Let $S \subset X$ be a very general hyperplane section of $X$ in the anticanonical embedding.
Then $S$ is smooth and $\Pic(S)$ is generated by the restriction $H_S$ of $H_X$ to $S$
by Noether--Lefschetz theorem (see \cite[Theorem~3.33]{Voisin2007b}).
In~\cite[Theorem~3]{mukai1989biregular} a stable globally generated vector bundle $E_S$ of rank~2 with $c_1(E_S) = H_S$ and $c_2(E_S) = 1 + \g/2$ is constructed
such that~\eqref{eq:h-es} holds, and at the end of \cite[\S2]{mukai1989biregular}
it is explained that it extends
to a vector bundle $E$ on $X$. In~\cite{Mukai-1992} it is shown
that $E$ is globally generated and~\eqref{eq:h-ex} holds. Stability of $E$ easily follows
from the stability of $E_S$ (a destabilizing subsheaf in $E$ would restrict to a destabilizing subsheaf of $E_S$).
\end{proof}

We call a bundle with the properties described in Theorem~\xref{theorem:mukai} a \emph{Mukai bundle}. It induces a map into a Grassmannian
\begin{equation*}
X \to \Gr\left(2,\frac{\g}{2} + 2\right)
\end{equation*}
such that $E$ is isomorphic to the pullback of the dual tautological bundle.
In~\cite{Mukai-1992} Mukai shows
that
\begin{itemize}
\item for $\g = 6$ the map $X \to \Gr(2,5)$ is either a closed embedding and the image is a complete intersection of two hyperplanes and a quadric,
or a double cover onto a linear section of codimension 3 branched in an intersection with a quadric
(see~\cite[Theorem~2.16]{Debarre-Kuznetsov2015} for an alternative proof);
\item for $\g = 8$ the map $X \to \Gr(2,6)$ is a closed embedding whose image is a transverse linear section of codimension 5;
\item for $\g = 10$ the map $X \to \Gr(2,7)$ is a closed embedding into the homogeneous space $\mathrm{G}_2/P \subset \Gr(2,7)$
of the group $\mathrm{G}_2$ (i.e., the simple algebraic group with Dynkin diagram of type $G_2$)
by a maximal parabolic subgroup $P \subset \mathrm{G}_2$, and the image is a transverse linear section of $\mathrm{G}_2/P$
of codimension 2.
\end{itemize}
For $g = 12$ one can show that the map $X \to \Gr(2,8)$ is a closed embedding, but a description of the image is not known.
For us in this case, only existence and properties of this vector are essential.

\begin{remark}\label{remark:factorization}
In fact, Mukai proves that for any factorization $\g = r \cdot s$ of the genus there is a nice stable vector bundle of rank $r$ on $X$.
In this way he also constructs an embedding of the genus 9 threefold into the symplectic
Lagrangian Grassmannian $\LGr(3,6)$, and with an additional trick an embedding of the genus 7 threefold into the orthogonal Lagrangian Grassmannian $\mathrm{OGr}(5,10)$.
For us it is important that the factorization~\mbox{$12 = 3 \cdot 4$} allows to construct a pair of vector bundles of rank 3 and 4 on a threefold of genus 12.
They correspond to the embedding of such threefold into $\Gr(3,7)$ studied by Mukai.
\end{remark}

We restrict to the case of even genus and factorization $\g = 2 \cdot (\g/2)$ as described
in Theorem~\xref{theorem:mukai}.
We prove some additional properties of
Mukai bundles in this case.
First, we show that a Mukai bundle is unique.

\begin{proposition}\label{ux-unique}
Let $X$ be a Fano threefold with $\rho(X)=1$, $\iota(X)=1$, and even genus~\mbox{$\g(X) \ge 6$}.
Let $E_1$, $E_2$ be two globally generated stable vector bundles on $X$ of rank~$2$ with~\mbox{$c_1 = H_X$} and~\mbox{$c_2 = \left(1 + {\g(X)}/{2}\right)L_X$}.
Then $E_1 \cong E_2$ and
\begin{equation}\label{eq:ext1-e}
\Ext^1(E_i,E_j(-1)) = 0.
\end{equation}
\end{proposition}
\begin{proof}
Let $S \subset X$ be a very general hyperplane section of $X$ in the anticanonical
embedding.
Then $S$ is a smooth K3 surface with $\Pic(S) = \Z\cdot H_X\vert_S$ by Noether--Lefschetz theorem (see \cite[Theorem~3.33]{Voisin2007b}).
By Riemann--Roch theorem one has~\mbox{$\chi(E_1\vert_S,E_2\vert_S) = 2$}.
It follows that either there is a nontrivial map $E_1\vert_S \to E_2\vert_S$, or by Serre duality a map in the opposite direction.
But by Maruyama Theorem~\cite{maruyama1981boundedness} the restrictions $E_1\vert_S$ and $E_2\vert_S$ are stable for general~$S$,
hence such a map has to be an isomorphism. It follows that in both cases we have $E_1\vert_S \cong E_2\vert_S$ for a very general~$S$.
Now applying~\mbox{$\Hom(E_1,-)$} to the exact sequence
\begin{equation*}
0 \to E_2(-H_X) \to E_2 \to E_2\vert_S \to 0
\end{equation*}
we deduce that either $\Hom(E_1,E_2) \ne 0$, and then
$E_1 \cong E_2$ by stability, or
\begin{equation*}
\Ext^1(E_1,E_2(-H_X)) \ne 0.
\end{equation*}
So, it remains to check that the latter is impossible (this will also prove~\eqref{eq:ext1-e}).

Assume on the contrary that there is a nontrivial extension
\begin{equation}\label{eq:e-f-e}
0 \to E_2(-H_X) \to F \to E_1 \to 0.
\end{equation}
Let us show that $F$ is semistable. Since
\begin{equation*}
c_1(F) = c_1(E_1) + c_1(E_2) - 2H_X = 0,
\end{equation*}
by Hoppe's criterion (\cite{Hoppe-1984})
it is enough to check that
\begin{equation*}
\Hom(\CO(H_X),F) = 0, \quad \Hom(\CO(H_X),\Lambda^2F) = 0, \quad \Hom(\CO(H_X),\Lambda^3F) = 0.
\end{equation*}
Since $\Lambda^3F \cong F^\vee$, the first and the last of these vanishings are
clear by~\eqref{eq:e-f-e} and stability of~$E_1$ and~$E_2$. For the
second vanishing note that $\Lambda^2F$ has a three step filtration with factors being
\begin{equation*}
\Lambda^2 (E_2(-H_X)) \cong \CO(-H_X),\quad
E_2(-H_X)\otimes E_1 \cong E_2^\vee\otimes E_1,
\quad
\Lambda^2E_1 \cong \CO(H_X).
\end{equation*}
Again, by stability
there are no maps from $\CO(H_X)$ to the first two factors, hence any map
\begin{equation*}
\CO(H_X) \to \Lambda^2F
\end{equation*}
splits off the last factor $\Lambda^2E_1 \cong \CO(H_X)$.
Then the composition
\begin{equation*}
F^\vee(H_X) \to F^\vee\otimes\Lambda^2F \to F
\end{equation*}
of the embedding $\CO(H_X) \to \Lambda^2F$ tensored with $F^\vee$ and the canonical contraction morphism
gives a morphism $F^\vee(H_X) \to F$ such that the composition
\begin{equation*}
E_1^\vee(H_X) \to F^\vee(H_X) \to F \to E_1
\end{equation*}
is an isomorphism. So, it splits off $E_1$ in the extension~\eqref{eq:e-f-e}.
This proves that $F$ is semistable.

On the other hand,
the discriminant $\Delta(F)$ of $F$ is
\begin{equation*}
\Delta(F) = 8c_2(F) = 8\left(2\left(1+\frac{\g(X)}{2}\right)-(2\g(X)-2)\right)L_X = -8(\g(X)-4)L_X,
\end{equation*}
so semistability of $F$ contradicts Bogomolov's inequality~\cite[Theorem~3.4.1]{huybrechts2010geometry}.
\end{proof}

\begin{remark}
For $\g(X) = 4$ there may be two non-isomorphic Mukai bundles. Indeed a Fano threefold $X$ of genus $\g = 4$ is a complete intersection of a quadric
and a cubic in~$\P^5$, see Table~\xref{table:Fanos-i-1}.
If the quadric is smooth then it is isomorphic to $\Gr(2,4)$ and thus carries
two tautological subbundles. Their restrictions to~$X$ give two non-isomorphic bundles of the type
discussed in Proposition~\xref{ux-unique}. Note that the vector bundle $F$ defined as an extension~\eqref{eq:e-f-e} is trivial in this case.
\end{remark}

Another fact that is useful for the discussion of derived categories is acyclicity of the Mukai bundle.

\begin{lemma}\label{lemma:e-exc}
Let $X$ be a Fano threefold with $\rho(X)=1$, $\iota(X)=1$, and even genus~\mbox{$\g(X)\ge 6$}.
If $E$ is the Mukai bundle on $X$ then $H^\udot(X,E(-H_X)) = 0$ and~\mbox{$\Ext^\udot(E,E) = \Bbbk$}.
\end{lemma}
\begin{proof}
Let $S$ be a very general hyperplane section of $X$
in the anticanonical
embedding. We have an exact sequence
\begin{equation*}
0 \to E(-H_X) \to E \to E_S \to 0.
\end{equation*}
Since the bundle $E$ is stable by Theorem~\xref{theorem:mukai} and $c_1(E(-H_X)) = -H_X$, we have
\begin{equation*}
H^0(X,E_X(-H_X)) = 0.
\end{equation*}
Therefore the restriction map $H^0(X,E) \to H^0(S,E_S)$ is injective. But by~\eqref{eq:h-ex} and~\eqref{eq:h-es}
the dimensions of $H^0(X,E)$ and $H^0(S,E_S)$ are equal, hence the latter map is an isomorphism.
Since $H^{>0}(X,E) = H^{>0}(S,E_S) = 0$ (again by~\eqref{eq:h-ex} and~\eqref{eq:h-es}),
we conclude that
\begin{equation*}
H^\udot(X,E_X(-H_X)) = 0.
\end{equation*}

For the second assertion, note that by Serre duality $\Ext^3(E,E)$ is dual to~\mbox{$\Hom(E,E(-H_X))$} which is zero by stability of $E$.
Similarly, $\Ext^2(E,E)$ is dual to~\mbox{$\Ext^1(E,E(-H_X))$} which is zero by~\eqref{eq:ext1-e}. Furthermore,
one has $\dim\Hom(E,E) = 1$
by stability of $E$. So, it remains to note that $\chi(E,E) = 1$ by Riemann--Roch theorem,
hence~\mbox{$\Ext^1(E,E) = 0$}.
\end{proof}

\subsection{A correspondence between Fano threefolds of index 1 and 2}

Let $X$ be a Fano threefold with $\rho(X) = 1$, $\iota(X) = 1$,
and even genus $\g(X) \ge 6$.
We consider the Mukai bundle $E$ of rank 2 on $X$, and from now on denote its dual by
\begin{equation*}
\CU_X = E^\vee.
\end{equation*}
It follows from Lemma~\xref{lemma:e-exc} (see also~\cite{kuznetsov2009derived}) that the pair of vector bundles $(\CO_X,\CU^\vee_X)$ is exceptional
and gives a semiorthogonal decomposition of the derived category of coherent sheaves
\begin{equation*}
\D^b(X) = \langle \CA_X, \CO_X, \CU_X^\vee \rangle
\end{equation*}
with the subcategory $\CA_X$ defined by
\begin{equation}\label{def-ax}
\CA_X = \langle \CO_X, \CU_X^\vee \rangle^\perp = \{ F \in \D^b(X)\ \mid \ H^\udot(X,F) = H^\udot(X,F\otimes\CU_X) = 0 \}.
\end{equation}

On the other hand, if $Y$ is a Fano threefold with $\rho(Y) = 1$, $\iota(Y) = 2$, and arbitrary
degree~\mbox{$\dd(Y)$}, there is a semiorthogonal decomposition
\begin{equation*}
\D^b(Y) = \langle \CB_Y, \CO_Y, \CO_Y(1) \rangle,
\end{equation*}
with the subcategory $\CB_Y$ defined by
\begin{equation}\label{def-by}
\CB_Y = \langle \CO_Y, \CO_Y(1) \rangle^\perp = \{ F \in \D^b(Y)\ \mid \ H^\udot(Y,F) = H^\udot(Y,F(-1)) = 0 \}.
\end{equation}

In the next lemma we show that the subcategories $\CA_X$ and $\CB_Y$ are preserved by all automorphisms of $X$ and $Y$.

\begin{lemma}\label{aut-aut}
The vector bundles $\CO_X$ and $\CU_X$ on a Fano threefold $X$ with $\rho(X) = 1$, $\iota(X) = 1$,
and even genus $\g(X)\ge 6$ are $\Aut(X)$-invariant.
In particular, the action of the group~\mbox{$\Aut(X)$} on $\D^b(X)$ preserves the subcategory $\CA_X$, so there is a morphism
\begin{equation*}
\Aut(X) \to \Aut(\CA_X)
\end{equation*}
to the group of autoequivalences of $\CA_X$. Similarly, the line bundles $\CO_Y$ and $\CO_Y(1)$ on a Fano threefold
with $\rho(Y) = 1$ and $\iota(Y) = 2$
are $\Aut(Y)$-invariant, so $\Aut(Y)$ acts on~$\CB_Y$ and there is a morphism
\begin{equation*}
\Aut(Y) \to \Aut(\CB_Y).
\end{equation*}
In both cases the image is contained in the subgroup of autoequivalences acting trivially on the numerical Grothendieck group.
\end{lemma}
\begin{proof}
The invariance of $\CO_X$, $\CO_Y$ and $\CO_Y(1)$ under automorphisms is clear, and invariance of $\CU_X$
follows from Proposition~\xref{ux-unique}. The categories $\CA_X$ and $\CB_Y$ are preserved by automorphisms
of $X$ and $Y$ by~\eqref{def-ax} and~\eqref{def-by}, hence the required morphisms. Finally, the automorphisms
of Fano threefolds of Picard rank 1 act trivially on their Chow groups,
hence by~\cite{kuznetsov2009derived} on the numerical
Grothendieck groups.
Therefore, the numerical classes of objects in $\CA_X$ and $\CB_Y$ are preserved
by automorphisms of~$X$ and~$Y$, respectively.
\end{proof}

We will use the following result

\begin{theorem}[\cite{kuznetsov2009derived}]
\label{theorem-ax-by}
For each smooth Fano threefold $X$ with $\rho(X) = 1$, $\iota(X) = 1$, and $\g(X) \in \{8,10,12\}$
there is a smooth Fano threefold $Y$ with $\rho(Y) = 1$, $\iota(Y) = 2$ and
\begin{equation*}
\dd(Y) = \frac{\g(X)}{2} - 1,
\end{equation*}
and an equivalence of categories $\CA_X \cong \CB_Y$.
\end{theorem}

In the rest of the section we give a proof of Theorem~\xref{theorem:S-vs-Sigma}, by considering consecutively all three values of $\g(X)$
and using the above equivalence of categories (explicitly in the first two cases, and implicitly in the third).
The proof consists of Propositions~\xref{proposition-v22-v5},
\xref{proposition-v18-v4}, and~\xref{proposition-v14-v3}
which will be established in the next subsections.
In the course of proof we will remind the construction of the threefold~$Y$ associated to a threefold $X$.

\subsection{Lines, conics, and derived categories}

In this subsection we show that the Hilbert scheme $S(X)$ of conics on $X$ can be thought of as
a moduli space of objects in the category $\CA_X$, defined by~\eqref{def-ax},
and the Hilbert scheme $\Sigma(Y)$ of lines on $Y$ can be thought of as a moduli space
of objects in the category $\CB_Y$, defined by~\eqref{def-by}.

We start with lines on a threefold $Y$ of index 2.

\begin{lemma}
For any line $L \subset Y$ on a Fano threefold $Y$ with $\rho(Y) = 1$ and $\iota(Y) = 2$
the ideal sheaf $I_L$ is an object of the category $\CB_Y$ defined by~\eqref{def-by}.
\end{lemma}
\begin{proof}
We have to check that
\begin{equation*}
H^\udot(Y,I_L) = H^\udot(Y,I_L(-1)) = 0.
\end{equation*}
The first follows
immediately from the exact sequence
\begin{equation}\label{eq:IL-OY-OL}
0 \to I_L \to \CO_Y \to \CO_L \to 0.
\end{equation}
For the second we twist the sequence~\eqref{eq:IL-OY-OL} by $\CO_Y(-1)$
and note that $H^\udot(Y,\CO_Y(-1)) = 0$ by Kodaira vanishing, and
$H^\udot(L,\CO_L(-1)) = 0$ since~\mbox{$L \cong \PP^1$}.
\end{proof}

An analogous statement for conics on $X$ is a bit more complicated.

\begin{lemma}\label{icinax}
For any conic $C \subset X$ on a Fano threefold $X$ with $\rho(X) = 1$, $\iota(X) = 1$,
and even genus $\g(X) \ge 6$ one has
\begin{equation}\label{eq:huxc}
H^\udot(C,\CU_X\vert_C) = 0.
\end{equation}
As a consequence,
the ideal sheaf $I_C$ is an object of the category $\CA_X$ defined by~\eqref{def-ax}.
\end{lemma}
\begin{proof}
First, let us show~\eqref{eq:huxc}.
Since $C$ is one-dimensional, we could only have
non-vanishing cohomology groups~$H^0$ and~$H^1$. On the other hand,
the Hilbert polynomial computation shows that
\begin{equation}\label{eq:after-Hilb-poly-computation}
\dim H^0(C,\CU_X\vert_C) = \dim H^1(C,\CU_X\vert_C).
\end{equation}
For this computation it is enough to assume that $C \cong \PP^1$ is smooth;
in that case $\CU_X\vert_C$ is a rank~2 vector bundle of degree $-H_X\cdot C = -2$ on $\P^1$, hence its Euler characteristic is zero.
By~\eqref{eq:after-Hilb-poly-computation} it is enough to check that $H^0(X,\CU_X\vert_C) = 0$.

Put $W = H^0(X,\CU_X^\vee)^\vee$ and let $X \to \Gr(2,W)$ be the map given by $\CU_X$. The pullback to $C$
of the tautological sequence on the Grassmannian
\begin{equation*}
0 \to \CU_X\vert_C \to W\otimes\CO_C \to (W/\CU_X)\vert_C \to 0
\end{equation*}
shows that
\begin{equation*}
H^0(C,\CU_X\vert_C) = \Ker\big(H^0(C,W\otimes\CO_C) = W \to H^0(C,(W/\CU_X)\vert_C\big).
\end{equation*}
Therefore $H^0(C,\CU_X\vert_C) \ne 0$ would imply that $C$ is contained in the zero locus of some~\mbox{$w \in W$} considered as a global section of the quotient bundle $W/\CU_X$.
The zero locus of this global section on the Grassmannian $\Gr(2,W)$ is nothing but the linearly embedded projective space
\begin{equation*}
\PP(W/w) \subset \Gr(2,W).
\end{equation*}
On the other hand, by~\cite{mukai1989biregular} the map
$X \to \Gr(2,W)$ is an embedding and its image is a
linear section of the Grassmannian (which is not dimensionally transverse),
i.\,e. there is a vector subspace $V \subset \Lambda^2W$
such that
\begin{equation*}
X = \Gr(2,W) \cap \PP(V) \subset \PP(\Lambda^2W).
\end{equation*}
Thus the zero locus of $w$ on $X$ is the intersection
\begin{equation*}
\PP(W/w) \cap \PP(V) \subset \PP(\Lambda^2W),
\end{equation*}
so it is a projective space itself. In particular, if it contains a conic
then it also contains its linear hull $\PP^2$. But $X$ cannot contain a plane by Lefschetz theorem.
This contradiction shows that actually $H^0(X,\CU_X\vert_C) = 0$ as it was claimed above, and thus proves~\eqref{eq:huxc}.

It remains to check that
\begin{equation*}
H^\udot(X,I_C) = H^\udot(X,I_C\otimes\CU_X) = 0.
\end{equation*}
The first follows
from the exact sequence
\begin{equation}\label{icseq}
0 \to I_C \to \CO_X \to \CO_C \to 0
\end{equation}
analogously to the case of lines.
For the second we tensor the sequence~\eqref{icseq} by $\CU_X$ to obtain
\begin{equation*}
0 \to I_C \otimes \CU_X \to \CU_X \to \CO_C \otimes \CU_X \to 0.
\end{equation*}
By Lemma~\xref{lemma:e-exc} we have $H^\udot(X,\CU_X) = 0$ and by~\eqref{eq:huxc}
\begin{equation*}
H^\udot(X,\CO_C \otimes \CU_X) = H^\udot(C,\CU_X\vert_C) = 0.
\end{equation*}
This completes the proof.
\end{proof}

\begin{remark}\label{conics-other-cases}
The same argument applied to Fano threefolds
of genus 9 (respectively,~7) and the natural vector bundle $\CU_X$
of rank~3 (respectively, 5) shows that
there is
a canonical morphism $\CU_X^\perp \to I_C$, where $I_C$ is the ideal sheaf of the conic
$C$, and its kernel is in~$\CA_X$. So, in these cases one should consider these kernels instead of $I_C$ and identify them
as objects of $\D^b(\Gamma)$, where $\Gamma$ is the associated curve of genus 3
(respectively,~7), see~\mbox{\cite[\S6.2 and~\S6.3]{kuznetsov2006hyperplane}} for details.
\end{remark}

The approach outlined in the Remark~\xref{conics-other-cases} was used in~\cite{Kuznetsov-V12} and~\cite{BrambillaFaenzi} to describe the Hilbert scheme of conics on Fano threefolds of index 1 and genus 7 and 9.
In the first case it was shown that $S(X) \cong \Sym^2(\Gamma)$,
where $\Gamma$ is the associated curve of genus~7, and in the second that $S(X) \cong \P_\Gamma(\CV)$,
where $\CV$ is a rank 2 vector bundle on the associated curve of genus 3 (see also Proposition~\xref{proposition:conics}).
Below we show that the vector bundle $\CV$ is simple;
this was claimed in Proposition~\xref{proposition:conics} and used in Corollary~\xref{corollary:high-genus-Fanos-Aut}
for the proof of finiteness of $\Aut(X)$.

\begin{lemma}\label{lemma:v-simple}
Let $X$ be a Fano threefold with $\rho(X) = 1$, $\iota(X) = 1$, and $\g(X) = 9$.
Let~$\Gamma$ be the curve of genus~$3$ and $\CV$ a rank $2$ vector bundle on $\Gamma$,
such that $S(X) \cong \P_\Gamma(\CV)$. Then $\CV$ is simple.
\end{lemma}
\begin{proof}
By~\cite[Proposition~3.10]{BrambillaFaenzi} we have
\begin{equation*}
\CV^\vee \cong \Phi^*(\CU_X^\vee),
\end{equation*}
where $\CU_X$ is the restriction of the tautological bundle from $\LGr(3,6)$ to $X$ (see also Remark~\xref{remark:factorization}), $\Phi \colon\D^b(\Gamma) \to \D^b(X)$ is the fully faithful functor
constructed in~\cite{kuznetsov2006hyperplane},
and $\Phi^*$ is its left adjoint functor
(see~\cite{BrambillaFaenzi} for details). Thus we have
\begin{equation*}
\Hom(\CV,\CV) \cong \Hom(\CV^\vee,\CV^\vee) \cong \Hom(\Phi^*(\CU_X^\vee),\Phi^*(\CU_X^\vee)) \cong \Hom(\CU_X^\vee,\Phi(\Phi^*(\CU_X^\vee))).
\end{equation*}
On the other hand, by~\cite[(3.14)]{BrambillaFaenzi} there is a distinguished triangle
\begin{equation*}
\CU_X(1)[-2] \to \CU_X^\vee \to \Phi(\Phi^*(\CU_X^\vee))\to \CU_X(1)[-1] .
\end{equation*}
Applying the functor $\Hom(\CU_X^\vee,-)$ to it we get an exact sequence
\begin{equation*}
\Ext^{-2}(\CU_X^\vee,\CU_X(1)) \to \Hom(\CU_X^\vee,\CU_X^\vee) \to \Hom(\CU_X^\vee,\Phi(\Phi^*(\CU_X^\vee))) \to \Ext^{-1}(\CU_X^\vee,\CU_X(1)).
\end{equation*}
Since both $\CU_X^\vee$ and $\CU_X(1)$ are pure sheaves, the $\Ext$ groups on the left and the right are zero, hence finally we have isomorphisms
\begin{equation*}
\Hom(\CV,\CV) \cong \Hom(\CU_X^\vee,\Phi(\Phi^*(\CU_X^\vee))) \cong \Hom(\CU_X^\vee,\CU_X^\vee).
\end{equation*}
It remains to notice that $\Hom(\CU_X^\vee,\CU_X^\vee)\cong\Bbbk$,
since the sheaf $\CU_X^\vee$ is simple (it is even exceptional), see Remark~\xref{remark:factorization}.
\end{proof}

\subsection{Conics on a Fano threefold of index 1 and genus 12}

Let $X$ be any smooth Fano threefold with $\rho(X) = 1$, $\iota(X) = 1$, and $\g(X) = 12$,
and let $Y$ be the smooth Fano threefold with $\rho(Y) = 1$, $\iota(Y) = 2$, and $\dd(Y) = 5$.
In this subsection we will show that~\mbox{$S(X) \cong \Sigma(Y)$}. In fact, this result is well known
(see \cite[Theorem~2.4]{Kollar2004b}, \cite[Proposition~III.1.6]{Iskovskikh-1980-Anticanonical}, \cite{Furushima1989a}),
but we will reprove it from the perspective of derived categories.

\begin{proposition}
\label{proposition-v22-v5}
There are isomorphisms $S(X) \cong \PP^2 \cong \Sigma(Y)$.
\end{proposition}
\begin{proof}
Let us first consider the threefold $Y$. Recall that
$Y$ is a linear section of the Grassmannian $\Gr(2,5) \cong \Gr(3,5)$,
see Table~\xref{table:Fanos-i-ge-2}.
Let $\CU_{2Y}$ and $\CU_{3Y}$ be the corresponding tautological bundles of rank 2 and~3 respectively.
By~\cite{orlov1991exceptional} there is a semiorthogonal decomposition
\begin{equation*}
\D^b(Y) = \langle \CU_{Y2},\CU_{Y3}, \CO_Y, \CO_Y(1) \rangle.
\end{equation*}
Moreover, $\Hom(\CU_{Y2},\CU_{Y3})$ is a three-dimensional vector space, and so
\begin{equation*}
\CB_Y = \langle \CU_{Y2},\CU_{Y3} \rangle \cong \D^b(\SQ_3).
\end{equation*}
where $\SQ_3$ is the Kronecker quiver
$$
\xymatrix{
\bullet\ar@{->}[r]
\ar@<1ex>@{->}[r]\ar@<-1ex>@{->}[r]
& \bullet
}
$$
with 3 arrows.
As it was explained in~\cite{kuznetsov2012instanton},
this equivalence gives an isomorphism
\begin{equation*}
\Sigma(Y) \cong \PP\big(\Hom(\CU_{Y2},\CU_{Y3})\big) \cong \PP^2.
\end{equation*}
Indeed, the ideal sheaf of every line can be written as the cokernel of a unique map~\mbox{$\CU_{Y2} \to \CU_{Y3}$},
and each such map has the ideal sheaf of a line as the cokernel.

Now consider a threefold $X$ of Picard rank $1$, index $1$, and genus $12$. Besides the vector bundle $\CU_X$ of rank 2,
there are stable vector bundles~$\CU_{X3}$ and~$\CU_{X4}$ on $X$ of ranks~3 and~4 respectively, see Remark~\ref{remark:factorization}.
These bundles are also exceptional and in~\cite{kuznetsov1996exceptional} it was proved that together with the rank 2 bundle $\CU_X^\vee$ they form a semiorthogonal decomposition
\begin{equation*}
\D^b(X) = \langle \CU_{X3},\CU_{X4},\CO_X,\CU_X^\vee \rangle.
\end{equation*}
Moreover, $\Hom(\CU_{X3},\CU_{X4})$ is again a three-dimensional vector space, and so
\begin{equation}\label{ax-genus12}
\CA_X = \langle \CU_{X3},\CU_{X4} \rangle \cong \D^b(\SQ_3).
\end{equation}
Now one can use the same arguments as in the case of $Y$
to show that $S(X) \cong \PP^2$. For completeness we sketch the arguments here.

First, the argument of Remark~\xref{conics-other-cases} shows that $\Hom^\udot(\CU_{X4},I_C) = \Bbbk$,
hence the decomposition of the ideal sheaf $I_C$ with respect to the exceptional pair
$(\CU_{X3},\CU_{X4})$ in $\CA_X$ takes form
of a short exact sequence
\begin{equation*}
0 \to \CU_{X3} \to \CU_{X4} \to I_C \to 0.
\end{equation*}
Conversely, by stability of $\CU_{X3}$ and $\CU_{X4}$ any morphism $\CU_{X3} \to \CU_{X4}$ is injective
and its cokernel is an ideal sheaf of a conic. Indeed, if $F$ denotes the cokernel then the dual sequence
\begin{equation}\label{eq:dual-sequence}
0 \to \CHom(F,\CO_X) \to \CU_{X4}^\vee \to \CU_{X3}^\vee \to \CExt^1(F,\CO_X) \to 0
\end{equation}
shows that $\CHom(F,\CO_X)$ is a rank 1 reflexive sheaf with $c_1 = 0$, hence is a line bundle isomorphic to $\CO_X$.
Thus $F' = \CExt^1(F,\CO_X)$ is a torsion sheaf with $c_1(F') = 0$. Dualizing
the sequence~\eqref{eq:dual-sequence} again one finds
an exact sequence
\begin{equation*}
0 \to \CU_{X3} \to \CU_{X4} \to \CO_X \to \CExt^2(F',\CO_X) \to 0.
\end{equation*}
The last sheaf thus is the structure sheaf of a subscheme $Z \subset X$, and the Hilbert polynomial
computation shows that
$p_Z(t) = 1 + 2t$,
hence $Z$ is a conic. Altogether, we deduce that
the equivalence~\eqref{ax-genus12} induces an isomorphism
\begin{equation*}
S(X) \cong \PP\big(\Hom(\CU_{X3},\CU_{X4})\big) \cong \PP^2.
\end{equation*}
The combination of the obtained isomorphisms proves the claim.
\end{proof}

The composition of equivalences $\CA_X \cong \D^b(\SQ_3) \cong \CB_Y$ mentioned in the proof takes the
bundles~$\CU_{X3}$ and~$\CU_{X4}$ to the bundles $\CU_{Y2}$ and $\CU_{Y3}$ respectively.
Therefore, it takes ideal sheaves of conics on $X$ to ideal sheaves of lines on $Y$. Thus the isomorphism $S(X) \cong \Sigma(Y)$ constructed in the proof is carried out by an equivalence $\CA_X \cong \CB_Y$.

\subsection{Conics on a Fano threefold of index 1 and genus 10}

In this subsection we prove the following

\begin{proposition}\label{proposition-v18-v4}
For every smooth Fano threefold $X$ with $\rho(X) = 1$, $\iota(X) = 1$, and~\mbox{$\g(X) = 10$}
there is a Fano threefold $Y$ with $\rho(Y) = 1$, $\iota(Y) = 2$, and $\dd(Y) = 4$ such
that~\mbox{$S(X) \cong \Sigma(Y)$}.
\end{proposition}

The proof of Proposition~\xref{proposition-v18-v4}
takes the rest of the subsection. We explain the construction of $Y$ from $X$ in the course of proof.

Recall that $X$
is a codimension~2 linear section
of a homogeneous space of the simple algebraic group $\mathrm{G}_2$,
see Table~\xref{table:Fanos-i-1}. The pencil of hyperplanes passing through~$X$
contains 6 singular elements (because the projectively dual variety is a sextic hypersurface),
so one can consider the double cover $Z \to \PP^1$ branched in the corresponding 6 points.
Thus, $Z$ is a smooth curve of genus~2.
We will show that~\mbox{$S(X) \cong \Pic^0(Z)$}.

It was proved in~\cite[\S6.4 and~\S8]{kuznetsov2006hyperplane} that there is a semiorthogonal decomposition
\begin{equation*}
\D^b(X) = \langle \D^b(Z),\CO_X,\CU_X^\vee \rangle.
\end{equation*}
In other words, we have $\CA_X \cong \D^b(Z)$.
Moreover, an explicit fully faithful Fourier--Mukai functor
\begin{equation*}
\Phi = \Phi_\CE\colon\D^b(Z) \to \D^b(X)
\end{equation*}
giving this equivalence was constructed. Its kernel $\CE$ was shown
to be a vector bundle on~\mbox{$X\times Z$} fitting into an exact sequence
\begin{equation*}
0 \to \CE \to \CO_X \boxtimes \CF_6 \to \CU_X^\vee \boxtimes \CF_3 \to \CE(H_X + H_Z) \to 0
\end{equation*}
for certain vector bundles $\CF_3$ and $\CF_6$ of ranks 3 and 6 on $Z$;
here $H_X$ is as usual the ample generator of $\Pic(X)$
and $H_Z$ is the canonical class of $Z$.
In particular, for each point~\mbox{$z \in Z$} there is an exact sequence
\begin{equation}\label{ezseq}
0 \to \CE_z \to \CO_X^{\oplus 6} \to {\CU_X^\vee}^{\oplus 3} \to \CE_z(H_X) \to 0.
\end{equation}
It follows that $r(\CE_z) = 3$ and $c_1(\CE_z) = -H_X$.

\begin{remark}\label{z-is-moduli}
In fact, one can check that all bundles $\CE_z$ are stable and that the family $\CE$ identifies
the curve $Z$ with the moduli space $\CM_X(3;-H_X,9L_X,-2P_X)$ of stable sheaves of rank 3 on $X$
with $c_1 = -H_X$,
$c_2 = 9L_X$ and $c_3 = -2P_X$.
Note also that the bundle $\CE$ is well defined only modulo a twist by a line bundle on $Z$.
We will discuss a normalization of~$\CE$ later.
\end{remark}

We proved in Lemma~\xref{icinax} that for each conic $C$ on $X$
the ideal sheaf $I_C$ is an object of the subcategory
$\CA_X = \Phi(\D^b(Z)) \subset \D^b(X)$,
hence there is an object of $\D^b(Z)$ which maps to $I_C$ under $\Phi$.
This object can be reconstructed by applying to $I_C$ the left adjoint functor~$\Phi^*$ of $\Phi$.
We compute the result in the next lemma.
For convenience we use shifts~\mbox{$I_C[-1]$} of the ideal sheaves.

\begin{lemma}\label{phis-ic}
The left adjoint functor $\Phi^*$ of $\Phi$ takes the shift $I_C[-1]$ of an ideal sheaf of a conic to a line bundle on $Z$.
\end{lemma}
\begin{proof}
Let $z \in Z$ be an arbitrary point. Then by adjunction one has
\begin{multline*}
\Hom^\udot(\Phi^*(I_C[-1]),\CO_z) =
\Hom^\udot(I_C[-1],\Phi(\CO_z)) = \\ =
\Hom^\udot(I_C[-1],\CE_z) =
\Hom^\udot(\CO_C[-2],\CE_z) =
H^\udot(C,\CE_z\vert_C).
\end{multline*}
The third equality above follows from the fact that $\CE_z \in \CO_X^\perp$ and from exact sequence~\eqref{icseq},
and the fourth equality follows from the Grothendieck duality because~\mbox{$\omega_{C/X} = \CO_C$}.
Since $\CE_z$ is a vector bundle and $C$ is a curve, the latter graded vector space a priori lives only in degrees $0$ and $1$
and its Euler characteristic is
\begin{equation*}
\chi(\CE_z\vert_C)=r(\CE_z) + c_1(\CE_z)\cdot C = 3 - 2 = 1.
\end{equation*}
So, if we show that $H^1(C,\CE_z\vert_C) = 0$
it would follow that
\begin{equation*}
\Hom^\udot(\Phi^*(I_C[-1]),\CO_z) = \Bbbk
\end{equation*}
for any point $z \in Z$ and hence $\Phi^*(I_C[-1])$ is a line bundle.

For the vanishing we note that by Serre duality we have
\begin{equation*}
H^1(C,\CE_z\vert_C)^\vee = H^0(C,\CE^\vee_z(-H_X)\vert_C)
\end{equation*}
and by the dual of~\eqref{ezseq} the latter space embeds into $H^0(C,\CU_X\vert_C^{\oplus 3})$ which is zero by~\eqref{eq:huxc}.
\end{proof}

If we twist $\CE$ with the pullback of a line bundle from $Z$, the functor $\Phi$ gets composed with the functor of tensor product by this line bundle,
and the adjoint $\Phi^*$ gets composed with the functor of tensor product by the dual line bundle.
Consequently, choosing this line bundle appropriately, we can ensure that the image of the shifted ideal sheaf of a chosen conic is the trivial line bundle. So,
we choose one conic $C_0$ on $X$ and normalize the bundle $\CE$ and the functor $\Phi = \Phi_\CE$ by requiring that
\begin{equation*}
\Phi^*(I_{C_0}[-1]) = \CO_Z,
\end{equation*}
or equivalently
\begin{equation*}
\Phi(\CO_Z) = I_{C_0}[-1].
\end{equation*}

\begin{proposition}
\label{proposition:S-Pic-iso}
The normalized functor
$$\Phi\colon\D^b(Z) \to \D^b(X)$$
gives an isomorphism $\Pic^0(Z) \cong S(X)$.
\end{proposition}
\begin{proof}
By Lemma~\xref{phis-ic} we know that $\Phi^*(I_C[-1])$ is a line bundle on $Z$. By Grothendieck--Riemann--Roch theorem
the class of $\Phi^*(I_C[-1])$ in the numerical Grothendieck group is independent of $C$ and thus coincides with the class of $\Phi^*(I_{C_0}[-1]) = \CO_Z$,
hence all these line bundles have degree zero. So, we can define a map
\begin{equation*}
S(X) \to \Pic^0(Z),
\quad
C \mapsto \Phi^*(I_C[-1]).
\end{equation*}
The map is well defined for families of conics, hence is a regular morphism.
To show that it is an isomorphism
we will check that it is \'etale and surjective, and then will construct the inverse map.

To check that the map is \'etale we note that its differential at point $C$ can be written as the composition
\begin{equation*}
\Hom(I_C,\CO_C) \to \Ext^1(I_C,I_C)
\xrightarrow{\ \Phi^*\ } \Ext^1\big(\Phi^*(I_C[-1]),\Phi^*(I_C[-1])\big).
\end{equation*}
The first map here is the isomorphism of Lemma~\xref{hilb-mod-iso} (see below) and the second is an isomorphism
because $I_C \in \CA_X$ by Lemma~\ref{icinax}, and the functor $\Phi^*$ when restricted to~$\CA_X$ is quasiinverse to the equivalence
$\Phi\colon\D^b(Z) \to \CA_X$ and hence is full and faithful.

Since the map $\Phi^*\colon S(X) \to \Pic^0(Z)$
is \'etale and $S(X)$ is proper, it follows that $\Phi^*$ is surjective.
Hence for any line bundle $\CL$
of degree 0 on $Z$ there is a conic $C \subset X$ such that
$$\Phi^*(I_C[-1]) = \CL.$$
Since $\Phi^*$ on $\CA_X$ is quasiinverse to $\Phi$ it follows that $\Phi(\CL) = I_C[-1]$, hence
\begin{equation*}
\CL \mapsto \Phi(\CL)[1]
\end{equation*}
is a well-defined map $\Pic^0(Z) \to S(X)$. This map is inverse to the map considered before since
$\Phi$ and $\Phi^*$ are quasiinverse to each other.
\end{proof}

The isomorphism we used in the proof of Proposition~\xref{proposition:S-Pic-iso}
is a special case of the following general result.

\begin{lemma}\label{hilb-mod-iso}
Let $X$ be a smooth projective variety of dimension $n$ with $\Pic^0(X) = 0$ and
let $p(t)$ be an integer valued polynomial of degree at most $n - 2$.
Let $\Hilb^p(X)$ be the Hilbert scheme of subschemes in $X$ with Hilbert polynomial $p$, and
let $\CM_X(1;0,-p)$ be the moduli space of Gieseker semistable sheaves on $X$ of rank $1$ with $c_1 = 0$ and with Hilbert polynomial~\mbox{$p_{\CO_X} - p$}.
Then the canonical morphism
\begin{equation}\label{eq:Hilb-CM}
\Hilb^p(X) \to \CM_X(1;0,-p),
\quad
(Z \subset X) \mapsto I_Z,
\end{equation}
where $I_Z$ is the ideal sheaf of $Z$,
is an isomorphism.
In particular, for any subscheme~\mbox{$Z \subset X$} of codimension at least $2$ there is an isomorphism
\begin{equation}\label{eq:Hilb-CM-Hom-Ext}
\Hom(I_Z,\CO_Z) \cong \Ext^1(I_Z,I_Z).
\end{equation}
\end{lemma}
\begin{proof}
To construct the inverse morphism we take an arbitrary scheme $S$ and consider a sheaf $\CF$ on $X \times S$
which is Gieseker semistable with the prescribed Hilbert polynomial on fibers over $S$ and consider its reflexive hull $\CF^{\vee\vee}$.
By \cite[Lemma~6.13]{Kollar-Projectivity-1990} the sheaf~$\CF^{\vee\vee}$ is locally free and the canonical map $\CF \to \CF^{\vee\vee}$ is an isomorphism in codimension 1.
Therefore, one has
\begin{equation*}
\CF_s^{\vee\vee} \cong \det(\CF_s) \cong \CO_X,
\end{equation*}
for any point $s \in S$. Therefore, up to a twist by a line bundle on~$S$, we have an isomorphism
$\CF^{\vee\vee} \cong \CO_{X\times S}$, and the canonical map $\CF \to \CF^{\vee\vee}$ identifies $\CF$ with a sheaf of ideals of a subscheme in $X \times S$.
It also follows from the proof of \cite[Lemma~6.13]{Kollar-Projectivity-1990}
that this subscheme is flat over $S$, and thus defines a map $S \to \Hilb^p(X)$.
This map is clearly inverse to the map~\mbox{$Z \mapsto I_Z$}, hence the first claim.

The second claim follows from the first, just because the left and the right hand sides
of~\eqref{eq:Hilb-CM-Hom-Ext} are the tangent spaces to the Hilbert scheme
and to the moduli space of semistable sheaves, respectively,
and the required isomorphism is the differential of the isomorphism~\eqref{eq:Hilb-CM}.
\end{proof}

Now starting from $X$ (or rather from the curve $Z$) we are going to construct a threefold~$Y$ of index 2 and degree 4 such that $\Sigma(Y) \cong \Pic^0(Z)$.
This construction, inverse to the construction of Remark~\xref{remark:hyperelliptic-curve}, is well known.
Let $\lambda_0,\ldots,\lambda_5 \in \PP^1$ be the branch points of the double cover $Z \to \PP^1$.
Choose an embedding $\mathbb{A}^1 \subset \PP^1$ so that the latter six points are contained in $\mathbb{A}^1$,
and denote their coordinates in $\mathbb{A}^1$ also by $\lambda_i$.
Let~$Y$ be the intersection of two quadrics given in $\P^5$ with homogeneous coordinates $x_0,\ldots,x_5$
by equations
\begin{equation*}
x_0^2+\ldots+x_5^2 = \lambda_0 x_0^2+\ldots+\lambda_n x_5^2 = 0,
\end{equation*}
so that the curve $B(Y)$ defined in Remark~\xref{remark:hyperelliptic-curve} is isomorphic to $Z$.
By~\cite{bondal1995semiorthogonal,kuznetsov2008derived} we have a semiorthogonal decomposition
\begin{equation*}
\D^b(Y) = \langle \D^b(Z), \CO_Y, \CO_Y(1) \rangle,
\end{equation*}
i.\,e. an equivalence $\Psi\colon\D^b(Z) \to \CB_Y$. Similarly to the case of the variety $X$,
this equivalence induces an isomorphism
$\Sigma(Y) \cong \Pic^0(Z)$ (see~\cite{DesaleRamanan}
or~\cite[\S5.3]{kuznetsov2012instanton} for detailed explanation, and
\cite{Fonarev-Kuznetsov} for a generalization).
Combining the two constructed isomorphisms we deduce Proposition~\xref{proposition-v18-v4}.

\subsection{Conics on a Fano threefold of index 1 and genus 8}

Let $X$ be a smooth Fano threefold with $\rho(X) = 1$, $\iota(X) = 1$, and $\g(X) = 8$.
In this subsection we discuss the associated Fano threefold of index 2, which in this case is just a cubic threefold, and construct an isomorphism~\mbox{$S(X) \cong \Sigma(Y)$}.

\begin{proposition}\label{proposition-v14-v3}
For every smooth Fano threefold $X$ with $\rho(X) = 1$, $\iota(X) = 1$, and genus~\mbox{$\g(X) = 8$}
there is a smooth Fano threefold $Y$ with $\rho(Y) = 1$, $\iota(Y) = 2$,
and degree~\mbox{$\dd(Y) = 3$} such that
$S(X) \cong \Sigma(Y)$.
\end{proposition}

The proof of Proposition~\xref{proposition-v14-v3}
takes the rest of the subsection.
Recall that $X$ is a linear section of the Grassmannian $\Gr(2,6)$ of codimension 5,
see Table~\xref{table:Fanos-i-1}.
Let $W$ be a six-dimensional vector space
and
\begin{equation*}
A \subset \Lambda^2W^\vee
\end{equation*}
be the five-dimensional space of linear equations of
$X \subset \Gr(2,W)$. Then the associated cubic threefold $Y$ is defined as
\begin{equation}\label{eq:cubic-3fold}
Y = \PP(A) \cap \Pf(W) \subset \PP(\Lambda^2W^\vee),
\end{equation}
where $\Pf(W) \subset \PP(\Lambda^2W^\vee)$ is the Pfaffian cubic hypersurface.

In this case we construct an isomorphism $S(X) \cong \Sigma(Y)$ geometrically.
Denote by
\begin{equation*}
R \subset \Gr(2,A)\times \Gr(4,W)
\end{equation*}
the locus of pairs $(A_2,W_4)$
consisting of a two-dimensional subspace $A_2 \subset A$ and a four-dimensional subspace $W_4 \subset W$
such that the composition
\begin{equation*}
A_2 \hookrightarrow A \hookrightarrow \Lambda^2W^\vee \to \Lambda^2W_4^\vee
\end{equation*}
is the zero map.
In other words, $R \subset \Gr(2,A)\times \Gr(4,W)$ is the zero locus of the natural section of the vector
bundle~\mbox{$\CU_A^\vee \boxtimes \Lambda^2\CU_W^\vee$}, where $\CU_A$ is the tautological bundle on $\Gr(2,A)$
and $\CU_W$ is the tautological bundle on~$\Gr(4,W)$.

\begin{proposition}\label{conics-lines-v14}
There are isomorphisms $S(X) \cong R \cong \Sigma(Y)$.
\end{proposition}
\begin{proof}
Given a point $(A_2,W_4) \in R$ we associate to it a conic in $X$ as follows.
Since the space $A_2$ maps to zero in $\Lambda^2W_4^\vee$, the image of $A$
in $\Lambda^2W_4^\vee$ is at most three-dimensional, hence the intersection
\begin{equation}\label{eq:X-cap-Gr}
X \cap \Gr(2,W_4) \subset \Gr(2,W)
\end{equation}
is a linear section of $\Gr(2,W_4)$ of codimension at most $3$. Since $\Gr(2,W_4)$ is a four-dimensional quadric,
this intersection is either a conic, or a plane, or a two-dimensional quadric, or has dimension larger than~$2$.
But by Lefschetz theorem $X$ contains neither planes, nor two-dimensional quadrics,
and is not contained in $\Gr(2,W_4)$.
Hence the intersection~\eqref{eq:X-cap-Gr} is a conic. Therefore, we have a map
\begin{equation*}
s\colon R \to S(X),
\quad
(A_2,W_4) \mapsto X \cap \Gr(2,W_4).
\end{equation*}
For the inverse map consider the tautological sequence
$$0 \to \CU_X^\perp \to W^\vee\otimes\CO_X \to \CU_X^\vee \to 0$$
and restrict it to a conic $C \subset X$:
\begin{equation*}
0 \to \CU_X^\perp\vert_C \to W^\vee\otimes\CO_C \to \CU_X^\vee\vert_C \to 0.
\end{equation*}
Since $H^1(C,\CO_C) = 0$, it follows that $H^1(C,\CU_X^\vee\vert_C) = 0$,
and so by Riemann--Roch theorem one has
$\dim H^0(C,\CU_X^\vee\vert_C) = 4$. Therefore
the subspace
$$H^0(C,\CU^\perp_X\vert_C) \subset H^0(C,W^\vee\otimes\CO_C) = W^\vee$$
is at least two-dimensional. Clearly, any linear function from this space vanishes
on any two-dimensional subspace $U \subset W$ parameterized by a point of
the conic $C$.
So, if this space is at least three-dimensional then $C$ is contained in the linear section $X \cap \Gr(2,3)$
of~\mbox{$\Gr(2,3)\cong\PP^2$},
hence this linear section is $\PP^2$, which gives a contradiction since $X$ cannot contain a plane by Lefschetz theorem. This means that
$H^0(C,\CU^\perp_X\vert_C)$ is a two-dimensional subspace in~$W^\vee$ and its annihilator is a four-dimensional subspace $W_4 \subset W$.
Since a conic in a four-dimensional quadric
$\Gr(2,W_4) \subset \PP(\Lambda^2W^4)$
is a linear section of codimension 3,
it follows that at least a two-dimensional subspace of linear equations of~$X$ restricts trivially to $\Lambda^2 W_4$.
Conversely, if a three-dimensional space of equations would restrict trivially to~$\Lambda^2 W_4$,
then the intersection~\eqref{eq:X-cap-Gr} would contain a two-dimensional quadric which is again forbidden by Lefschetz theorem.
Thus, the space of equations restricting trivially to $\Lambda^2W_4$ is a two-dimensional subspace $A_2 \subset A$, the pair $(A_2,W_4)$ is a point of~$R$,
and~$C \mapsto (A_2,W_4)$ is a morphism $S(X) \to R$ inverse to the morphism $s$ above.

On the other hand, given a point $(A_2,W_4) \in R$ we can associate with it the line
\begin{equation*}
L = \PP(A_2) \subset \PP(A).
\end{equation*}
Note that by definition of $R$ each skew form in $A_2$ has a four-dimensional isotropic subspace $W_4$ and hence is degenerate. Thus $L \subset \Pf(W)$, hence $L \subset Y$,
so that the map
\begin{equation*}
\sigma\colon R \to \Sigma(Y),
\quad
(A_2,W_4) \mapsto \PP(A_2) \subset Y
\end{equation*}
is well defined. To construct the inverse map we note that by~\cite[Appendix~A]{KuMaMa}
for each line
$$L = \PP(A_2) \subset Y$$
there is a unique four-dimensional subspace $W_4 \subset W$ isotropic for all
skew forms in $L$. Thus $L \mapsto (A_2,W_4)$ is a morphism $\Sigma(Y) \to R$ which is clearly inverse to the morphism $\sigma$ above.
\end{proof}

\begin{remark}
One can also describe the isomorphism of Proposition~\xref{conics-lines-v14} via derived categories.
For this note that by~\cite{kuznetsov2004derived,kuznetsov2006homological} there is an equivalence $\CA_X \cong \CB_Y$.
Moreover, the equivalence is given by the Fourier--Mukai functor
\begin{equation*}
\Phi = \Phi_{I_Z(H_Y)}\colon \D^b(X) \to \D^b(Y)
\end{equation*}
with the kernel being the $\CO_Y(H_Y)$-twist of the ideal sheaf $I_Z$ of an irreducible
four-dimensional subvariety
\begin{equation*}
Z \subset X\times Y
\end{equation*}
of all points $(U,a)$ such that the kernel of the skew form $a \in A$ intersects the two-dimensional subspace $U \subset W$.
One can check that under this functor the ideal sheaf of a conic on $C$ goes to the ideal sheaf of the corresponding
line on $Y$. However, this verification is more complicated than the direct geometric proof given above,
so we skip~it.
\end{remark}

\newcommand{\etalchar}[1]{$^{#1}$}
\def\cprime{$'$}

\newblock

\end{document}